\documentclass{amsart}%

\usepackage{graphicx}
\usepackage{multicol,multirow}
\usepackage{amsmath,amssymb,amsfonts}
\usepackage{mathrsfs}
\usepackage{rotating}
\usepackage{appendix}
\usepackage[numbers]{natbib}
\usepackage{amsthm}
\usepackage{lipsum}

\theoremstyle{plain}
          \newtheorem{theorem}{Theorem}[section]
          \newtheorem{proposition}[theorem]{Proposition}
          \newtheorem{lemma}[theorem]{Lemma}
          \newtheorem{corollary}[theorem]{Corollary}
        \theoremstyle{definition}
          \newtheorem{definition}[theorem]{Definition}
          \newtheorem{example}[theorem]{Example}
          \newtheorem{remark}[theorem]{Remark}

\numberwithin{equation}{section}

\def\acts{\curvearrowright}

\newcommand{\iden}{1_G}

\newcommand{\supp}{\mathrm{supp}}

\newcommand{\F}{\mathcal{F}}
\newcommand{\K}{\mathcal{K}}

\newcommand{\Prob}{\mathbb{P}}

\newcommand{\dist}{\mathrm{dist}}

\newcommand{\Stab}{\mathrm{Stab}}
\newcommand{\Rep}{U_0}

\newcommand{\Hom}{\mathrm{Hom}}

\newcommand{\Aut}{\mathrm{Aut}}
\newcommand{\Cay}{\mathrm{Cay}}

\newcommand{\NP}{\mathrm{NP}}
\newcommand{\RP}{\mathrm{RP}}

\newcommand{\weight}{\mathrm{w}}

\usepackage{bbding}
\usepackage{bbm}

\usepackage{xargs}                      
\usepackage[pdftex,dvipsnames]{xcolor}  
\usepackage[colorinlistoftodos,prependcaption,textsize=tiny]{todonotes}
\newcommandx{\unsure}[2][1=]{\todo[linecolor=red,backgroundcolor=red!25,bordercolor=red,#1]{#2}}
\newcommandx{\change}[2][1=]{\todo[linecolor=blue,backgroundcolor=blue!25,bordercolor=blue,#1]{#2}}
\newcommandx{\info}[2][1=]{\todo[linecolor=OliveGreen,backgroundcolor=OliveGreen!25,bordercolor=OliveGreen,#1]{#2}}
\newcommandx{\improvement}[2][1=]{\todo[linecolor=Plum,backgroundcolor=Plum!25,bordercolor=Plum,#1]{#2}}
\newcommandx{\thiswillnotshow}[2][1=]{\todo[disable,#1]{#2}}
\usepackage[unicode=true, colorlinks = true, linkcolor=WildStrawberry, citecolor=NavyBlue, linktocpage=true, bookmarks=false]{hyperref}

\allowdisplaybreaks

\begin{document}

\title[Counting independent sets in amenable groups]{Counting independent sets in amenable groups}

\author[Raimundo Brice\~no]{Raimundo Brice\~no}
\address{Facultad de Matem\'aticas\\Pontificia Universidad Cat\'olica de Chile\\Santiago\\Chile}
\email{raimundo.briceno@mat.uc.cl}

\thanks{The author acknowledges the support of ANID/FONDECYT de Iniciación en Investigación 11200892.}



\begin{abstract}
Given a locally finite graph $\Gamma$, an amenable subgroup $G$ of graph automorphisms acting freely and almost transitively on its vertices, and a $G$-invariant activity function $\lambda$, consider the free energy $f_G(\Gamma,\lambda)$ of the hardcore model defined on the set of independent sets in $\Gamma$ weighted by $\lambda$.

Under the assumption that $G$ is finitely generated and its word problem can be solved in exponential time, we define suitable ensembles of hardcore models and prove the following: if $\|\lambda\|_\infty < \lambda_c(\Delta)$, there exists a randomized $\epsilon$-additive approximation scheme for $f_G(\Gamma,\lambda)$ that runs in time $\mathrm{poly}((1+\epsilon^{-1})\lvert \Gamma/G \rvert)$, where $\lambda_c(\Delta)$ denotes the critical activity on the $\Delta$-regular tree. In addition, if $G$ has a finite index linearly ordered subgroup such that its algebraic past can be decided in exponential time, we show that the algorithm can be chosen to be deterministic. On the other hand, we observe that if $\|\lambda\|_\infty > \lambda_c(\Delta)$, there is no efficient approximation scheme, unless $\mathrm{NP} = \mathrm{RP}$. This recovers the computational phase transition for the partition function of the hardcore model on finite graphs and provides an extension to the infinite setting.

As an application in symbolic dynamics, we use these results to develop efficient approximation algorithms for the topological entropy of subshifts of finite type with enough safe symbols, we obtain a representation formula of pressure in terms of random trees of self-avoiding walks, and we provide new conditions for the uniqueness of the measure of maximal entropy based on the connective constant of a particular associated graph.
\end{abstract}

\maketitle

\setcounter{tocdepth}{1}
\tableofcontents

\section{Introduction}

Suppose that we are given a finite simple graph $\Gamma = (V,E)$ and we are asked to count its number of \emph{independent sets}. An independent set is a subset $I \subseteq V$ such that $(v,v') \notin E$ (i.e., $(v,v')$ is not an edge) for all $v,v' \in I$. For example, if $\Gamma$ is the $4$-cycle $C_4$ with $V = \{v_1,v_2,v_3,v_4\}$ and $E = \{(v_1,v_2),(v_2,v_3),(v_3,v_4),(v_4,v_1)\}$, it can be checked that there are exactly $7$ different independent sets, namely $\emptyset$, $\{v_1\}$, $\{v_2\}$, $\{v_3\}$, $\{v_4\}$, $\{v_1,v_3\}$, and $\{v_2,v_4\}$. A common generalization of this question is to ask for the ``number'' of \emph{weighted} independent sets in $\Gamma$: given a parameter $\lambda > 0$---usually called \emph{activity} or \emph{fugacity}---, we ask for the value of the summation
$$
Z_\Gamma(\lambda) := \sum_{I \in X(\Gamma)} \lambda^{\lvert I \rvert},
$$
where $X(\Gamma)$ denotes the collection of all independent sets in $\Gamma$ and $\lvert I \lvert $, the cardinality of a given independent set $I$.  Notice that we recover the original problem---i.e., to compute $\lvert X(\Gamma) \rvert$---if we set $\lambda = 1$. The sum $Z_\Gamma(\lambda)$ corresponds to the normalization factor of the probability distribution $\Prob_{\Gamma,\lambda}$ on $X(\Gamma)$ that assigns to each $I \in X(\Gamma)$ a probability proportional to $\lambda^{\lvert I \rvert}$, i.e., the so-called \emph{partition function} (also known as the \emph{independence polynomial}) of the well-known \emph{hardcore model} from statistical physics.

In general, it is not possible to compute exactly $Z_\Gamma(\lambda)$ efficiently \cite{1-luby}, even for the case $\lambda = 1$ \cite{1-xia}; technically, to compute $Z_\Gamma(\lambda)$ is an \emph{$\mathrm{NP}$-hard} problem and to compute $\lvert X(\Gamma) \rvert$ is a \emph{$\#\mathrm{P}$-complete} problem. Therefore, one may attempt to at least find ways to approximate $Z_\Gamma(\lambda)$ efficiently.

In recent years, there has been a great deal of attention to the complexity of approximating partition functions of spin systems (e.g., see \cite{1-barvinok}). Among these systems, the hardcore model, possibly together with the Ising model \cite{3-sinclair}, occupies the most important place. One of the most notable results, due to Weitz \cite{1-weitz}, and then Sly \cite{1-sly} and Sly-Sun \cite{2-sly}, is the existence of a \emph{computational phase transition} for having a \emph{fully polynomial-time approximation scheme (FPTAS)} for the approximation of $Z_\Gamma(\lambda)$. In simple terms, Weitz developed an FPTAS, a particular kind of efficient deterministic approximation algorithm, on the family of finite graphs with bounded degree $\Delta$ provided $\lambda < \lambda_c(\Delta)$, where $\lambda_c(\Delta) := \frac{(\Delta-1)^{(\Delta-1)}}{(\Delta-2)^{\Delta}}$ denotes the critical activity for the hardcore model on the $\Delta$-regular tree $\mathbb{T}_\Delta$. Conversely, a couple of years later, Sly and Sun managed to prove that the existence of even a \emph{fully polynomial-time randomized approximation scheme (FPRAS)}---which is a probabilistic and therefore weaker version of an FPTAS---for $\lambda > \lambda_c(\Delta)$ would imply that $\NP = \RP$, the equivalence of two well-known computational complexity classes which are widely believed to be different \cite{1-arora}.
 
The work of Weitz exploited a technique based on trees of self-avoiding walks and a special notion of correlation decay known as \emph{strong spatial mixing} that, in particular, holds when the graph is $\mathbb{T}_\Delta$ and $\lambda < \lambda_c(\Delta)$. Later, Sinclair et al. \cite{1-sinclair} studied refinements of Weitz's result by considering families of finite graphs parameterized by their \emph{connective constant} instead of their maximum degree, and established that there exists an FPTAS for $Z_\Gamma(\lambda)$ for families of graphs with connective constant bounded by $\mu$, whenever $\lambda < \frac{\mu^\mu}{(\mu-1)^{(\mu+1)}}$.

Now, if $\Gamma$ is an infinite graph, most of these concepts stop making sense. One way to deal with this issue is by choosing an appropriate normalization and by using the \emph{DLR formalism}. The idea is roughly the following: suppose that we have a sequence $\{\Gamma_n\}_n$ of finite subgraphs that ``exhausts'' $\Gamma$ in some way. This sequence induces two other sequences: a sequence $\{Z_{\Gamma_n}(\lambda)\}_n$ of partition functions and a sequence $\{\Prob_{\Gamma_n,\lambda}\}_n$ of probability distributions. A way to extend the idea of ``number of weighted independent sets (per site)'' in $\Gamma$ is by considering the sequence $\{Z_{\Gamma_n}(\lambda)^{1/\lvert \Gamma_n \rvert}\}_n$ and hoping that it converges. Under the right assumptions on $\Gamma$ and $\{\Gamma_n\}_n$, this is exactly the case and something similar happens to $\{\Prob_{\Gamma_n,\lambda}\}_n$. Moreover, there is an intimate connection between the properties of the limit measures and our ability to estimate the value of $\lim_n \lvert \Gamma_n \rvert^{-1} \log Z_{\Gamma_n}(\lambda)$, i.e., to ``approximately count'' it. We denote this limit---which, a priori, may depend on the sequence $\{\Gamma_n\}_n$---by $f_{\{\Gamma_n\}_n}(\Gamma,\lambda)$ and call it the \emph{free energy} of the hardcore model $(\Gamma,\lambda)$, one of the most crucial quantities in statistical physics \cite{1-baxter,1-georgii,1-simon}.

It can be checked that if $\Gamma$ is finite, to approximate the partition function $Z_\Gamma(\lambda)$ with a multiplicative error (in polynomial time) is equivalent to approximate the free energy $f_{\{\Gamma\}_n}(\Gamma,\lambda)$---where $\{\Gamma\}_n$ is the constant sequence which immediately exhausts the graph---with an additive error \cite{2-sinclair} (in polynomial time). Therefore, the problem of approximating $f_{\{\Gamma\}_n}(\Gamma,\lambda)$ recovers the problem of approximating the partition function in the finite case and, at the same time, extends the problem to the infinite setting.

The main goal of this paper is to establish a computational phase transition for the free energy on ensembles of---possibly infinite---hardcore models. In other words, we aim to prove the existence of an efficient additive approximation algorithm for the free energy when the activity is low and to establish that there is no efficient approximation algorithm for the free energy when the activity is high, unless $\mathrm{NP} = \mathrm{RP}$.

There have been many recent works related to the study of correlation decay properties and its relation to approximation algorithms for the free energy (and related quantities such as pressure, capacity, and entropy) in the infinite setting \cite{1-gamarnik,1-marcus,1-briceno,1-wang,1-pavlov,3-marcus,2-marcus}. In this work we put all these results in a single framework, which also encompasses the results from Weitz, Sly and Sun, and Sinclair et al., and at the same time generalizes them.

In 2009, Gamarnik and Katz \cite{1-gamarnik} introduced what they called the \emph{sequential cavity method}, which can be regarded as a sort of infinitary \emph{self-reducibility} property \cite{1-jerrum}. Combining this method with Weitz's results, they managed to prove that the free energy of the hardcore model in the Cayley graph of $\mathbb{Z}^d$ with canonical generators admits a (deterministic) $\epsilon$-additive approximation algorithm that runs in time polynomial in $\epsilon^{-1}$ whenever $\lambda < \lambda_c(2d)$, where $2d$ is the maximum degree of the graph. Related results were also proven by Pavlov in \cite{1-pavlov}, who developed an approximation algorithm for the \emph{hard square entropy}, i.e., the free energy of the hardcore model in the Cayley graph of $\mathbb{Z}^2$ with canonical generators and activity $\lambda = 1$. Later, there were also some explorations due to Wang et al. \cite{1-wang} in Cayley graphs of $\mathbb{Z}^2$ with respect to other generators (e.g., the \emph{non-attacking kings} system) in the context of information theory and algorithms for approximating capacities.

In this paper we prove that all these results fit and can be generalized to hardcore models $(\Gamma,\lambda)$ such that (1) $\Gamma$ is a locally finite graph, (2) $G \acts \Gamma$ is free and almost transitive for some countable amenable subgroup $G \leq \Aut(\Gamma)$, and (3) $\lambda: V \to \mathbb{Q}_{>0}$ is a---not necessarily constant---$G$-invariant activity function. In addition, for the algorithmic implications we assume that $G$ satisfies some of the recursion-theoretic assumptions described below. Given this setting, we consider a F{\o}lner sequence $\{F_n\}_n$, a fundamental domain $U_0 \subseteq V$ of $G \acts \Gamma$, and the sequence of finite subgraphs $\{\Gamma_n\}_n$ induced by $\{F_nU_0\}_n$. First, we show that $f_{\{\Gamma_n\}_n}(\Gamma,\lambda)$ is independent of $\{F_n\}_n$ and $U_0$, and that the limit $f_{\{\Gamma_n\}_n}(\Gamma,\lambda)$---which we denote by $f_G(\Gamma,\lambda)$ to emphasize the independence of $\{F_n\}_n$ and $U_0$---can be expressed as an infimum over some suitable family of finite subgraphs of $\Gamma$. Next, based on results from \cite{1-gurevich,2-briceno}, we prove in Theorem \ref{thm:smb} that $f_G(\Gamma,\lambda)$ can be obtained as the pointwise limit of a Shannon-McMillan-Breiman type ratio with regards to any Gibbs measure on $X(\Gamma)$. In Theorem \ref{thm:rep1}, we prove that if $\lambda$ is such that $(\Gamma,\lambda)$ satisfies strong spatial mixing, then $f_G(\Gamma,\lambda)$ corresponds to the evaluation of a random information function, based on ideas about random invariant orders and the \emph{Kieffer-Pinsker formula} for measure-theoretical entropy introduced in \cite{1-alpeev}. Then, in Theorem \ref{thm:rep2}, using the previous representation theorem and the techniques from \cite{1-weitz}, we provide a formula for $f_G(\Gamma,\lambda)$ in terms of trees of self-avoiding walks in $\Gamma$. These first three theorems can be regarded as a preprocessing treatment of $f_G(\Gamma,\lambda)$ in order to obtain an \emph{arboreal representation of free energy} to develop approximation techniques, but we believe that they are of independent interest.

Later, we consider a finitely generated amenable group $G$ with a prescribed set of generators $S$ such that its word problem can be solved in exponential time. This last requirement seems to be natural and many groups satisfy it (for example, any linear group, including all abelian, all nilpotent groups and, more generally, all virtually polycyclic groups). Given a positive integer $\Delta$ and $\lambda_0 > 0$, we denote by $\mathcal{H}_G^\Delta(\lambda_0)$ the ensemble of hardcore models $(\Gamma,\lambda)$ such that $G \acts \Gamma$ is free and almost transitive, the maximum degree of $\Gamma$ is bounded by $\Delta$, and the values of $\lambda$ are bounded from above by $\lambda_0$. Then, in Theorem \ref{thm:rand-approx}, we establish the following algorithmic implication: if $\lambda_0 < \lambda_c(\Delta)$, there exists an \emph{additive FPRAS} on $\mathcal{H}_G^{\Delta}(\lambda_0)$ for $f_G(\Gamma,\lambda)$, where $\lambda_c(\Delta)$ denotes the critical activity on the $\Delta$-regular tree $\mathbb{T}_\Delta$. This can be considered as a confirmation in the amenable setting of the ``algorithmic version''---as called in \cite{1-weitz}---of Conjecture 2.1 in \cite{1-sokal}. In addition, under the extra assumption that $G$ has a finite index linearly ordered subgroup $(H, \prec)$ such that its algebraic past $\Phi_\prec = \{g \in H: g \prec 1_G\}$ can be decided in exponential time, we prove that the algorithm can be chosen to be deterministic, i.e., there exists an \emph{additive FPTAS}. Groups that satisfy this extra condition include all finitely generated abelian groups, nilpotent groups like the Heisenberg group $H_3(\mathbb{Z})$, and solvable groups like the Baumslag-Solitar groups $BS(1,n)$. On the other hand, in Corollary \ref{cor:phase} we observe that if $\lambda_0 > \lambda_c(\Delta)$, there is no additive FPRAS unless $\mathrm{NP} = \mathrm{RP}$. In particular, we obtain that the results from Weitz, Sly, and Sun correspond to the special case when $G$ is the trivial (and orderable) group.

By an \emph{additive FPRAS}, we mean a probabilistic algorithm that given $(\Gamma,\lambda)$ and $\epsilon > 0$, outputs a number $\hat{f}$ such that $\lvert f_G(\Gamma,\lambda) - \hat{f} \rvert < \epsilon$ with probability greater than $3/4$ in time polynomial in $\lvert \Gamma / G \rvert$ and $\epsilon^{-1}$. Here, $\lvert \Gamma / G \rvert$ denotes the size of some (or any) fundamental domain of the action $G \acts \Gamma$, and therefore, all the information we need in order to construct $\Gamma$. On the other hand, by an \emph{additive FPTAS}, we mean an additive FPRAS with success probability equal to $1$ instead of just $3/4$. We assume throughout the paper that the standard functions and arithmetic operations of the numerical values involved can be computed exactly in one unit of time.

Finally, as an application in symbolic dynamics, we show how to use these results to establish representation formulas and efficient approximation algorithms for the topological entropy of nearest-neighbor subshifts of finite type with enough safe symbols. Also, we consider the pressure of single-site potentials with a vacuum state, which includes systems such as the \emph{Widom-Rowlinson model} and some other weighted graph homomorphisms from $\Gamma$ to any finite graph, among others. These results can also be regarded as an extension of the works by Marcus and Pavlov in $\mathbb{Z}^d$ (see \cite{1-marcus,2-marcus,3-marcus}), who developed additive approximation algorithms for the entropy and free energy (or pressure) of general $\mathbb{Z}^d$-subshifts of finite type, with special emphasis in the case $d=2$. We believe that these implications are relevant, especially in the light of results like the one from Hochman and Meyerovitch. In \cite{1-hochman}, Hochman and Meyerovitch proved that the set of topological entropies that a nearest-neighbor $\mathbb{Z}^2$-subshift of finite type can achieve coincides with the set of \emph{non-negative right-recursively enumerable real numbers}. This class of numbers includes numbers that are poorly computable or even not computable. In addition, we discuss the case of the monomer-dimer model and counting independent sets of line graphs, which is a special case that does not exhibit a phase transition. As a byproduct of our results, we also give sufficient conditions for the existence of a unique measure of maximal entropy for subshifts on arbitrary amenable groups.

We remark that our results---considering related ones, like those obtained by Gamarnik and Katz in \cite{1-gamarnik}---are novel in at least three aspects:
\begin{enumerate}
\item {\bf Almost transitive framework.} The generalization to the almost transitive case provides enough flexibility so that (i) other systems (such as subshifts of finite type, matchings, etc.) can be represented through reductions in terms of independent sets in suitable graphs and (ii)  the measurement of (the size of) fundamental domains as a way to measure computational complexity provides a way to obtain a computational phase transition. These aspects---to our knowledge---are new, even in the relevant case $G = \mathbb{Z}^d$, i.e., the family of graphs such that $\mathbb{Z}^d$ acts almost transitively on them.
\item {\bf Algorithms for graphs with exponential growth.} Our approach, which pro\-vides polynomial time approximation algorithms, works for ame\-na\-ble groups not only of polynomial growth but also exponential growth. A relevant case that is fully explored in Section \ref{sec8} is the family of Baumslag-Solitar groups $BS(1,n)$ for $n \geq 2$, which have exponential growth but admit even a deterministic approximation algorithm for free energy.
\item {\bf Lack of orderability.} If a group does not have an orderable subgroup of finite index it is less clear how to obtain a sequential cavity method as in \cite{1-gamarnik}, which exploits the existence of an invariant deterministic order of the group at hand (like, for example, the lexicographic order in $\mathbb{Z}^d$). Our free energy representation formulas in terms of invariant random orders provide a way to develop randomized approximation algorithms for groups that are not necessarily orderable.
\end{enumerate}

The paper is organized as follows: in Section \ref{sec2}, we introduce the basic concepts regarding graphs, homomorphisms, independent sets, group actions, Cayley graphs, and partition functions; in Section \ref{sec3}, we rigorously define free energy based on the notion of amenability and show some robustness properties of its definition; in Section \ref{sec4}, we define Gibbs measures and relevant spatial mixing properties; in Section \ref{sec5}, we develop the formalism based on trees of self-avoiding walks and discuss some of its properties; in Section \ref{sec6}, we present the formalism of invariant (deterministic and random) orders of a group; in Section \ref{sec7}, we prove Theorem \ref{thm:smb}, Theorem \ref{thm:rep1}, and Theorem \ref{thm:rep2}, which provide a randomized sequential cavity method that allows us to obtain an arboreal representation of free energy; in Section \ref{sec8}, we prove Theorem \ref{thm:rand-approx} and establish the algorithmic implications of our results; in Section \ref{sec9}, we provide reductions that allow us to translate the problem of approximating pressure of a single-site potential and the topological entropy of a subshift into the problem of counting independent sets and discuss other consequences that are implicit in our results.

\section{Preliminaries}
\label{sec2}

\subsection{Graphs}

A {\bf graph} will be a pair $\Gamma = (V,E)$ such that $V$ is a countable set---the {\bf vertices}---and $E \subseteq V \times V$ is a symmetric relation---the {\bf edges}---. Let $\leftrightarrow$ be the equivalence relation generated by $E$, i.e., $v \leftrightarrow v'$ if and only if there exist $n \in \mathbb{N}_0$ and $\{v_i\}_{0 \leq i \leq \ell}$ such that $v = v_0$, $v' = v_n$, and $(v_i,v_{i+1}) \in E$ for every $0 \leq i < n$. Denote by $n(v,v')$ the smallest $n$ with this property. This induces a notion of distance in $\Gamma$ given by 
$$
\dist_\Gamma(v,v') = 
\begin{cases}
n(v,v')	&	\text{if } v \leftrightarrow v',	\\
+\infty	&	\text{otherwise.} 
\end{cases}
$$

Given a set $U \subseteq V$, we define its {\bf boundary} $\partial U$ as the set $\{v \in V: \dist_\Gamma(v,U) = 1\}$, where $\dist_\Gamma(U,U') = \inf_{v \in U, v' \in U'} \dist_\Gamma(v,v')$. In addition, given $\ell \geq 0$ and $v \in V$, we define the {\bf ball centered at $v$ with radius $\ell$} as $B_\Gamma(v,\ell) := \{v' \in V: \dist_\Gamma(v,v') \leq \ell\}$.

A graph $\Gamma$ is
\begin{itemize}
\item {\bf loopless}, if $E$ is anti-reflexive (i.e., there is no vertex related to itself);
\item {\bf connected}, if $v \leftrightarrow v'$ for every $v,v' \in V$; and
\item {\bf locally finite}, if every vertex is related to only finitely many vertices.
\end{itemize}

Sometimes we will write $V(\Gamma)$ and $E(\Gamma)$---instead of just $V$ and $E$---to emphasize $\Gamma$.

\subsection{Homomorphisms}

Consider graphs $\Gamma_1$ and $\Gamma_2$. A {\bf graph homomorphism} is a map $g: V(\Gamma_1) \to V(\Gamma_2)$ such that
$$
(v,v') \in E(\Gamma_1) \implies (g(v), g(v')) \in E(\Gamma_2).
$$

We denote by $\Hom(\Gamma_1,\Gamma_2)$ the set of graph homomorphisms from $\Gamma_1$ to $\Gamma_2$.

A {\bf graph isomorphism} is a bijective map $g: V(\Gamma_1) \to V(\Gamma_2)$ such that
$$
(v,v') \in E(\Gamma_1) \iff (g(v), g(v')) \in E(\Gamma_2).
$$
If a map like this exists, we say that $\Gamma_1$ and $\Gamma_2$ are \emph{isomorphic}, denoted by $\Gamma_1 \cong \Gamma_2$.

A {\bf graph automorphism} is a graph isomorphism from a graph $\Gamma$ to itself. We denote by $\Aut(\Gamma)$ the set of graph automorphisms of $\Gamma$. This set is a group when considering composition $\circ$ as the group operation and the identity map $\mathrm{id}_\Gamma: V \to V$ as the identity group element $1_{\Aut(\Gamma)}$. In this case, instead of writing $g_1 \circ g_2$, we will simply write $g_1g_2$ to emphasize the group structure.

\subsection{Independent sets}

Given a subset $U \subseteq V$, the {\bf induced subgraph} by $U$, denoted $\Gamma[U]$, is the graph with set of vertices $U$ and set of edges $E \cap (U \times U)$. A subset $I \subseteq V$ is called an {\bf independent set} if $\Gamma[I]$ has no edges. We can also represent an independent set by its indicator function, i.e., by the map $x: V \to \{0,1\}$ so that
$$
\left[x(v) = 1 \text{ and } (v,v') \in E \right] \implies x(v') = 0.
$$
In addition, if we consider the finite graph $H_0 := (\{0,1\},\{(0,0),(0,1),(1,0)\})$, then $x$ can be also understood as a graph homomorphism from $\Gamma$ to $H_0$ (see Figure \ref{fig:hard}).

\begin{figure}[ht]
\centering
\includegraphics[scale = 0.8]{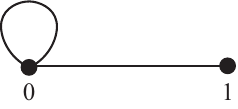}
\caption{The graph $H_0$}
\label{fig:hard}
\end{figure}

We denote by $X(\Gamma)$ the {\bf set of independent sets} of $\Gamma$. Notice that $X(\Gamma) \subseteq \{0,1\}^{V}$ can be identified with the set $\Hom(\Gamma,H_0)$ and that the \emph{empty independent set} $0^{V}$ always belongs to $X(\Gamma)$. Sometimes we will denote this independent set by $0^\Gamma$.

\subsection{Group actions}

Let $G$ be a subgroup of $\Aut(\Gamma)$. Given $g \in G$ and $v \in V$, the map $(g,v) \mapsto g \cdot v := g(v)$ is a {\bf (left) group action}, this is to say, $1_G \cdot v = v$ and $(gg') \cdot v = g \cdot (g' \cdot v)$ for all $g' \in G$, where $1_G = 1_{\Aut(\Gamma)}$. In this case, we say that \emph{$G$ acts on $\Gamma$} and denote this fact by $G \acts \Gamma$.

The group $G$ also acts on $\{0,1\}^{V}$ by precomposition. Given $g \in G$ and $x \in \{0,1\}^{V}$, consider the map $(g,x) \mapsto g \cdot x := x \circ g^{-1}$. A subset $X \subseteq \{0,1\}^{V}$ is called {\bf $G$-invariant} if $g \cdot X = X$ for all $g \in G$, where $g \cdot X := \{g \cdot x: x \in X\}$. Clearly, if $x \in X(\Gamma)$, then $g \cdot x$ and $g^{-1} \cdot x$ also belong to $X(\Gamma)$, since $g \in \Aut(\Gamma)$ and $x \in \Hom(\Gamma,H_0)$. Therefore, $X(\Gamma)$ is $G$-invariant and the action $G \acts X(\Gamma)$ is well defined.

We will usually use the letter $v$ to denote vertices in $V$, the letter $g$ to denote graph automorphisms in $G$, and the letter $x$ to denote independent sets in $X(\Gamma)$. In order to distinguish the action of $G$ on $V$ from the action of $G$ on $X(\Gamma)$, we will write $g v$ instead of $g \cdot v$, without risk of ambiguity.

The action $G \acts \Gamma$ is always {\bf faithful}, i.e., for all $g \in G \setminus \{1_G\}$, there exists $v \in V$ such that $g v \neq v$. The {\bf $G$-orbit} of a vertex $v \in V$ is the set $Gv := \{g v: g \in G\}$. The set of all $G$-orbits of $\Gamma$, denoted by $\Gamma/G$, is a partition of $V$ and it is called the {\bf quotient} of the action.

We say that a subset $\emptyset \neq U \subseteq V$ is {\bf dynamically generating} if $GU = V$, where $GF := \{g v: g \in F, v \in U\}$ for any $F \subseteq G$, and a {\bf fundamental domain} if it is also minimal, i.e., if $U' \subsetneq U$, then $GU' \subsetneq V$. The action $G \acts \Gamma$ is {\bf almost transitive} if $\lvert \Gamma/G \rvert < +\infty$ and {\bf transitive} if $\lvert \Gamma/G \rvert = 1$. A graph $\Gamma$ is called \emph{almost transitive} (resp. \emph{transitive}) if $\Aut(\Gamma) \acts \Gamma$ is almost transitive (resp. transitive).

The {\bf index} of a subgroup $H \leq G$, denoted by $[G:H]$, is the cardinality of the set of cosets $\{Hg: g \in G\}$. We will usually consider subgroups of finite index. In this case, we have that $\lvert \Gamma/H \rvert = \lvert \Gamma/G \rvert[G:H]$.

The {\bf $G$-stabilizer} of a vertex $v \in V$ is the subgroup $\Stab_G(v) := \{g \in G: g v = v\}$. Notice that, since $\Stab_G(gv) = g\Stab_G(v)g^{-1}$ for every $g \in G$, we have that $\lvert \Stab_G(v) \rvert = \lvert \Stab_G(v') \rvert$ for all $v' \in Gv$. If $\lvert \Stab_G(v) \rvert < \infty$ for all $v$, we say that the action is {\bf almost free}, and if $\lvert \Stab_G(v) \rvert = 1$ (i.e., if $\Stab_G(v) = \{1_G\}$) for all $v,$ we say that the action is {\bf free}.

A relevant observation is that if we assume that $\Gamma$ is countable and $G \acts \Gamma$ is almost transitive and almost free, then $G$ must be a countable group. In this work, we will only consider almost free and almost transitive actions. In this case, there exists a finite fundamental domain $\Rep \subseteq V$ such that $\lvert U_0 \rvert = \lvert \Gamma/G \rvert$ and, if $\Gamma$ is locally finite, then $\Gamma$ must have {\bf bounded degree}, i.e., there is a uniform bound on the number of vertices that each vertex is related to. In this case, we denote by $\Delta(\Gamma)$ the maximum degree among all vertices of the graph $\Gamma$.

\subsection{Transitive case: Cayley graphs}

Consider a subset $S \subseteq G$ that we assume to be \emph{symmetric}, i.e., $S = S^{-1}$, where $S^{-1} = \{s^{-1} \in S: s \in S\}$. We define the {\bf (right) Cayley graph} as $\Cay(G,S) = (V,E)$, where
$$
V = G	\quad	\text{ and }	\quad	E = \{(g,sg): g \in G, s \in S\}.
$$

Cayley graphs are a natural construction used to represent groups in a geometric fashion. In this context, it is common to ask that $1_G \notin S$, $S$ to be finite, and $S$ to be \emph{generating}, i.e., $G = \left<S\right>$, where
$$
\left<S\right> := \left\{s_1 \cdots s_k: s_i \in S \text{ for all } 1 \leq i \leq k \text{ and } k \in \mathbb{N}\right\}.
$$

Groups that have a set $S$ satisfying these conditions are called {\bf finitely generated}. Notice that if $1_G \notin S$, then $\Cay(G,S)$ is loopless, if $S$ is finite, then $\Cay(G,S)$ has bounded degree and, if $S$ is generating, then $\Cay(G,S)$ is connected. Now, suppose that $G \acts \Gamma$ is transitive (and free). Then, there exists a symmetric set $S \subseteq G$ such that
$$
\Gamma \cong \Cay(G,S).
$$

Indeed, it suffices to take $S = \{g \in G: (v,g v) \in E\}$, where $v \in V$ is arbitrary (see \cite{1-sabidussi}). 

We will be interested in Cayley graphs $\Gamma = \Cay(G,S)$ and their subgroup of automorphisms induced by group multiplication as a special and relevant case: given $g \in G$, we can define $f_g: \Gamma \to \Gamma$ as $f_g(g') = g'g$ and it is easy to check that $f_g \in \Aut(\Gamma)$. Then $G$ acts (as a group, from the left) on $\Gamma$ so that $g \cdot g' = f_{g^{-1}}(g') = g'g^{-1}$ for all $g' \in G$ and $G \hookrightarrow \Aut(\Gamma)$ by identifying $g$ with $f_{g^{-1}}$. In addition, via this identification, $G$ acts transitively on $\Gamma$ as a subgroup of graph automorphisms. In particular, every Cayley graph is transitive.

\subsection{Partition functions}

Given a graph $\Gamma = (V,E)$, let's consider $\lambda: V \to \mathbb{R}_{>0}$, an {\bf activity function}. We call the pair $(\Gamma,\lambda)$ a {\bf hardcore model}. We will say that a hardcore model $(\Gamma,\lambda)$ is \emph{finite} if $\Gamma$ is finite. If $U \subseteq V$ is a finite subset, a fact that we denote by $U \Subset V$, and $x \in X(\Gamma)$ is an independent set, we define the {\bf $\lambda$-weight} of $x$ on $U$ as
$$
\weight_\lambda(x,U) := \prod_{v \in U} \lambda(v)^{x(v)}
$$
and the {\bf $(\Gamma,U,\lambda)$-partition function} as
$$
Z_\Gamma(U,\lambda) := \sum_{x \in X(\Gamma,U)} \weight_\lambda(x,U) =  \sum_{x \in X(\Gamma,U)} \prod_{v \in U} \lambda(v)^{x(v)},
$$
where $X(\Gamma,U) := \{x \in X(\Gamma): x(v) = 0 \text{ for all } v \notin U\}$ is the finite set corresponding to the subset of independent sets of $\Gamma$ \emph{supported} on $U$. It is easy to check that there is an identification between $X(\Gamma,U)$ and $X(\Gamma[U])$. Then, the quantity $Z_\Gamma(U,\lambda)$ corresponds to the summation of independent sets of $\Gamma[U]$ weighted by $\lambda$. In the special case $\lambda \equiv 1$, we have that $Z_\Gamma(U,1) = \lvert X(\Gamma,U) \rvert = \lvert X(\Gamma[U]) \rvert$, i.e., the partition function is exactly the number of independent sets supported on $U$. If $(\Gamma,\lambda)$ is finite, we will simply write $Z_\Gamma(\lambda)$ instead of $Z_\Gamma(V,\lambda)$.

\begin{remark}
Notice that if $(v,v) \in E$ or $\lambda(v) = 0$, then $Z_\Gamma(U,\lambda) = Z_{\Gamma}(U \setminus \{v\},\lambda)$; due to this fact, we usually ask $\lambda$ to be strictly positive and that $\Gamma$ is loopless.
\end{remark}

\section{Free energy}
\label{sec3}

Now, suppose that we have an increasing sequence $\{U_n\}_n$ of finite subsets of vertices exhausting $\Gamma$, i.e., $U_n \subseteq U_{n+1}$ and $\bigcup_n U_n = V$. Tentatively, we would like to define the exponential growth rate of $Z_\Gamma(V_n,\lambda)$ as
$$
\lim_n \frac{\log Z_\Gamma(U_n, \lambda)}{\lvert U_n \rvert}.
$$

In order to guarantee the existence of this limit, we will provide a self-contained argument based on the particular properties of the hardcore model and \emph{amenability}. The reader that is familiar with this kind of arguments may skim over the next part and go then to Section \ref{sec4}.

\subsection{Amenability}

Let
$$
\F(G) := \{F \subseteq G: 0 < \lvert F \rvert < \infty\}
$$
be the set of finite nonempty subsets of $G$. Given $g \in G$ and $K,F \subseteq G$, we denote $Fg = \{hg: h \in F\}$, $gF = \{g h: h \in F\}$, $F^{-1} := \{g^{-1}: g \in F\}$, and $KF = \{hg: h \in K, g \in F\}$.

We say that $\{F_n\}_n \subseteq \F(G)$ is a {\bf right F{\o}lner sequence} if
$$
\lim_n \frac{\lvert F_ng \triangle F_n \rvert}{\lvert F_n \rvert} = 0	\quad	\text{for all } g \in G,
$$
where $\triangle$ denotes the symmetric difference. Similarly, $\{F_n\}_n$ is a {\bf left F{\o}lner sequence} if
$$
\lim_n \frac{\lvert g F_n \triangle F_n \rvert}{\lvert F_n \rvert} = 0	\quad	\text{for all } g \in G,
$$
and $\{F_n\}_n$ is a {\bf two-sided F{\o}lner sequence} if it is both a left and a right F{\o}lner sequence. The group $G$ is said to be {\bf amenable} if it has a (\emph{right} or \emph{left}) \emph{F{\o}lner sequence}. Notice that $\{F_n\}_n$ is left F{\o}lner if and only if $\{F_n^{-1}\}_n$ is right F{\o}lner. A F{\o}lner sequence $\{F_n\}_n$ is a {\bf F{\o}lner exhaustion} if in addition $F_n \subseteq F_{n+1}$ and $\bigcup_n F_n = G$. Every countable amenable group has a two-sided F{\o}lner exhaustion (see \cite[Theorem 4.10]{1-kerr}).

Every virtually amenable group is amenable. Moreover, the class of amenable groups contains all finite and all abelian groups, and it is closed under the operations of taking subgroups, and forming quotients, extensions, and directed unions (see \cite{1-ceccherini}).

\subsection{Growth rate of independent sets}

Given $\emptyset \neq U \Subset V$, define $\varphi_U: \F(G) \to \mathbb{R}$ as
$$
\varphi_{U}(F) := \log Z_\Gamma(F U, \lambda).
$$

From now on, we will assume that $\lambda: V \to \mathbb{R}_{>0}$ is $G$-invariant, this is to say,
$$
\lambda(g v) = \lambda(v)	\quad	\text{for all } g \in G.
$$

In other words, $\lambda$ is constant along the $G$-orbits, so it achieves at most $\lvert \Gamma/G \rvert$ different values. We denote by $\lambda_+$ and $\lambda_-$ the maximum and minimum among these values, respectively.

Now, let $W$ be an abstract set, $M$ a finite subset of $W$, and $k \in \mathbb{N}$. We will say that a finite collection $\mathcal{K}$ of nonempty finite subsets of $W$, with possible repetitions, is a {\bf $k$-cover} of $M$ if $\sum_{K \in \K} \mathbbm{1}_K \geq k  \mathbbm{1}_M$, where $\mathbbm{1}_A: W \to \{0,1\}$ denotes the indicator function of a set $A \subseteq W$. The following lemma is due to Downarowicz, Frej, and Romagnoli.

\begin{lemma}[{\cite{1-downarowicz}}]
\label{lem:downarowicz}
Let $Y$ be a subset of $A^{n}$, where $A$ is a finite set and $n \in \mathbb{N}$. Let $\mathcal{K}$ be a $k$-cover of the set of coordinates $M = \{1,\dots, n\}$. For $K \in \mathcal{K}$, let $Y_{K} = \{y_K: y \in Y\}$, where $y_K$ is the restriction of $y$ to $K$. Then,
$$
\lvert Y \rvert \leq \prod_{K \in \mathcal{K}}\lvert Y_{K}\rvert^{\frac{1}{k}}.
$$
\end{lemma}

Given $\varphi: \F(G) \to \mathbb{R}$, we will say that $\varphi$ satisfies {\bf Shearer's inequality} if
$$
\varphi(F) \leq \frac{1}{k} \sum_{K \in \K} \varphi(K)
$$ 
for all $F \in \F(G)$ and for all $k$-cover $\K$ of $F$ with $K \subseteq F$ for all $K \in \K$. We have the following theorem.

\begin{theorem}[{\cite[Theorem 4.48]{1-kerr}}]
\label{kerr-li}
Given a countable amenable group $G$, suppose that $\varphi: \F(G) \to \mathbb{R}$ satisfies Shearer's inequality and $\varphi(Fg) = \varphi(F)$ for all $F \in \F(G)$ and $g \in G$. Then,
$$
\lim_n \frac{\varphi(F_n)}{\lvert F_n \rvert} = \inf_{F \in \F(G)} \frac{\varphi(F)}{\lvert F \rvert}
$$
for any F{\o}lner sequence $\{F_n\}_n$.
\end{theorem}

Considering the two previous results, we obtain the next lemma.

\begin{lemma}
\label{lem:rational}
Given a fundamental domain $U_0$ of $G \acts \Gamma$ and $\lambda: V \to \mathbb{Q}_{> 0}$ such that $\lambda(v) = \frac{p_v}{q_v}$ with $p_v, q_v \in \mathbb{N}$ for all $v \in V$, we have that, for any F{\o}lner sequence $\{F_n\}_n$,
$$
\lim_n \frac{\varphi(F_n)}{\lvert F_n \rvert} = \inf_{F \in \F(G)} \frac{\varphi(F)}{\lvert F \rvert},
$$
where $\varphi: \F(G) \to \mathbb{R}$ is given by $\varphi(F) = \log Z_\Gamma(FU_0, \lambda) + \lvert F \rvert \sum_{v \in U_0} \log q_v$.
\end{lemma}

\begin{proof}
Given $F \in \F(G)$ and $k \in \mathbb{N}$, let $\K$ be a $k$-cover of $F$ with $K \subseteq F$ for all $K \in \K$. Notice that
\begin{align*}
Z_\Gamma(FU_0,\lambda)	&	=	\sum_{x \in X(\Gamma, FU_0)}  \prod_{v \in FU_0} \left(\frac{p_v}{q_v}\right)^{x(v)}						\\
						&	=	\frac{1}{\prod_{v \in FU_0} q_v} \sum_{x \in X(\Gamma, FU_0)}  \prod_{v \in FU_0} p_v^{x(v)} q_v^{1-x(v)}.
\end{align*}

Consider $q := \max_v q_v$, $p := \max_v p_v$, $A := \{-q,\dots,-1\} \cup \{1,\dots,p\}$, and
$$
Y := \left\{y \in A^{FU_0}: -q_v \leq y(v) \leq p_v \text{ and } \left[y(v) \geq 1 \land (v,v') \in E(\Gamma)\right] \implies y(v') \leq -1\right\}.
$$

Notice that
$$
\lvert Y \rvert = \sum_{x \in X(\Gamma,FU_0)}  \prod_{v \in FU_0} p_v^{x(v)} q_v^{1-x(v)}.
$$

Therefore, by Lemma \ref{lem:downarowicz}, and noticing that $\lvert Y_{KU_0} \rvert = \prod_{v \in KU_0} q_v \cdot Z_\Gamma(KU_0,\lambda)$, we have that
$$
\prod_{v \in FU_0} q_v \cdot Z_\Gamma(FU_0,\lambda) = \lvert Y \rvert \leq \prod_{K \in \mathcal{K}}\lvert Y_{KU_0} \rvert^{\frac{1}{k}} \leq \prod_{K \in \mathcal{K}} \left(\prod_{v \in KU_0} q_v \cdot Z_\Gamma(KU_0,\lambda)\right)^{\frac{1}{k}},
$$
where we use that $\{KU_0: K \in \K\}$ is a $k$-cover of $FU_0$. Therefore, by $G$-invariance of $\lambda$,
\begin{align*}
\varphi(F)	&	= 	\log Z_\Gamma(FU_0, \lambda) + \lvert F \rvert \sum_{v \in U_0} \log q_v	\\
		&	\leq	\frac{1}{k} \sum_{K \in \mathcal{K}}\left( \log Z_\Gamma(KU_0,\lambda) + \lvert K \rvert\sum_{v \in U_0} \log q_v \right)	\\
		&	=	\frac{1}{k} \sum_{K \in \mathcal{K}} \varphi(K),
\end{align*}
so $\varphi$ satisfies Shearer's inequality. On the other hand, by $G$-invariance of $X(\Gamma)$ and $\lambda$, it follows that $\varphi(Fg) = \varphi(F)$ for all $F \in \F(G)$ and $g \in G$. Therefore, by Theorem \ref{kerr-li}, we conclude.
\end{proof}

\begin{proposition}
\label{prop:energy}
Given a fundamental domain $U_0$ of $G \acts \Gamma$, we have that
$$
\lim_n \frac{\log Z_\Gamma(F_nU_0, \lambda)}{\lvert F_n \rvert} = \inf_{F \in \F(G)} \frac{\log Z_\Gamma(FU_0, \lambda)}{\lvert F \rvert},
$$
for any F{\o}lner sequence $\{F_n\}_n$.
\end{proposition}

\begin{proof}
First, suppose that $\lambda$ only takes rational values, i.e., $\lambda: V \to \mathbb{Q}_{>0}$ so that $\lambda(v) = \frac{p_v}{q_v}$ for all $v \in V$. By Lemma \ref{lem:rational}, for $\varphi(F) = \log Z_\Gamma(FU_0, \lambda) + \lvert F \rvert \sum_{v \in U_0} q_v$, we have that
\begin{align*}
\lim_{n} \frac{\log Z_\Gamma(F_nU_0, \lambda)}{\lvert F_n \rvert} + \sum_{v \in U_0} q_v	&	=	\lim_{n} \frac{\varphi(F_n)}{\lvert F_n \rvert}	\\
															&	= 	\inf_{F \in \F(G)} \frac{\varphi(F)}{\lvert F \rvert}	\\
															&	=	\inf_{F \in \F(G)} \frac{\log Z_\Gamma(FU_0, \lambda)}{\lvert F \rvert} + \sum_{v \in U_0} q_v,
\end{align*}
and, after cancelling out $\sum_{v \in U_0} q_v$, we obtain that
$$
\lim_{n} \frac{\log Z_\Gamma(F_nU_0, \lambda)}{\lvert F_n \rvert} =  \inf_{F \in \F(G)} \frac{\log Z_\Gamma(FU_0, \lambda)}{\lvert F \rvert}.
$$

Now, given a general $\lambda$, we can always approximate it by some $G$-invariant $\tilde{\lambda}: V \to \mathbb{Q}_{>0}$ arbitrarily close in the supremum norm. Given $\epsilon > 0$, pick $\tilde{\lambda}$ so that $\tilde{\lambda}(v) \leq \lambda(v)\ \leq (1+\epsilon)\tilde{\lambda}(v)$ for every $v$. Then,
\begin{align*}
\log Z_\Gamma(FU_0,\tilde{\lambda})	&	\leq	\log Z_\Gamma(FU_0,\lambda)	\\
								&	\leq	\log Z_\Gamma(FU_0,(1+\epsilon)\tilde{\lambda})	\\
								&	\leq	\lvert FU_0 \rvert\log(1+\epsilon) + \log Z_\Gamma(FU_0,\tilde{\lambda}),
\end{align*}
so,
$$
\frac{\log Z_\Gamma(FU_0,\tilde{\lambda})}{\lvert F \rvert} \leq \frac{\log Z_\Gamma(FU_0,\lambda)}{\lvert F \rvert} \leq \lvert U_0 \rvert\log(1+\epsilon) + \frac{\log Z_\Gamma(FU_0,\tilde{\lambda})}{\lvert F \rvert}.
$$

Therefore,
\begin{align*}
\liminf_n \frac{\log Z_\Gamma(F_nU_0,\lambda)}{\lvert F_n \rvert}	&	\geq	\inf_{F \in \F(G)} \frac{\log Z_\Gamma(FU_0,\lambda)}{\lvert F \rvert}		\\
											&	\geq	\inf_{F \in \F(G)} \frac{\log Z_\Gamma(FU_0,\tilde{\lambda})}{\lvert F \rvert}	\\
											&	= \lim_n \frac{\log Z_\Gamma(F_nU_0,\tilde{\lambda})}{\lvert F_n \rvert}	\\
											&	\geq \limsup_n \frac{\log Z_\Gamma(F_nU_0,\lambda)}{\lvert F_n \rvert} - \lvert U_0 \rvert\log(1+\epsilon),
\end{align*}
and since $\epsilon$ was arbitrary, we conclude.
\end{proof}

In order to fully characterize $\lim_n \frac{\log Z_\Gamma(U_n, \lambda)}{\lvert U_n \rvert}$, we have the following lemma. 

\begin{lemma}
\label{lem:stab}
Let $\{F_n\}_n$ be F{\o}lner sequence and $U_0$ a fundamental domain. Then, for any F{\o}lner sequence $\{F_n\}_n$,
$$
\lim_n \frac{\lvert F_n U_0 \rvert}{\lvert F_n \rvert} = \sum_{v \in U_0} \lvert \Stab_G(v) \rvert^{-1}.
$$
\end{lemma}

\begin{proof}
First, pick $v \in U_0$. Since $\Stab_G(v)$ is finite and $\{F_n\}_n$ is a F{\o}lner sequence, we have that $\lim_n \frac{\lvert F_n \Stab_G(v) \rvert}{\lvert F_n \rvert} = 1$. On the other hand, $F_n \Stab_G(v) v = F_n v$ and for each $v' \in F_n v$, there exist exactly $\lvert \Stab_G(v) \rvert$ different elements $g \in F_n \Stab_G(v)$ such that $g v = v'$. In other words,
$$
\lvert F_n \Stab_G(v) \rvert = \lvert F_n v \rvert \lvert \Stab_G(v) \rvert,
$$
so,
$$
\lim_n \frac{\lvert F_n v \rvert}{\lvert F_n \rvert} = \lim_n \frac{\lvert F_n \Stab_G(v) \rvert}{\lvert F_n \rvert\lvert \Stab_G(v) \rvert} = \lvert \Stab_G(v) \rvert^{-1}.
$$

Therefore,
$$
\lim_n \frac{\lvert F_n U_0 \rvert}{\lvert F_n \rvert} = \sum_{v \in U_0}  \lim_n \frac{\lvert F_n v \rvert}{\lvert F_n \rvert} = \sum_{v \in U_0} \lvert \Stab_G(v) \rvert^{-1}.
$$
\end{proof}

Now, given a fundamental domain $U_0$, define
$$
f_{G}(\Gamma,U_0,\lambda) := \inf_{F \in \F(G)}\frac{\log Z_\Gamma(FU_0, \lambda)}{\lvert FU_0 \rvert}.
$$
which, by Proposition \ref{prop:energy} and Lemma \ref{lem:stab}, is equal to
$$
\left(\sum_{v \in U_0} \lvert \Stab_G(v) \rvert^{-1}\right)^{-1} \lim_n \frac{\log Z_\Gamma(F_nU_0, \lambda)}{\lvert F_n \rvert}
$$
for any F{\o}lner sequence $\{F_n\}_n$ and, in particular, for any F{\o}lner exhaustion. Notice that, since $GU_0 = V$, the sequence $\{U_n\}_n$ defined as $U_n = F_nU_0$ is an exhaustion of $V$ in the sense that we were looking for. Now we will see that $f_{G}(\Gamma,U_0,\lambda)$ is independent of $U_0$.

\begin{proposition}
Given two fundamental domains $U_0$ and $U_0'$ of $G \acts \Gamma$, we have that
\begin{equation*}
f_{G}(\Gamma,U_0,\lambda) = f_{G}(\Gamma,U_0',\lambda).
\end{equation*}
\end{proposition}

\begin{proof}
Since $V = GU_0 = GU_0'$, there must exist $K,K' \in \F(G)$ such that $U_0' \subseteq KU_0$ and $U_0 \subseteq K'U_0'$. Then, for every $F \in \F(G)$,
\begin{align*}
FU_0 \triangle FU_0'	&	=		(FU_0 \setminus FU_0') \cup (FU_0' \setminus FU_0)	\\
				&	\subseteq	(FK'U_0' \setminus FU_0') \cup (FKU_0 \setminus FU_0)	\\
				&	=		(FK' \setminus F)U_0' \cup (FK \setminus F)U_0.
\end{align*}

Therefore, $\lvert FU_0 \triangle FU_0 \rvert \leq \lvert FK' \setminus F \rvert\lvert U_0' \rvert + \lvert FK \setminus F \rvert\lvert U_0 \rvert$. Now, notice that for $U, U' \Subset V$, we always have that
\begin{enumerate}
\item $Z_\Gamma(U \cup U',\lambda) \leq Z_\Gamma(U,\lambda) \cdot Z_\Gamma(U',\lambda)$, provided $U \cap U' = \emptyset$;
\item $Z_\Gamma(U,\lambda) \leq Z_\Gamma(U',\lambda)$, provided $U \subseteq U'$; and
\item $Z_\Gamma(U,\lambda) \leq (2\max\{1,\lambda_+\})^{\lvert U \rvert}$,
\end{enumerate}
so
\begin{align*}
\log Z_\Gamma(FU_0,\lambda)	&	\leq \log Z_\Gamma(FU_0 \cap FU_0',\lambda) + \log Z_\Gamma(FU_0 \setminus FU_0',\lambda)			\\
						&	\leq \log Z_\Gamma(FU_0',\lambda) + \log Z_\Gamma(FU_0 \triangle FU_0',\lambda)						\\
						&	\leq \log Z_\Gamma(FU_0',\lambda) + \lvert FU_0 \triangle FU_0' \rvert\log(2\max\{1,\lambda_+\})						\\
						&	\leq \log Z_\Gamma(FU_0',\lambda) + (\lvert FK' \setminus F \rvert\lvert U_0' \rvert + \lvert FK \setminus F \rvert\lvert U_0 \rvert)\log(2\max\{1,\lambda_+\}).
\end{align*}

Finally, since $\lvert U_0 \rvert = \lvert U_0' \rvert$ and $\lvert FU_0 \rvert = \lvert F \rvert\lvert U_0 \rvert$, it follows by amenability that
\begin{align*}
f_G(\Gamma,U_0,\lambda)	&	=	\lim_n \frac{\log Z_\Gamma(F_nU_0,\lambda)}{\lvert F_nU_0 \rvert}		\\
						&	\leq	\lim_n \frac{\log Z_\Gamma(F_nU_0',\lambda)}{\lvert F_nU_0' \rvert} 	\\
						&	\qquad	+ \lim_n\left( \frac{\lvert F_nK' \setminus F_n \rvert}{\lvert F_n \rvert} + \frac{\lvert F_nK \setminus F_n \rvert}{\lvert F_n \rvert}\right)\log(2\max\{1,\lambda_+\})	\\
						&	=	f_{G}(\Gamma,U_0',\lambda),
\end{align*}
and by symmetry of the argument, we conclude.
\end{proof}

Then, we can consistently define the {\bf Gibbs $(\Gamma,\lambda)$-free energy according to $G$} as
\begin{equation*}
f_G(\Gamma,\lambda) := f_G(\Gamma,U_0,\lambda),
\end{equation*}
where $U_0$ is an arbitrary fundamental domain of $G \acts \Gamma$. In addition, it is easy to see that if $G_1$ and $G_2$ act almost transitively on $\Gamma$ and the $G_1$-orbits and $G_2$-orbits coincide, i.e., $G_1 v = G_2 v$ for all $v \in V$, then
\begin{equation*}
f_{G_1}(\Gamma,\lambda) = f_{G_2}(\Gamma,\lambda).
\end{equation*}

In particular, we have that $f_G(\Gamma,\lambda)$ is equal for all $G$ acting transitively on $\Gamma$. Then, we can define the {\bf Gibbs $(\Gamma,\lambda)$-free energy} as
\begin{equation*}
f(\Gamma,\lambda) := \inf_{\emptyset \neq U \Subset V} \frac{\log Z_\Gamma(U, \lambda)}{\lvert U \rvert},
\end{equation*}
which is a quantity that only depends on the graph $\Gamma$ and the activity function $\lambda$, and satisfies that $f(\Gamma,\lambda) = f_G(\Gamma,\lambda)$ for any $G \leq \Aut(\Gamma)$ acting transitively on $\Gamma$.

\begin{remark}
In the almost transitive case, $f_G(\Gamma,\lambda)$ does not necessarily coincide with $f(\Gamma,\lambda)$ for $G$ acting almost transitively: consider the graph $\Gamma$ obtained by taking the disjoint union of $\Gamma_1 = \Cay(\mathbb{Z},\emptyset)$ and $\Gamma_2 = \Cay(\mathbb{Z},\{1,-1\})$ and the constant activity function $\lambda \equiv 1$. Then, $f_{\mathbb{Z}}(\Gamma_1,1) = \log 2$ and $f_{\mathbb{Z}}(\Gamma_2,1) = \log\left(\frac{1+\sqrt{5}}{2}\right)$, so
$$
f_{\mathbb{Z}}(\Gamma,1) = \frac{1}{2}(f_{\mathbb{Z}}(\Gamma_1,1) + f_{\mathbb{Z}}(\Gamma_2,1)) > f_{\mathbb{Z}}(\Gamma_2,1) \geq \inf_{\emptyset \neq U \Subset V} \frac{\log Z_\Gamma(U, 1)}{\lvert U \rvert}.
$$

The value of $f_{\mathbb{Z}}(\Gamma_2,1)$ corresponds to the \emph{topological entropy} of the \emph{golden mean shift} (see \cite[Example 4.1.4]{1-lind} and Section \ref{sec9}).
\end{remark}

The main theme of this paper will be to explore our ability to approximate $f_G(\Gamma,\lambda)$. From now, and without much loss of generality, we will assume that $G \acts \Gamma$ is free (see Section \ref{subsec:9-7} for a reduction of the almost free case to the free case).

\section{Gibbs measures}
\label{sec4}

Given a graph $\Gamma = (V,E)$, consider the set $\{0,1\}^{V}$ endowed with the product topology and the set $X(\Gamma)$, with the subspace topology. The set of independent sets $X(\Gamma)$ is a compact and metrizable space. A base for the topology is given by the {\bf cylinder sets}
$$
[x_U] := \{x' \in X(\Gamma): x'_U = x_U\}
$$
for $U \Subset V$ and $x \in X(\Gamma)$, where $x_U$ denotes the \emph{restriction} of $x$ from $V$ to $U$. If $U$ is a singleton $\{v\}$, we will omit the brackets and simply write $x_v$ and the same convention will hold in analogous instances. Given $W \subseteq V$, we denote by $\mathcal{B}_W$ the smallest $\sigma$-algebra generated by
$$
\{[x_U]: U \Subset W, x \in X(\Gamma)\},
$$
and by $\mathcal{B}_\Gamma$ the {\bf Borel $\sigma$-algebra}, which corresponds to $\mathcal{B}_{V}$.

Let $\mathcal{M}(X(\Gamma))$ be the set of Borel probability measures $\Prob$ on $X(\Gamma)$. We say that $\Prob$ is {\bf $G$-invariant} if $\Prob(A) = \Prob(g \cdot A)$ for all $A \in \mathcal{B}_\Gamma$ and $g \in G$, and {\bf $G$-ergodic} if $g \cdot A = A$ for all $g \in G$ implies that $\Prob(A) \in \{0,1\}$. We will denote by $\mathcal{M}_G(X(\Gamma))$ and $\mathcal{M}_G^{\mathrm{erg}}(X(\Gamma))$ the set of $G$-invariant and the subset of $G$-invariant measures that are $G$-ergodic, respectively.

For $\Prob \in \mathcal{M}(X(\Gamma))$, define the {\bf support} of $\Prob$ as
$$
\supp(\Prob) := \left\{x \in X(\Gamma): \Prob([x_U]) > 0 \text{ for all } U \Subset V\right\}.
$$

Given $\emptyset \neq U \Subset V$ and $y \in X(\Gamma)$, we define $\pi^y_U$ to be the probability distribution on $X(\Gamma,U)$ given by 
$$
\pi_U^y(x) := \weight^y_\lambda(x,U) Z^y_\Gamma(U,\lambda)^{-1},
$$
where $\weight^y_\lambda(x,U) = \weight_\lambda(x,U) \mathbbm{1}_{[y_{U^{\rm c}}]}(x)$ and $Z^y_\Gamma(U,\lambda) = \sum_x \weight^y_\lambda(x,U)$. In other words, to each independent set $x$ supported on $U$, we associate a probability proportional to its $\lambda$-weight over $U$, $\prod_{v \in U} \lambda(v)^{x(v)}$, provided $x_U$ is \emph{compatible} with $y_{U^{\rm c}}$, in the sense that the element $z \in \{0,1\}^{V}$ such that $z_U = x_U$ and $z_{U^{\rm c}} = y_{U^{\rm c}}$ is an independent set.

Now, given an activity function $\lambda: V \to \mathbb{R}_{>0}$, consider the hardcore model $(\Gamma, \lambda)$ and the collection $\pi_{\Gamma,\lambda} = \{\pi^y_U: U \Subset V, y \in X(\Gamma)\}$. We call $\pi_{\Gamma,\lambda}$ the {\bf Gibbs $(\Gamma,\lambda)$-spe\-ci\-fi\-ca\-tion}. A measure $\Prob \in \mathcal{M}(X(\Gamma))$ is called a {\bf Gibbs measure (for $(\Gamma,\lambda)$)} if for all $U \Subset V$, $U' \subseteq U$, and $x \in X(\Gamma)$,
$$
\Prob([x_{U'}] \vert \mathcal{B}_{U^{\rm c}})(y) = \pi^y_U([x_{U'}])	\quad	\Prob\text{-a.s. in } y,
$$
where $\pi^y_U([x_U'])$ denotes the marginalization
$$
\pi^y_U([x_U']) = \sum_{x' \in X(\Gamma[U]): x'_{U'} = x_{U'}} \pi^y_U(x')
$$
and $\Prob(A \vert \mathcal{B}_U) = \mathbb{E}_{\Prob}(\mathbbm{1}_{A} \vert \mathcal{B}_U)$ for $A \in \mathcal{B}_\Gamma$. We denote by $\mathcal{M}_{\mathrm{Gibbs}}(\Gamma,\lambda)$ the set of Gibbs measures for $(\Gamma,\lambda)$.

An important question in statistical physics is whether the set of Gibbs measures is empty or not, and if not, whether there is a unique or multiple Gibbs measures \cite{1-georgii}.

\subsection{The locally finite case}

The model described in \cite[Example 4.16]{1-georgii} can be understood as an attempt to formalize the idea of a system where there is a single particle $1$ (uniformly distributed) or none, i.e., $0$ everywhere. There, it is proven that this model cannot be represented as a Gibbs measure. This example can be also viewed as a hardcore model in a countable graph that is \emph{complete} (i.e., there is an edge between any pair of different vertices) and, in particular, in a non-locally finite graph. In other words, there exist examples of non-locally finite graphs $\Gamma$ such that the $(\Gamma,\lambda)$-specification $\pi_{\Gamma,\lambda}$ has no Gibbs measure.

From now on, we will always assume that $\Gamma$ is locally finite. In this case, the existence of Gibbs measures is guaranteed (see \cite{2-brightwell,2-dobrushin}) and, moreover, every Gibbs measure must be a \emph{Markov random field} that is \emph{fully supported}.

Indeed, it can be checked that $\pi_{\Gamma,\lambda}$ is an example of a \emph{Markovian specification} (see \cite[Example 8.24]{1-georgii}). In this case, any Gibbs measure $\Prob \in \mathcal{M}_{\mathrm{Gibbs}}(\Gamma,\lambda)$ satisfies the following \emph{local Markov property}:
$$
\Prob([x_U] \vert \mathcal{B}_{U^{\rm c}})(y) = \Prob([x_U] \vert \mathcal{B}_{\partial U})(y)	\quad	\Prob\text{-a.s. in } y,
$$
for any $U \Subset V$ and $x \in X(\Gamma)$. In other words, $\Prob$ is a {\bf Markov random field}, so any event supported on a finite set conditioned to a specific value on its boundary is independent of events supported on the complement.

In addition, it can be checked that any Gibbs measure $\Prob$ must be {\bf fully supported}, i.e., $\supp(\Prob) = X(\Gamma)$. Indeed, it suffices to check that $X(\Gamma) \subseteq \supp(\Prob)$; the other direction follows directly from the definition of $\pi_{\Gamma,\lambda}$ and Gibbs measures. Now, given $x \in X(\Gamma)$ and $U \Subset V$, we would like to check that $\Prob([x_U]) > 0$. To prove this, observe that given $x \in X(\Gamma)$, we have that $z \in \{0,1\}^{V}$ defined as $z_U = x_U$, $z_{\partial U} \equiv 0$, and $z_{W^c} = y_{W^c}$, always belongs to $X(\Gamma)$ for any $y \in X(\Gamma)$, where $W = U \cup \partial U$. In particular, $\pi^y_{W}(z) > 0$ for any $y \in X(\Gamma)$. Then, considering that $\partial (W^c)$ is finite,
\begin{align*}	
\Prob([x_U])	&	\geq	\Prob([z_{W}])																			\\
			&	=	\sum_{y \in X(\Gamma,\partial W): \Prob([y_{\partial W}]) > 0} \Prob([z_{W}] \vert [y_{\partial W}])\Prob([y_{\partial W}])	\\
			&	=	\sum_{y \in X(\Gamma,\partial W): \Prob([y_{\partial W}]) > 0} \pi^y_{W}(z)\Prob([y_{\partial W}])					\\
			&	\geq	 \pi^{y^*}_{W}(z)\Prob([y^*_{\partial W}]) > 0,
\end{align*}
since $\Prob$ is a probability measure and there must exist $y^* \in X(\Gamma)$ such that $\Prob([y^*_{\partial W}]) > 0$. In other words, $X(\Gamma)$ satisfies the property (D*) introduced in \cite[1.14 Remark]{1-ruelle}, which guarantees full support.

\subsection{Spatial mixing and uniqueness}

Given a Gibbs $(\Gamma,\lambda)$-specification $\pi_{\Gamma,\lambda}$, we define two \emph{spatial mixing} properties fundamental to this work.

\begin{definition}
We say that a hardcore model $(\Gamma,\lambda)$ exhibits {\bf strong spatial mixing (SSM)} if there exists a \emph{decay rate} function $\delta: \mathbb{N} \to \mathbb{R}_{\geq 0}$ such that $\lim_{\ell \to \infty} \delta(\ell) = 0$ and for all $U \Subset V$, $v \in U$, and $y,z \in X(\Gamma)$,
$$
\vert \pi^y_U([0^v]) - \pi^z_U([0^v]) \rvert \leq \delta(\dist_\Gamma(v,D_U(y,z))),
$$
where  $[0^v]$ denotes the event that the vertex $v$ takes the value $0$ and
$$
D_U(y,z) := \{v' \in U^c: y(v') \neq z(v')\}.
$$
\end{definition}

This definition is equivalent (see \cite[Lemma 2.3]{2-marcus}) to the---a priori---stronger following property: for all $U' \subseteq U \Subset V$ and $x,y,z \in X(\Gamma)$,
$$
\lvert \pi^y_U([x_{U'}]) - \pi^z_U([x_{U'}]) \rvert \leq \lvert U' \rvert\delta(\dist_\Gamma(U',D_U(y,z))).
$$

Similarly, we say that $(\Gamma,\lambda)$ exhibits {\bf weak spatial mixing (WSM)} if for all $U' \subseteq U \Subset V$ and $x,y,z \in X(\Gamma)$,
$$
\lvert  \pi^y_U([x_{U'}]) - \pi^z_U([x_{U'}]) \rvert \leq \lvert U' \rvert \cdot \delta(\dist_\Gamma(U',U^c)).
$$

Clearly, SSM implies WSM. Moreover, it is well-known that, in this context, WSM (and therefore, SSM) implies uniqueness of Gibbs measures \cite{2-weitz}. In other words, $\mathcal{M}_{\mathrm{Gibbs}}(\Gamma,\lambda) = \{\Prob_{\Gamma,\lambda}\}$, where $\Prob_{\Gamma,\lambda}$ denotes the unique Gibbs measure for $(\Gamma,\lambda)$. In this case, $\Prob_{\Gamma,\lambda}$ is always $\Aut(\Gamma)$-invariant.

We say that $(\Gamma,\lambda)$ exhibits {\bf exponential SSM} (resp. {\bf exponential WSM}) if there exist constants $C,\alpha > 0$ such that $\pi_{\Gamma,\lambda}$ exhibits SSM (resp. WSM) with decay rate function $\delta(n) = C \cdot \exp(-\alpha \cdot n)$. 

Given $U \subseteq V$, we denote by $\Gamma \setminus U$ the subgraph induced by $V \setminus U$, i.e., $\Gamma[V \setminus U]$. We have the following result due to Gamarnik and Katz.

\begin{proposition}[{\cite[Proposition 1]{1-gamarnik}}]
\label{prop:GKssm}
If a hardcore model $(\Gamma,\lambda)$ satisfies SSM, then so does the hardcore model $(\Gamma',\lambda)$ for any subgraph $\Gamma'$ of $\Gamma$. The same assertion applies to exponential SSM. Moreover, for every $U \subseteq V$ and $v \in V \setminus U$, the following identity holds:
$$
\Prob_{\Gamma,\lambda}([0^v] \vert [0^U]) = \Prob_{\Gamma \setminus U,\lambda}([0^v]),
$$
where $\Prob_{\Gamma,\lambda}$ and $\Prob_{\Gamma \setminus U,\lambda}$ are the unique Gibbs measures for $(\Gamma,\lambda)$ and $(\Gamma \setminus U,\lambda)$, respectively, and $[0^U]$ denotes the event that all the vertices in $U$ take the value $0$. In particular, $\Prob_{\Gamma,\lambda}([0^v] \vert [0^U])$ is always well-defined, even if $U$ is infinite.
\end{proposition}

\begin{remark}
\label{ssm}
Notice that any event of the form $[x_U]$ can be translated into an event of the form $[0^{U' }]$ for a suitable set $U'$: it suffices to define $U' = U \cup \partial\{v \in U: x_U(v) = 1\}$ since, deterministically, every neighbor of a vertex colored $1$ must be $0$, so Proposition \ref{prop:GKssm} still holds for more general events. We also remark that in \cite{1-gamarnik} it is assumed that $\lambda$ is a constant function. Here we drop this assumption, but it is direct to check that the same proof of \cite[Proposition 1]{1-gamarnik} also applies to the more general non-constant case.
\end{remark}

\subsection{Families of hardcore models}

We will denote by $\mathcal{H}$ the family of hardcore models $(\Gamma, \lambda)$ such that $\Gamma$ is a countable locally finite graph and $\lambda$ is any activity function $\lambda: V(\Gamma) \to \mathbb{R}_{>0}$.

Given a countable group $G$, we will denote by $\mathcal{H}_G$ the set of hardcore models $(\Gamma,\lambda)$ in $\mathcal{H}$ for which $G$ is isomorphic to some subgroup of $\Aut(\Gamma)$ such that $G \acts \Gamma$ is free and almost transitive and $\lambda: V(\Gamma) \to \mathbb{R}_{>0}$ is a $G$-invariant activity function.

Given a positive integer $\Delta$, we will denote by $\mathcal{H}^\Delta$ the set of hardcore models $(\Gamma,\lambda)$ in $\mathcal{H}$ such that $\Delta(\Gamma) \leq \Delta$. Notice that any hardcore model defined on the $\Delta$-regular (infinite) tree $\mathbb{T}_\Delta$ belongs to $\mathcal{H}^\Delta$. 

Given $\lambda_0 > 0$, we will denote by $\mathcal{H}(\lambda_0)$ the family of hardcore models $(\Gamma,\lambda)$ in $\mathcal{H}$ such that $\lambda_+ \leq \lambda_0$.

We will also combine the notation for these families in the natural way; for example, $\mathcal{H}_G^\Delta(\lambda_0)$ will denote the set of hardcore models $(\Gamma,\lambda)$ in $\mathcal{H}$ such that $G \acts \Gamma$ is free and almost transitive, $\lambda$ is $G$-invariant, $\Delta(\Gamma) \leq \Delta$, and $\lambda_+ \leq \lambda_0$.

\section{Trees}
\label{sec5}
 
Given a graph $\Gamma$, a {\bf trail} $w$ in $\Gamma$ is a finite sequence $w = (v_1,\dots,v_n)$ of vertices such that consecutive vertices are adjacent in $\Gamma$ and the edges $(v_i,v_{i+1})$ involved are not repeated. For a fixed vertex $v \in V(\Gamma)$, the {\bf tree of self-avoiding walks starting from $v$}, denoted by $T_{\mathrm{SAW}}(\Gamma,v)$, is defined as follows:
\begin{enumerate}
\item Consider the set $W_0$ of trails starting from $v$ that repeat no vertex and the set $W_1$ of trails that repeat a single vertex exactly once and then stop (i.e., the set of non-backtracking walks that end immediately after performing a cycle). We define $T_{\mathrm{SAW}}(\Gamma,v)$ to be a rooted tree with root $\rho = (v)$ such that the set of vertices $V(T_{\mathrm{SAW}}(\Gamma,v))$ is $W_0 \cup W_1$ and the set of (undirected) edges $E(T_{\mathrm{SAW}}(\Gamma,v))$ corresponds to all the pairs $(w,w')$ such that $w'$ is a one vertex extension of $w$ or vice versa. In simple words, $T_{\mathrm{SAW}}(\Gamma,v)$ is a rooted tree that represents all self-avoiding walks in $\Gamma$ that start from $v$. It is easy to check that the set of leaves of $T_{\mathrm{SAW}}(\Gamma,v)$ contains $W_1$, but they are not necessarily equal (e.g., see vertex $b$ in Figure \ref{fig:diagram1}).
\item For $u \in V(\Gamma)$, consider an arbitrary ordering $\partial\{u\} = \{u_1, \dots, u_d\}$ of its neighbors. Given $w \in W_1$, we can represent this walk as a sequence $$w = (v,\dots, u,u_i,\dots,u_j,u),$$ with $u_i,u_j \in \partial\{u\}$. Notice that $i \neq j$, since we are not repeating edges. Considering this, we condition the ``terminal'' trail $w$ to be $1$ (\emph{occupied}) if $i < j$ and to be $0$ (\emph{unoccupied}) if $i > j$, inducing the corresponding effect of this conditioning in the graph (i.e., removing the vertex and its neighbors or just removing the vertex, respectively).
\end{enumerate}

Given a hardcore model $(\Gamma,\lambda)$, a vertex $v \in V(\Gamma)$, a subset $U \subseteq V(\Gamma)$, and an independent set $x \in X(\Gamma)$, we are interested in computing the marginal probability that $v$ is unoccupied in $\Gamma$ given the partial configuration $x_U$, i.e., $\Prob_{\Gamma,\lambda}([0^v] \vert [x_U])$. Notice that if $(\Gamma,\lambda)$ satisfies SSM (which includes the particular but relevant case of $\Gamma$ being finite), then this probability is always well defined due to Proposition \ref{prop:GKssm}, even if $U$ is infinite.

\begin{figure}[ht]
\centering
\includegraphics[scale = 0.75]{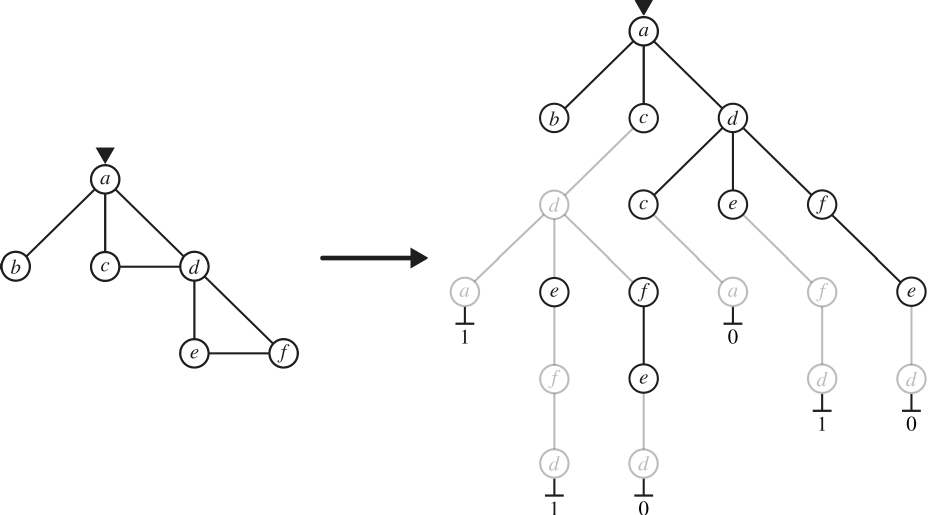}
\caption{A representation of a graph $\Gamma$ and its corresponding tree of self-avoiding walks $T_{\mathrm{SAW}}(\Gamma,v)$ including the conditioning of terminal trails ($\bot$). Here, the order of each neighborhood is alphabetical and every trail/vertex is represented by the final vertex of the trail in $\Gamma$ starting from $v = a$. See also \cite{1-weitz} for an explanation of the same picture}
\label{fig:diagram1}
\end{figure}

To understand better $\Prob_{\Gamma,\lambda}([0^v] \vert [x_U])$, we consider $(T_{\mathrm{SAW}}(\Gamma,v),\overline{\lambda})$ to be the hardcore model where $\overline{\lambda}(w) = \lambda(u)$ for every trail $w$ ending in $u$. In this context, a condition $x_U$ in $(\Gamma,\lambda)$ is translated into the condition $\overline{x_U}$ in $T_{\mathrm{SAW}}(\Gamma,v)$, whose support is the set $W(U)$ of trails $w$ that end in $u$ for some $u \in U$, and $\overline{x}(w) = x(u)$ for all these $w$s. We have the following result from \cite{1-weitz}, that we adapt to the more general non-constant $\lambda$ case and we include its proof for completeness.

\begin{theorem}[{\cite[Theorem 3.1]{1-weitz}}]
\label{thm21}
For every finite hardcore model $(\Gamma,\lambda)$, every $v \in V(\Gamma)$, and $U \subseteq V(\Gamma)$,
$$
\Prob_{\Gamma,\lambda}([0^v] \vert [x_U]) = \Prob_{T_{\mathrm{SAW}}(\Gamma,v),\overline{\lambda}}([0^\rho] \vert [\overline{x_U}]).
$$
\end{theorem}

\begin{proof}
Instead of probabilities, we work with the ratios
$$
R_{\Gamma,\lambda}(v,x_U) := \frac{\Prob_{\Gamma,\lambda}([1^v] \vert [x_U])}{\Prob_{\Gamma,\lambda}([0^v] \vert [x_U])},
$$
where if $v \in U$ and $x_U(v)$ is equal to $1$ or $0$, we let $R_{\Gamma,\lambda}(v,x_U)$ to be $\infty$ or $0$, respectively. Notice that
$$
\Prob_{\Gamma,\lambda}([0^v] \vert [x_U]) = \frac{1}{1+R_{\Gamma,\lambda}(v,x_U)}	\quad \mbox{ and } \quad \Prob_{\Gamma,\lambda}([1^v] \vert [x_U]) = \frac{R_{\Gamma,\lambda}(v,x_U)}{1+R_{\Gamma,\lambda}(v,x_U)}.
$$

Given a finite tree $T$ rooted at $\rho$, let's denote by $\left\{\rho_1,\dots,\rho_d\right\}$ the set of neighbors $\partial \{\rho\}$ of $\rho$ and by $T_i$, for $i=1,\dots,d$, the corresponding subtrees starting from $\rho_i$, i.e., $V(T) = \{\rho\} \cup V(T_1) \cup \cdots \cup V(T_d)$. If we have a condition $x_U$ on $U$, we define $U_i = U \cap V(T_i)$ and $x_{U_i} = (x_{U})\vert_{U_i}$. Considering this, we have that
\begin{align*}
R_{T,\lambda}(\rho, x_U)	&	=	\frac{\Prob_{T,\lambda}([1^\rho] \vert [x_U])}{\Prob_{T,\lambda}([0^\rho] \vert [x_U])}	\\
					&	=	\frac{\lambda(\rho) \cdot Z^{x_U}_{T \backslash \{\rho \cup \partial \{\rho\}\}}(\lambda)}{Z^{x_U}_T(\lambda)} \cdot \frac{Z^{x_U}_T(\lambda)}{Z^{x_U}_{T \backslash \{\rho\}}(\lambda)}	\\
					&	=	\lambda(\rho) \cdot \prod_{i=1}^{d}\frac{Z^{x_{U_i}}_{T_i \backslash \{\rho_i\}}(\lambda_i)}{Z^{x_{U_i}}_{T_i}(\lambda)}	\\										
					&	=	\lambda(\rho) \cdot \prod_{i=1}^{d} \Prob_{T_i,\lambda}([0^{\rho_i}] \vert [x_{U_i}])	\\
					&	=	\lambda(\rho) \cdot \prod_{i=1}^{d} \frac{1}{1+R_{T_i,\lambda}(\rho_i, x_{U_i})},
\end{align*}
where
$$
Z^{x_U}_\Gamma(\lambda) := \sum_{y \in X(\Gamma): y_U = x_U} \prod_{v \in V(\Gamma)} \lambda(v)^{y(v)}.
$$

Notice that this gives us a linear recursive procedure for computing $R_{T,\lambda}(\rho, x_U)$, and therefore $\Prob_{T,\lambda}([0^\rho] \vert [x_U])$, with base cases: $R_{T,\lambda}(\rho, x_U) = 0 \text{ or } +\infty$ if $\rho$ is fixed, and $R_{T,\lambda}(\rho, x_U) = \lambda(\rho)$ if $\rho$ is free and isolated.

Now, consider an arbitrary hardcore model $(\Gamma,\lambda)$ and $v \in V(\Gamma)$ with neighbors $\partial\{v\} = \{u_1,\dots,u_d\}$. We consider the auxiliary hardcore model $(\Gamma',\lambda')$, where
\begin{itemize}
	\item $V(\Gamma') = V(\Gamma) \backslash \{v\} \cup \{v_1,\dots,v_d\}$,
	\item $E(\Gamma') = E(\Gamma) \backslash \{(v,u_i)\}_{i=1,\dots,d} \cup \{(v_i,u_i)\}_{i=1,\dots,d}$,
	\item $\lambda'(v_i) = \lambda(v)^{1/d}$ for $i=1,\dots,d$, and $\lambda'(u) = \lambda(u)$, otherwise.
\end{itemize}

Notice that
\begin{align*}
R_{\Gamma,\lambda}(v, x_U)	&	=	\frac{\Prob_{\Gamma,\lambda}\left([1^v] \middle\vert [x_U]\right)}{\Prob_{\Gamma,\lambda}\left([0^v] \middle\vert [x_U]\right)}	\\
						&	=	\frac{\Prob_{\Gamma',\lambda'}\left([1^{\{v_1, \dots, v_d\}}] \middle\vert [x_U]\right)}{\Prob_{\Gamma',\lambda'}\left([0^{\{v_1, \dots, v_d\}}] \middle\vert [x_U]\right)}	\\
						&	=	\prod_{i=1}^{d}\frac{\Prob_{\Gamma',\lambda'}\left([0^{\{v_1,\dots,v_{i-1}\}}1^{\{v_i, \dots, v_d\}}] \middle\vert [x_U]\right)}{\Prob_{\Gamma',\lambda'}\left([0^{\{v_1,\dots,v_i\}}1^{\{v_{i+1}, \dots, v_d\}}] \middle\vert [x_U]\right)}	\\
						&	=	\prod_{i=1}^{d}\frac{\Prob_{\Gamma',\lambda'}\left([1^{v_i}] \middle\vert [x_Uz_i]\right)}{\Prob_{\Gamma',\lambda'}\left([0^{v_i}] \middle\vert [x_Uz_i]\right)}	\\
						&	=	\prod_{i=1}^{d}R_{\Gamma',\lambda'}(v_i,x_Uz_i),
\end{align*}
where $z_i = 0^{\{v_1,\dots,v_{i-1}\}}1^{\{v_{i+1}, \dots, v_d\}}$ and $x_Uz_i$ is the concatenation of $x_U$ and $z_i$. Now, since $v_i$ is connected only to $u_i$, notice that
$$
R_{\Gamma',\lambda'}(v_i,x_Uz_i)	=	\frac{\lambda'(v_i) \cdot Z^{x_Uz_i}_{\Gamma' \backslash \{v_i,u_i\}}(\lambda')}{Z^{x_U z_i}_{\Gamma' \backslash \{v_i\}}(\lambda')}
																	=	\frac{\lambda^{1/d}(v)}{1+R_{\Gamma' \backslash \{v_i\},\lambda'}(u_i,x_Uz_i)}.
$$

Therefore,
$$
R_{\Gamma,\lambda}(v, x_U)	= \prod_{i=1}^{d}\frac{\lambda^{1/d}(v)}{1+R_{\Gamma' \backslash \{v_i\},\lambda'}(u_i,x_Uz_i)} = \lambda(v) \cdot \prod_{i=1}^{d}\frac{1}{1+ R_{\Gamma' \backslash \{v_i\},\lambda'}(u_i,x_Uz_i)}.
$$

Notice that the previous recursion can increase the original number of vertices, but the number of free vertices always decreases, so the recursion ends. Then, we have that
\begin{enumerate}
\item $R_{T,\lambda}(v, x_U) = \lambda(\rho) \cdot f\left(R_{T_1,\lambda}(\rho_1, x_{U_1}),\dots,R_{T_d,\lambda}(\rho_d, x_{U_d})\right)$ and
\item $R_{\Gamma,\lambda}(v, x_U) = \lambda(v) \cdot f\left(R_{\Gamma' \backslash \{v_1\},\lambda'}(u_1,x_U\tau_1),\dots,R_{\Gamma' \backslash \{v_d\},\lambda'}(u_d,x_U\tau_d)\right)$,
\end{enumerate}
where $f(r_1,\dots,r_d) = \prod_{i=1}^{d}\frac{1}{1+r_i}$. Now we proceed by induction in the number of free vertices. We can consider the base case where there are no free vertices (besides $v$) and the theorem is trivial.  Then, if we know that the theorem is true when we have $n$ free vertices, we prove it for $n+1$. Notice that if $R_{\Gamma',\lambda'}(v,x_U)$ involves $n+1$ free vertices, then $R_{\Gamma' \setminus  \{v\},\lambda'}(v_i,x_Uz_i)$ involves $n$ free vertices, so by the induction hypothesis,
$$
R_{\Gamma' \backslash \{v_i\},\lambda}(u_i,x_Uz_i) = R_{T_{\mathrm{SAW}}(\Gamma,v),\overline{\lambda}}(\rho_i,\overline{x_{U}z_i}).
$$

Then, noticing that the rooted subtree $(T_i,\rho_i)$ and the condition $\overline{x_{U}z_i}$ gives exactly the tree of self-avoiding walks of $\Gamma' \backslash \{v_i\}$ starting from $u_i$ under the condition $x_Uz_i$, we are done.
\end{proof}

\begin{figure}[ht]
\centering
\includegraphics[scale = 0.75]{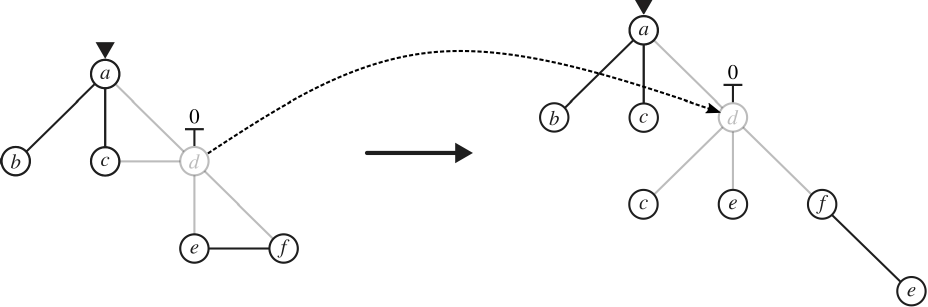}
\caption{A condition ($\top$) on $\Gamma$ and its representation on the tree of self-avoiding walks $T_{\mathrm{SAW}}(\Gamma,v)$ for $v = a$. The only relevant portion of the tree for computing the marginal probability associated to the root is its connected component}
\label{fig:diagram2}
\end{figure}

\begin{remark}
The recursions presented in the proof of Theorem \ref{thm21} give us a recursive procedure to compute the marginal probability of the root $\rho$ of a tree $T$ being occupied which requires linear time with respect to the size of the tree. On the other hand, if $\Gamma$ is such that $\Delta(\Gamma) \leq \Delta$, then $T_{\mathrm{SAW}}(\Gamma,v)$ is a subtree of $\mathbb{T}_\Delta$ and its size of $T_{\mathrm{SAW}}(\Gamma,v)$ can be (at most) exponential in the size of $\Gamma$. Since hardcore models are Markov random fields and we are interested in the sensitivity of the root $\rho$ associated to $v$, we only need to consider the graph obtained after pruning all the subtrees below $W(U)$ (see Figure \ref{fig:diagram2}).
\end{remark}

Before stating the main results concerning hardcore models and strong spatial mixing, we will establish the following bounds.

\begin{lemma}
\label{lem:probbound}
Given a finite hardcore model $(\Gamma,V) \in \mathcal{H}^\Delta$ and $v \in V$, we have that
$$
0 < \frac{1}{1+\lambda_+} \leq \Prob_{\Gamma,\lambda}([0^v]) \leq \frac{(1+\lambda_+)^\Delta}{\lambda_- + (1+\lambda_+)^\Delta} < 1
$$
and
$$
0 < \frac{\lambda_-}{\lambda_- + (1+\lambda_+)^\Delta} \leq \Prob_{\Gamma,\lambda}([1^v]) \leq \frac{\lambda_+}{1+\lambda_+}  < 1.
$$
\end{lemma}

\begin{proof}
Notice that, since a $1$ at $v$ forces $0$s in $\partial\{v\}$,
$$
\Prob_{\Gamma,\lambda}([1^v]) = \Prob_{\Gamma,\lambda}([1^v] \vert [0^{\partial \{v\}}]) \Prob_{\Gamma,\lambda}([0^{\partial \{v\}}]) \leq	\Prob_{\Gamma,\lambda}([1^v] \vert [0^{\partial \{v\}}]),
$$
so, considering that $\Prob_{\Gamma,\lambda}$ is a Markov random field and $\frac{\lambda}{1+\lambda}$ is increasing in $\lambda > 0$, we obtain that 
$$
\Prob_{\Gamma,\lambda}([1^v]) \leq \Prob_{\Gamma,\lambda}([1^v] \vert [0^{\partial \{v\}}]) = \frac{\lambda(v)}{1+\lambda(v)} \leq \frac{\lambda_+}{1+\lambda_+} < 1,
$$
and
$$
\Prob_{\Gamma,\lambda}([0^v]) = 1 - \Prob_{\Gamma,\lambda}([1^v]) \geq 1 -\frac{\lambda_+}{1+\lambda_+} = \frac{1}{1+\lambda_+} > 0.
$$

On the other hand, by Theorem \ref{thm21}, without loss of generality, we can suppose that $\Gamma$ is a tree rooted at $v$. Then, if $\Gamma_i$ denotes the $i$th subtree of $\Gamma$ rooted at $v_i \in \partial \{v\}$,
$$
\frac{\Prob_{\Gamma,\lambda}([1^v])}{\Prob_{\Gamma,\lambda}([0^v])}	 = \lambda(v) \cdot \prod_{i=1}^{d} \frac{1}{1+\frac{\Prob_{(\Gamma_i,\lambda)}([1^{v_i}])}{\Prob_{(\Gamma_i,\lambda)}([0^{v_i}])}} \geq	\lambda_- \cdot \prod_{i=1}^{d} \frac{1}{1+\frac{\frac{\lambda_+}{1+\lambda_+}}{\frac{1}{1+\lambda_+}}} \geq \frac{\lambda_-}{(1+\lambda_+)^{\Delta}}.
$$

Therefore, since $\Prob_{\Gamma,\lambda}([0^v]) = 1 - \Prob_{\Gamma,\lambda}([1^v])$, we have that
$$
\Prob_{\Gamma,\lambda}([1^v])	\geq \frac{\lambda_-}{\lambda_- + (1+\lambda_+)^\Delta} > 0 \quad \text{ and } \quad \Prob_{\Gamma,\lambda}([0^v]) \leq  \frac{(1+\lambda_+)^\Delta}{\lambda_- + (1+\lambda_+)^\Delta} < 1.
$$
\end{proof}

We define the {\bf critical activity function} $\lambda_c: [2,+\infty) \to (0,+\infty]$ as
$$
\lambda_c(t) := \frac{(t-1)^{(t-1)}}{(t-2)^{t}}.
$$

We have the following result.

\begin{proposition}[{\cite{1-kelly}}]
\label{prop:kelly}
For every $\Delta \in \mathbb{N}$, the hardcore model $(\mathbb{T}_{\Delta}, \lambda_0)$ exhibits WSM if and only if $\lambda_0 \leq \lambda_c(\Delta)$. If the inequality is strict, then $(\mathbb{T}_{\Delta}, \lambda_0)$ exhibits exponential WSM with a decay rate $\delta$ involving constants that depend on $\Delta$ and $\lambda_0$.
\end{proposition}

We summarize in the following theorem the main results from \cite{1-weitz}, that relate the correlation decay in $(\mathbb{T}_{\Delta}, \lambda_0)$ with the correlation decay in $\mathcal{H}^\Delta(\lambda_0)$. Here again, as in Theorem \ref{thm21}, the results in \cite{1-weitz} are focused on the constant activity case. However, we can also adapt the results to the non-constant case by considering that the main tool used in \cite{1-weitz} to prove them is \cite[Theorem 4.1]{1-weitz}, which is based on hardcore models with non-constant activity functions.

\begin{theorem}[{\cite[Theorem 2.3 and Theorem 2.4]{1-weitz}}]
\label{thm:wsmssm}
Fix $\Delta \in \mathbb{N}$ and $\lambda_0 > 0$. Then,
\begin{enumerate}
\item If $(\mathbb{T}_\Delta,\lambda_0)$ exhibits WSM with decay rate $\delta$, then $(\mathbb{T}_\Delta,\lambda_0)$ exhibits SSM with rate
$$\frac{(1+\lambda_0)(\lambda_0 + (1+\lambda_0)^{\Delta})}{\lambda_0}\delta;$$
\item If $(\mathbb{T}_\Delta,\lambda_0)$ exhibits SSM with decay rate $\delta$, then $(\Gamma,\lambda)$ exhibits SSM with rate $\delta$ for every $(\Gamma,\lambda) \in \mathcal{H}^\Delta(\lambda_0)$.
\end{enumerate}
\end{theorem}

Then, combining Proposition \ref{prop:kelly} and Theorem \ref{thm:wsmssm}, we have that if $\lambda_0 \leq \lambda_{c}(\Delta)$, then every hardcore model $(\Gamma,\lambda) \in \mathcal{H}^\Delta(\lambda_0)$ exhibits SSM with the same decay rate $\delta$, that would be exponential if the inequality is strict. In addition, observe that if $(\Gamma,\lambda)$ is a hardcore model such that $(T_{\mathrm{SAW}}(\Gamma,v),\overline{\lambda})$ exhibits SSM with decay rate $\delta$ for every $v \in V(\Gamma)$, then $(\Gamma,\lambda)$ exhibits SSM with decay rate $\delta$ as well. This follows from Theorem \ref{thm21}, since SSM is a property that depends on finitely supported events and the probabilities involved can be translated into probabilities defined on finite hardcore models which at the same time can be translated into events on finite subtrees of $T_{\mathrm{SAW}}(\Gamma,v)$. Considering this, we have the following theorem, which can be understood as a generalization of Theorem \ref{thm21} to the infinite setting.

\begin{theorem}
\label{thm:tsaw}
Given a hardcore model $(\Gamma,\lambda)$ and $v \in V(\Gamma)$ such that $(T_{\mathrm{SAW}}(\Gamma,v),\overline{\lambda})$ exhibits SSM, then for every $x \in X(\Gamma)$ and $U \subseteq V(\Gamma)$,
$$
\Prob_{\Gamma,\lambda}([0^v] \vert [x_U]) = \Prob_{T_{\mathrm{SAW}}(\Gamma,v),\overline{\lambda}}([0^\rho] \vert [\overline{x_U}]).
$$
\end{theorem}

\begin{proof}
Assume that $(T_{\mathrm{SAW}}(\Gamma,v),\overline{\lambda})$ exhibits SSM with decay rate $\delta$. Then, for every $\ell \in \mathbb{N}$,
\begin{align*}
\Prob_{\Gamma,\lambda}([0^v] \vert [x_U])	&	=	\sum_{w \in \{0,1\}^{\partial B_\Gamma(v,\ell) \setminus U}} \Prob_{\Gamma,\lambda}([0^v] \vert [x_Uw])\Prob_{\Gamma,\lambda}([w] \vert [x_U])	\\
									&	\leq	\sum_{w \in \{0,1\}^{\partial B_\Gamma(v,\ell) \setminus U}} (\Prob_{\Gamma,\lambda}([0^v] \vert [x_U0^{\partial B_\Gamma(v,\ell) \setminus U}]) + \delta(\ell))\Prob_{\Gamma,\lambda}([w] \vert [x_U])	\\
									&	=	\Prob_{\Gamma,\lambda}([0^v] \vert [x_U0^{\partial B(v,\ell) \setminus U}]) + \delta(\ell)
\end{align*}
and, similarly,
\begin{align*}
\Prob_{\Gamma,\lambda}([0^v] \vert [x_U])	&	\geq	\Prob_{\Gamma,\lambda}([0^v] \vert [x_U0^{\partial B_\Gamma(v,\ell) \setminus U}]) - \delta(\ell).
\end{align*}

Therefore, since $\lim_{\ell \to \infty} \delta(\ell) = 0$,
$$
\Prob_{\Gamma,\lambda}([0^v] \vert [x_U]) = \lim_{\ell \to \infty} \Prob_{\Gamma,\lambda}([0^v] \vert [x_U0^{\partial B_\Gamma(v,\ell) \setminus U}]),
$$
and, by the same argument,
$$
\Prob_{(T_{\mathrm{SAW}}(\Gamma,v),\overline{\lambda})}([0^\rho] \vert [\overline{x_U}]) = \lim_{\ell \to \infty} \Prob_{(T_{\mathrm{SAW}}(\Gamma,v),\overline{\lambda})}([0^\rho] \vert [\overline{x_U}0^{\partial B_{T_{\mathrm{SAW}}(\Gamma,v)}(\rho,\ell) \setminus W(U)}]).
$$

Considering this, the Markov random field property, and Proposition \ref{prop:GKssm}, we have that
\begin{align*}
\Prob_{\Gamma,\lambda}([0^v] \vert [x_U])	&	=	\lim_{\ell \to \infty} \Prob_{\Gamma,\lambda}([0^v] \vert [x_U0^{\partial B_\Gamma(v,\ell) \setminus U}])	\\
									&	=	\lim_{\ell \to \infty} \Prob_{\Gamma \cap B(v,\ell),\lambda}([0^v] \vert [x_{U \cap B_\Gamma(v,\ell)}])	\\
									&	= 	\lim_{\ell \to \infty} \Prob_{T_{\mathrm{SAW}}(\Gamma \cap B_{T_{\mathrm{SAW}}(\Gamma,v)}(v,\ell),v),\overline{\lambda}}([0^\rho] \vert [\overline{x_{U \cap B_\Gamma(v,\ell)}}])	\\
									&	= 	\lim_{\ell \to \infty} \Prob_{T_{\mathrm{SAW}}(\Gamma,v),\overline{\lambda}}([0^\rho] \vert [\overline{x_U}0^{\partial B_{T_{\mathrm{SAW}}(\Gamma,v)}(\rho,\ell) \setminus W(U)}])		\\
									&	=	\Prob_{T_{\mathrm{SAW}}(\Gamma,v),\overline{\lambda}}([0^\rho] \vert [\overline{x_U}]).
\end{align*}
\end{proof}

Notice that Theorem \ref{thm:tsaw} requires that $(T_{\mathrm{SAW}}(\Gamma,v),\overline{\lambda})$ exhibits SSM rather than the graph $(\Gamma,\lambda)$, since SSM on $T_{\mathrm{SAW}}(\Gamma,v)$ may be a stronger condition than SSM on $\Gamma$. A key fact is that if $(\mathbb{T}_\Delta, \lambda_0)$ exhibits SSM, then $(T, \lambda_0)$ exhibits SSM for every subtree $T$ of $\mathbb{T}_\Delta$. Then, since for every $\Gamma$ with $\Delta(\Gamma) \leq \Delta$, we have that $T_{\mathrm{SAW}}(\Gamma,v)$ is a subtree of $\mathbb{T}_\Delta$, it follows that $(T_{\mathrm{SAW}}(\Gamma,v),\overline{\lambda})$, and therefore $(\Gamma,\lambda_0)$, exhibit SSM. Considering this, we have the following corollary.

\begin{corollary}
Fix $\Delta \in \mathbb{N}$. Then, every $(\Gamma,\lambda) \in \mathcal{H}^\Delta(\lambda_c(\Delta))$ exhibits SSM and for every $v \in V(\Gamma)$, $x \in X(\Gamma$), and $U \subseteq V(\Gamma)$,
$$
\Prob_{\Gamma,\lambda}([0^v] \vert [x_U]) = \Prob_{T_{\mathrm{SAW}}(\Gamma,v),\overline{\lambda}}([0^\rho] \vert [\overline{x_U}]).
$$
\end{corollary}

Since we are ultimately interested in studying the interplay between the SSM property on $T_{\mathrm{SAW}}(\Gamma,v)$ and $\Gamma$, we may wonder whether is really necessary to have control over the full $\Delta$-regular tree $\mathbb{T}_\Delta$. In \cite{1-sinclair}, a refinement of this fact was proved by considering the \emph{connective constant} of the graphs involved.

\subsection{Connective constant}

Given a graph $\Gamma$, a vertex $v$, and an integer $k$, let $N_\Gamma(v,k)$ denote the \emph{number of self-avoiding walks} in $\Gamma$ of length $k$ starting from $v$. Following \cite{1-sinclair}, we consider the {\bf connective constant} $\mu(\mathcal{G})$ of a family of finite graphs $\mathcal{G}$. Here, $\mu(\mathcal{G})$ is defined as the infimum over all $\mu > 0$ for which there exist $a,c > 0$ such that for any $\Gamma \in \mathcal{G}$ and any $v \in V(\Gamma)$, it holds that $\sum^\ell_{k=1} N_\Gamma(v,k) \leq c\mu^\ell$ for all $\ell \geq a\log\lvert V(\Gamma) \rvert$. This definition extends the more usual definition of connective constant for a single infinite almost transitive graph $\Gamma$, which is given by
$$
\mu(\Gamma) := \max_{v \in V(\Gamma)} \lim_{\ell \to \infty} N_\Gamma(v,\ell)^{1/\ell}.
$$

Indeed, if $\Gamma$ is almost transitive, then $\mu(\Gamma) = \mu(\mathcal{G}(\Gamma))$, where $\mathcal{G}(\Gamma)$ denotes the family of finite subgraphs of $\Gamma$. Notice that $\mu(\Gamma)$ exists due to Fekete's lemma and that, if $\Gamma$ is connected, then $\mu(\Gamma) = \lim_{\ell \to \infty} N_\Gamma(v,\ell)^{1/\ell}$ for arbitrary $v$. Roughly, the connective constant measures the growth rate of the number of self-avoiding walks according to their length or, equivalently, the \emph{branching} of $T_{\mathrm{SAW}}(\Gamma,v)$. In general, it is not an easy task to compute $\mu(\Gamma)$ (e.g., see \cite{1-duminil}).

Considering this, we extend the definition of strong spatial mixing to families of graphs as follows: Given a family of graphs $\mathcal{G}$ and a family of activity functions $\Lambda = \{\lambda^\Gamma\}_{\Gamma \in \mathcal{G}}$ with $\lambda^\Gamma: V(\Gamma) \to  \mathbb{R}_{>0}$, we say that $(\mathcal{G}, \Lambda)$ satisfies strong spatial mixing if there exists a decay rate function $\delta: \mathbb{N} \to  \mathbb{R}_{\geq 0}$ such that $\lim_{\ell \to \infty} \delta(\ell) = 0$ and for all $\Gamma \in \mathcal{G}$, for all $U \Subset V(\Gamma)$, $v \in U$, and $y,z \in X(\Gamma)$,
$$
\lvert \pi^y_{\Gamma, U}([0^v]) - \pi^z_{\Gamma, U}([0^v]) \rvert \leq \delta(\dist_\Gamma(v,D_U(y,z))),
$$
where $\pi^y_{\Gamma, U}$ denotes the specification element corresponding to the hardcore model $(\Gamma, \lambda^\Gamma)$. We translate into this language the following result from \cite{1-sinclair}.

\begin{theorem}[\cite{1-sinclair}]
\label{thm:sinclair}
Let $\mathcal{G}$ be a family of almost transitive locally finite graphs and $\Lambda = \{\lambda^{\Gamma}\}_{\Gamma \in \mathcal{G}}$ a set of activity functions such that
$$
\sup_{\Gamma \in \mathcal{G}} \lambda^\Gamma_+ < \lambda_{c}(\mu(\mathcal{G})+1).
$$ 

Then, $(\mathcal{G}, \Lambda)$ exhibits exponential SSM.
\end{theorem}

Notice that if a graph has maximum degree $\Delta$, then $\mu(\Gamma) \leq \Delta-1$. In addition, observe that $N_\Gamma(v,\ell) = N_{T_{\mathrm{SAW}}(\Gamma,v)}(\rho,\ell)$. We have the following corollary.

\begin{corollary}
If $(\Gamma,\lambda)$ is a hardcore model such that
$$
\lambda_+ < \lambda_c(\mu(\Gamma)+1),
$$
then $(T_{\mathrm{SAW}}(\Gamma,v),\overline{\lambda})$ exhibits (exponential) SSM for every $v \in V(\Gamma)$. In particular, $(\Gamma,\lambda)$ exhibits (exponential) SSM and for every $v \in V(\Gamma)$, $x \in X(\Gamma$), and $U \subseteq V(\Gamma)$,
$$
\Prob_{\Gamma,\lambda}([0^v] \vert [x_U]) = \Prob_{T_{\mathrm{SAW}}(\Gamma,v),\overline{\lambda}}([0^\rho] \vert [\overline{x_U}]).
$$
\end{corollary}

\section{Orders}
\label{sec6}

We have already explored the main combinatorial and measure-theoretical tools that we require to establish the main results. In this section, we present some concepts of a more group-theoretical nature, namely, our ability to order a given group.

\subsection{Orderable groups}

Let $\prec$ be a strict total order on $G$. We say that $\prec$ is an {\bf invariant (right) order} if, for all $h_1,h_2,g \in G$,
$$
h_1 \prec h_2 \implies h_1g \prec h_2g.
$$

We call the pair $(G,\prec)$ a {\bf (linearly) ordered group}. The associated {\bf algebraic past} of an ordered group $(G,\prec)$  is the set $\Phi_\prec := \{g \in G: g \prec 1_G\}$ and it is a semigroup which satisfies that
$$
G = \Phi_\prec \sqcup \{\iden\} \sqcup \Phi_\prec^{-1}.
$$

Notice that $h \prec g \iff hg^{-1} \in \Phi_\prec$, so $\Phi_\prec$ fully determines $\prec$ and vice versa.

The class of orderable groups contains all torsion-free abelian groups and it is closed under the operations of taking subgroups, and forming extensions, arbitrary direct products, and directed unions.

A group $G$ is called {\bf virtually orderable (resp. ordered)} if there exists an orderable (resp. ordered) subgroup $H \leq G$ of finite index $[G:H]$. Notice that if $G \acts \Gamma$ is almost transitive and free with fundamental domain $U_0$, then $H \acts \Gamma$ is also almost transitive and free with fundamental domain $KU_0$, where $K \in \F(G)$ is any finite set of representatives. In particular, $\lvert \Gamma/H \rvert = \lvert \Gamma/G \rvert[G:H]$. For this reason, since we are interested in almost transitive actions, there is no loss of generality if, given a virtually orderable group, we assume that it is just orderable: the free energies $f_G(\Gamma,\lambda)$ and $f_H(\Gamma,\lambda)$ will be equal and the only effect of passing to a finite-index subgroup will be that the size of the fundamental domain of the action is multiplied by a constant factor (i.e., the index of the subgroup $[G:H]$). This last point is relevant, since the size of the fundamental domain will play a role later when measuring the computational complexity of some problems.

Given a finitely generated group $G$, a generating set $S$, and its corresponding Cayley graph $\Gamma = \Cay(G,S)$, we define the volume growth function as $g_\Gamma(n) = \lvert B_\Gamma(1_G,n) \rvert$. We say that $G$ has {\bf polynomial growth} if $g_\Gamma(n) \leq p(n)$ for some polynomial $p$. It is well-known that groups with polynomial growth are amenable and a classic result due to Gromov asserts that they are virtually nilpotent \cite{1-gromov}. Without further detail, from Schreier's Lemma, it is also well-known that finite index subgroups of finitely generated groups are also finitely generated \cite[Proposition 4.2]{1-lyndon} and finitely generated nilpotent groups have a torsion-free nilpotent subgroup with finite index \cite[Proposition 2]{1-segal}. From this, and since torsion-free nilpotent groups are orderable \cite[p.37]{1-mura}, it follows that any finitely generated group $G$ with polynomial growth is amenable and virtually orderable. In particular, all our results that apply to amenable and virtually orderable groups will hold for groups of polynomial growth, but they will also hold in groups of super-polynomial---namely, exponential---growth. This includes solvable groups that are not virtually nilpotent \cite{1-milnor} and, more concretely, cases like the Baumslag-Solitar groups $\mathrm{BS}(1,n)$, that can also be ordered. On the other hand, not every amenable group is virtually orderable; for example, the direct sum of countably many non-trivial finite groups always results in a countable group that is amenable but not virtually orderable. In order to address these cases, we introduce a randomized generalization of invariant orders.

\subsection{Random orders}

Consider now the set of relations $\{0,1\}^{G \times G}$ endowed with the product topology and the closed subset $\mathrm{Ord}(G)$ of strict total orders $\prec$ on $G$. We will consider the action $G \acts \mathrm{Ord}(G)$ given by
$$
h_1(g ~\cdot \prec)h_2 \iff (h_1g)\prec(h_2g)
$$
for $h_1,h_2,g \in G$ and $\prec \in \mathrm{Ord}(G)$. An {\bf invariant random order} on $G$ is a $G$-invariant Borel probability measure on $\mathrm{Ord}(G)$. Notice that a fixed point for the action $G \acts \mathrm{Ord}(G)$ corresponds to a (deterministic) invariant order on $G$. The space of invariant random orders will be denoted by $\mathcal{M}_G(\mathrm{Ord}(G))$.

Invariant random orders were introduced in \cite{1-alpeev} in order to answer problems about predictability in topological dynamics through what they called the \emph{Kieffer-Pinsker formula} for the Kolmogorov-Sinai entropy of a group action.

Now, as in the deterministic case, we can also define a notion of past for the group. An {\bf invariant random past} on $G$ is a random function $\tilde{\Phi}: G \to \{0,1\}^G$ or, equivalently, a Borel probability measure on $(\{0,1\}^G)^G$ that satisfies, for almost every instance of $\tilde{\Phi}$, the following properties:
\begin{enumerate}
\item for all $g \in G$, the condition $g \notin \tilde{\Phi}(g)$ holds;
\item for all $g,h \in G$, if $g \in \tilde{\Phi}(h)$, then $\tilde{\Phi}(g) \subseteq \tilde{\Phi}(h)$;
\item if $g \neq h$, then either $g \in \tilde{\Phi}(h)$ or $h \in \tilde{\Phi}(g)$; and
\item for all $g \in G$, the random subsets $\tilde{\Phi}(g)$ and $\tilde{\Phi}(\iden)g$ have the same distribution.
\end{enumerate}

Notice that if $\prec$ is an invariant random order, then the random function $g \mapsto \{h \in G: h \prec g\}$ defines an invariant random past.

In contrast to deterministic invariant orders, every countable group $G$ admits at least one invariant random total order. Namely, consider the random process $(\chi_g)_{g \in G}$ of independent random variables such that each $\chi_g$ has uniform distribution on $[0,1]$. This process is invariant and each realization of it induces an order on $G$ almost surely. We call such order the \emph{uniform random order}.

\section{Counting}
\label{sec7}

From now on, given $(\Gamma,\lambda) \in \mathcal{H}_G$, we always assume that there is some (or any) fixed fundamental domain $U_0$ for $G \acts \Gamma$ and we introduce the auxiliary function $\phi_\lambda: X(\Gamma) \to \mathbb{R}$ given by
$$
\phi_\lambda(x) = \frac{1}{\lvert \Gamma / G \rvert}\sum_{v \in U_0} x(v)\log\lambda(v).
$$

\subsection{A pointwise Shannon-McMillan-Breiman type theorem}

The next theorem establishes a pointwise Shannon-McMillan-Breiman type theorem for Gibbs measures (related results can be found in \cite{1-gurevich} and \cite{2-briceno}). In order to prove it we use the Pointwise Ergodic Theorem \cite{1-lindenstrauss}, which requires F{\o}lner sequence $\{F_n\}_n$ to be \emph{tempered}, a technical condition that is satisfied by every F{\o}lner sequence up to a subsequence and that we will assume without further detail. 

\begin{theorem}
\label{thm:smb}
Let $G$ be a countable amenable group. For every $(\Gamma,\lambda) \in \mathcal{H}_G$ and every $\Prob \in \mathcal{M}_{\mathrm{Gibbs}}(\Gamma,\lambda)$,
$$
\lim_n \left[ - \frac{1}{\lvert F_nU_0 \rvert} \log \Prob([x_{F_nU_0}]) \right] = -\int{\phi_\lambda} d\mathbb{Q} + f_G(\Gamma,\lambda) \quad \mathbb{Q}(x)\text{-a.s. in } x,
$$
for any tempered F{\o}lner sequence $\{F_n\}_n$ and any $\mathbb{Q} \in \mathcal{M}_G^{\mathrm{erg}}(X(\Gamma))$.
\end{theorem}

\begin{proof}
Consider the sets $U_n = F_nU_0$ and $M_n = U_n \cup \partial U_n$. Notice that, by amenability, $\lim_{n \to \infty} \frac{\lvert M_n \rvert}{\lvert U_n \rvert} = 1$. Indeed, define $K = \{g \in G: \dist_\Gamma(U_0,gU_0) \leq 1\}$. Then, $1_G \in K$ and $U_0 \cup \partial U_0 \subseteq KU_0$. Since $\Gamma$ is locally finite and the action is free, $K$ is finite. In addition, $U_n \cup \partial U_n \subseteq F_nKU_0$. Therefore, by amenability,
$$
1 \leq \lim_n \frac{\lvert U_n \cup \partial U_n \rvert}{\lvert U_n \rvert} \leq \lim_n \frac{\lvert F_nKU_0 \rvert}{\lvert F_nU_0 \rvert} = 1,
$$
so $\lim_n \frac{\lvert \partial U_n \rvert}{\lvert U_n \rvert} = 0$. Fix independent sets $x \in X(\Gamma[U_n])$, $z_1,z_2 \in X(\Gamma[M_n \setminus U_n])$, and $y \in X(\Gamma)$ such that $xz_iy \in X(\Gamma)$ for $i = 1,2$. Then,
\begin{align*}
\frac{\pi^y_{M_n}(xz_1)}{\pi^y_{M_n}(xz_2)}	&	=	\frac{\weight^y_\lambda(xz_1,M_n) Z^y_\Gamma(M_n,\lambda)^{-1}}{\weight^y_\lambda(xz_2,M_n) Z^y_\Gamma(M_n,\lambda)^{-1}}	\\
									&	=	\frac{\weight_\lambda(xz_1,M_n)1_{[y_{M_n^c}]}(xz_1)}{\weight_\lambda(xz_2,M_n) 1_{[y_{M_n^c}]}(xz_2)}	\\
									&	=	\frac{\prod_{v \in M_n \setminus U_n} \lambda(v)^{z_1(v)}}{\prod_{v \in M_n \setminus U_n} \lambda(v)^{z_2(v)}} \\
									&	\leq	\prod_{v \in M_n \setminus U_n} \lambda_+^{z_1(v)}\lambda_-^{-z_2(v)}.
\end{align*}

Taking $z_2 = 0^\Gamma$ and adding over all possible $z_1$, we obtain that
$$
1 \leq \frac{\pi^y_{M_n}([x_{U_n}])}{\pi^y_{M_n}(x_{U_n}0^{M_n \setminus U_n})} = \sum_{\substack{z \in \{0,1\}^{M_n \setminus U_n}:\\ x_{U_n}z \in X(\Gamma[M_n])}}\frac{\pi^y_{M_n}(x_{U_n}z)}{\pi^y_{M_n}(x_{U_n}0^{M_n \setminus U_n})} \leq  (2\max\{1,\lambda_+\})^{\lvert M_n \setminus U_n \rvert}.
$$

On the other hand, we have that
$$
\frac{\pi^y_{M_n}([x_{U_n}0^{M_n \setminus U_n}])}{\pi^{0^\Gamma}_{M_n}([x_{U_n}0^{M_n \setminus U_n}])} = \frac{\weight^y_\lambda(x_{U_n}0^{M_n \setminus U_n},M_n) Z^y_\Gamma(M_n,\lambda)^{-1}}{\weight^{0^\Gamma}_\lambda(x_{U_n}0^{M_n \setminus U_n},M_n) Z^{0^\Gamma}_\Gamma(M_n,\lambda)^{-1}} = \frac{Z^y_\Gamma(M_n,\lambda)}{Z_\Gamma(M_n,\lambda)},
$$
since
$$
\weight^y_\lambda(x_{U_n}0^{M_n \setminus U_n},M_n) = \weight^{0^\Gamma}_\lambda(x_{U_n}0^{M_n \setminus U_n},M_n) = \weight_\lambda(x_{U_n},U_n)
$$ 
and $Z^{0^\Gamma}_\Gamma(M_n,\lambda) = Z_\Gamma(M_n,\lambda)$. In addition,
$$
1 \leq \frac{Z^y_\Gamma(M_n,\lambda)}{Z_\Gamma(M_n,\lambda)} \leq (2\max\{1,\lambda_+\})^{\lvert M_n \setminus U_n \rvert},
$$
since
\begin{align*}
Z^y_\Gamma(M_n,\lambda)	&	\leq	Z_\Gamma(M_n,\lambda)	\\
						&	=	\sum_{x \in X(\Gamma[U_n])}\sum_{\substack{z \in X(\Gamma[M_n \setminus U_n]):\\ xz \in X(\Gamma[M_n])}} \prod_{v \in U_n} \lambda(v)^{x(v)}\prod_{v \in M_n \setminus U_n} \lambda(v)^{z(v)}	\\
						&	\leq	\sum_{x \in X(\Gamma[U_n])}\sum_{\substack{z \in X(\Gamma[M_n \setminus U_n]):\\ xz \in X(\Gamma[M_n])}} \prod_{v \in U_n} \lambda(v)^{x(v)} \max\{1,\lambda_+\}^{\lvert M_n \setminus U_n \rvert}	\\
						&	\leq	2^{\lvert M_n \setminus U_n \rvert} \max\{1,\lambda_+\}^{\lvert M_n \setminus U_n \rvert}\sum_{x \in X(\Gamma[U_n])} \prod_{v \in U_n} \lambda(v)^{x(v)}	\\
						&	=	(2\max\{1,\lambda_+\})^{\lvert M_n \setminus U_n \rvert} Z_\Gamma(U_n,\lambda)	\\
						&	\leq	(2\max\{1,\lambda_+\})^{\lvert M_n \setminus U_n \rvert} Z^y_\Gamma(M_n,\lambda). 
\end{align*}

Therefore,
\begin{align*}
1	&	\leq	\frac{\pi^y_{M_n}([x_{U_n}])}{\pi^{0^\Gamma}_{M_n}([x_{U_n}0^{M_n \setminus U_n}])}	\\
	&	=	\frac{\pi^y_{M_n}([x_{U_n}])}{\pi^y_{M_n}([x_{U_n}0^{M_n \setminus U_n}])}\frac{\pi^y_{M_n}([x_{U_n}0^{M_n \setminus U_n}])}{\pi^{0^\Gamma}_{M_n}([x_{U_n}0^{M_n \setminus U_n}])}	\\
	&	=	\frac{\pi^y_{M_n}([x_{U_n}])}{\pi^y_{M_n}([x_{U_n}0^{M_n \setminus U_n}])}\frac{Z^y_\Gamma(M_n,\lambda)}{Z_\Gamma(M_n,\lambda)}	\\
	&	\leq	(2\max\{1,\lambda_+\})^{2\lvert M_n \setminus U_n \rvert}.
\end{align*}

In particular, since $\Prob([x_{F_nU_0}]) = \mathbb{E}_{\Prob}(\Prob([x_{U_n}] \vert \mathcal{B}_{M_n^c})(y)) = \mathbb{E}_{\Prob}(\pi^y_{M_n}([x_{U_n}]))$, we have that
$$
1 \leq \frac{\Prob([x_{F_nU_0}])}{\pi^{0^\Gamma}_{M_n}([x_{U_n}0^{M_n \setminus U_n}])} \leq (2\max\{1,\lambda_+\})^{2\lvert M_n \setminus U_n \rvert},
$$
so
$$
\lvert \log\Prob([x_{F_nU_0}]) - \log\pi^{0^\Gamma}_{M_n}(x_{U_n}0^{M_n \setminus U_n}) \rvert \leq 2\lvert M_n \setminus U_n \rvert\log(2\max\{1,\lambda_+\}).
$$

Now, since $\weight^{0^\Gamma}_\lambda(x_{M_n},M_n) = \weight_\lambda(x_{M_n},M_n)$ for every $x$, we have that
$$
\pi^{0^\Gamma}_{M_n}(x_{U_n}0^{M_n \setminus U_n}) = \weight_\lambda(x_{U_n}0^{M_n \setminus U_n},M_n) Z_\Gamma(M_n,\lambda)^{-1} = \weight_\lambda(x_{U_n},U_n) Z_\Gamma(M_n,\lambda)^{-1}.
$$

Therefore,
$$
\lvert \log\Prob([x_{F_nU_0}]) - (\log\weight_\lambda(x_{U_n},U_n) - \log Z_\Gamma(M_n,\lambda)) \rvert \leq 2\lvert M_n \setminus U_n \rvert\log(2\max\{1,\lambda_+\}),
$$
so
\begin{align*}
	&	\lim_{n \to \infty} \lvert  -\frac{\log\Prob([x_{F_nU_0}])}{\lvert U_n \rvert} + \left(\frac{\log\weight_\lambda(x_{U_n},U_n)}{\lvert U_n \rvert} - \frac{\log Z_\Gamma(M_n,\lambda)}{\lvert U_n \rvert}\right)\rvert	\\
\leq	&	\lim_n 2\frac{\lvert M_n \setminus U_n \rvert}{\lvert U_n \rvert}\log(2\max\{1,\lambda_+\}) = 0,
\end{align*}
and we conclude that
\begin{align*}
-\lim_n \frac{\log\Prob([x_{F_nU_0}])}{\lvert F_nU_0 \rvert}	&	=	\lim_n \frac{-\log\weight_\lambda(x_{U_n},U_n)}{\lvert U_n \rvert} + \frac{\log Z_\Gamma(M_n,\lambda)}{\lvert U_n \rvert}	\\
										&	=	-\lim_n \left(\frac{1}{\lvert F_n \rvert} \sum_{g \in F_n} \frac{1}{\lvert U_0 \rvert} \sum_{v \in gU_0} x(v)\log\lambda(v)\right) + f_G(\Gamma,\lambda),
\end{align*}
where we have used that $Z_\Gamma(U_n,\lambda) \leq Z_\Gamma(M_n,\lambda) \leq (2\max\{1,\lambda_+\})^{\lvert M_n \setminus U_n \rvert}Z_\Gamma(U_n,\lambda)$. Finally, notice that
$$
\frac{1}{\lvert F_n \rvert} \sum_{g \in F_n}  \frac{1}{\lvert U_0 \rvert} \sum_{v \in gU_0} x(v)\log\lambda(v) = \frac{1}{\lvert F_n \rvert} \sum_{g \in F_n} \phi_\lambda(g \cdot x)
$$
and by the Pointwise Ergodic Theorem, we obtain that
$$
\lim_n \frac{1}{\lvert F_n \rvert} \sum_{g \in F_n} \phi_\lambda(g \cdot x) = \int{\phi_\lambda}d\mathbb{Q} \quad \mathbb{Q}\text{-a.s.}, 
$$
so
$$
-\lim_n \frac{\log\Prob([x_{F_nU_0}])}{\lvert F_nU_0 \rvert} = - \int{\phi_\lambda}d\mathbb{Q} + f_G(\Gamma,\lambda),
$$
and we conclude the proof.
\end{proof}

We have the following lemma.

\begin{lemma}
\label{lem:bound}
Given $\Delta \in \mathbb{N}$ and $(\Gamma,\lambda) \in \mathcal{H}^\Delta$, there exists a constant $C = C(\Delta,\lambda_-,\lambda_+) > 0$ such that for every $U \Subset V(\Gamma)$, $x \in X(\Gamma)$, and $v \in U$,
$$
\pi^x_U([x_v]) \geq C.
$$
\end{lemma}

\begin{proof}
Fix $(\Gamma,\lambda) \in \mathcal{H}_G^\Delta$, $U \Subset V(\Gamma)$, $x \in X(\Gamma)$, and $v \in U$. Notice that if $U^c \cap \partial\{v\} \neq \emptyset$ and $u \in U^c \cap \partial\{v\}$ is such that $x(u) = 1$, then necessarily $x_v = 0^v$ and $\pi^x_U([x_v]) = 1$. On the other hand, if $U \cap \partial\{v\} = \emptyset$, then $\pi^x_U([x_v]) = \Prob_{\Gamma[U'],\lambda}([x_v])$ for $U' = U \setminus \{u \in U: x(u') = 1 \text{ for some } u' \in U^c \cap \partial\{u\}\}$, so, by Lemma \ref{lem:probbound},
$$
\pi^x_U([x_v]) = \Prob_{\Gamma[U'],\lambda}([x_v]) \geq \min\left\{\frac{1}{1+\lambda_+}, \frac{\lambda_-}{\lambda_- + (1+\lambda_+)^\Delta}\right\} \geq \frac{\min\{1,\lambda_-\}}{\lambda_- + (1+\lambda_+)^\Delta} > 0,
$$
Therefore, by taking $C = \min\left\{1, \frac{\min\{1,\lambda_-\}}{\lambda_- + (1+\lambda_+)^\Delta}\right\} = \frac{\min\{1,\lambda_-\}}{\lambda_- + (1+\lambda_+)^\Delta}$, we conclude.
\end{proof}

\subsection{A randomized sequential cavity method}

Suppose now that $(\Gamma,\lambda) \in \mathcal{H}_G$ is such that the Gibbs $(\Gamma,\lambda)$-specification satisfies SSM and let $\Prob$ be the unique Gibbs measure. Considering this, we define the function $I_{\Prob}: X(\Gamma) \times (\{0,1\}^G)^G \to \mathbb{R}$ given by
$$
I_{\Prob}(x,\Phi) := \limsup_{n \to \infty} I_{\Prob, n}(x,\Phi),
$$
where
$$
I_{\Prob,n}(x,\Phi) := -\frac{1}{\lvert \Gamma/G \rvert}\log\Prob([x_{U_0}] \vert [x_{(\Phi(1_G) \cap F_n)U_0}])
$$
and $\{F_n\}_n$ is any exhaustion of $G$ (not necessarily F{\o}lner). Here, $\Phi \in (\{0,1\}^G)^G$ should be understood as an instance of a random invariant past and $I_{\Prob,n}$ as a rudimentary information function. In this sense, notice that $I_{\Prob,n}(x,\Phi)$ only depends on the values of $x$ restricted to $\Phi(1_G) \in \{0,1\}^G$, so $\Phi$ carries redundant information. This redundancy will play a role in the next results when taking averages.

\begin{lemma}
If $(\Gamma,\lambda) \in \mathcal{H}_G$ is such that the Gibbs $(\Gamma,\lambda)$-specification satisfies SSM and $\Prob$ is the unique Gibbs measure, then the function $I_{\Prob}$ is measurable, non-negative, defined everywhere, and bounded. Moreover,
$$
 I_{\Prob}(x,\Phi) = -\frac{1}{\lvert \Gamma/G \rvert}\log\Prob([x_{U_0}] \vert [x_{\Phi(1_G)U_0}]).
$$
\end{lemma}

\begin{proof}
Since $\Prob([x_{U_0}] \vert [x_{(\Phi(1_G) \cap F_n)U_0}])$ depends on finitely many coordinates in both $X(\Gamma)$ and $(\{0,1\}^G)^G$, $0 < \Prob([x_{U_0}] \vert [x_{(\Phi(1_G) \cap F_n)U_0}]) < 1$, and $-\log(\cdot)$ is a continuous function, $I_{\Prob,n}$ is measurable and since $I_{\Prob}$ is a limit superior, it is measurable as well.

By SSM and Proposition \ref{prop:GKssm}, $$\lim_{n \to \infty} \Prob([x_{U_0}] \vert [x_{(\Phi(1_G) \cap F_n)U_0}]) = \Prob([x_{U_0}] \vert [x_{\Phi(1_G)U_0}])$$ is always a well-defined limit. By Lemma \ref{lem:bound}, and since $(\Gamma,\lambda) \in \mathcal{H}^\Delta_G$ for some $\Delta$, there exists a constant $C > 0$ such that for every $v \in V$, $U \Subset V \setminus \{v\}$, and $x \in X(\Gamma)$,
$$
\pi^x_U([x_v]) \geq C.
$$

Now, combined with the SSM property, this implies that for every $v \in V$, $U \subseteq V$, and $x \in X(\Gamma)$, we have that
$$
\Prob([x_v] \vert [x_U]) \geq C.
$$ 

Indeed, if $v \in U$, this is direct, since $\Prob([x_v] \vert [x_U]) = 1$. On the other hand, if  $v \notin U$, by SSM,
$$
\Prob([x_v] \vert [x_U]) = \lim_{\ell \to \infty} \Prob([x_v] \vert [x_{U \cup B_\Gamma(v,\ell)^c}]) = \lim_{\ell \to \infty} \pi_{U \cup B_\Gamma(v,\ell)^c}^x([x_v]) \geq \lim_{\ell \to \infty} C = C.
$$

Therefore, by conditioning and iterating, we obtain that
$$
1 \geq \Prob([x_{U_0}] \vert [x_{\Phi(1_G) U_0}]) = \prod_{i=1}^{\lvert U_0 \rvert} \Prob([x_{v_i}] \vert [x_{\Phi U_0 \cup \{v_1,\dots,v_{i-1}\}}]) \geq \prod_{i=1}^{\lvert U_0 \rvert} C = C^{\lvert \Gamma / G \rvert},
$$
so
$$
0 \leq I_{\Prob}(x,\Phi) = -\log\Prob([x_{U_0}] \vert [x_{\Phi(1_G)U_0}]) \leq -\lvert \Gamma / G \rvert\log C < +\infty,
$$
i.e., $I_{\Prob}(x,\Phi)$ is bounded. 
\end{proof}

Following \cite{1-alpeev}, given an invariant random past $\tilde{\Phi}: G \to \{0,1\}^G$ on $G$ with law $\tilde{\nu} \in \mathcal{M}_G((\{0,1\}^G)^G)$, we denote
$$
\mathbb{E}_{\tilde{\Phi}}f(\tilde{\Phi}) = \int{f(\tilde{\Phi})}d\tilde{\nu}(\tilde{\Phi}),
$$
for $f \in L^1(\tilde{\nu})$.  Now, since $I_{\Prob}$ is measurable, non-negative, and bounded, we have that for every $\mathbb{Q} \in \mathcal{M}_G(X(\Gamma))$, the function $I_{\Prob}$ is integrable with respect to $\mathbb{Q} \times \tilde{\nu}$ and by Tonelli's theorem, the function $\mathbb{E}_{\tilde{\Phi}}I_{\Prob}: X(\Gamma) \to \mathbb{R}$ is integrable, defined $\mathbb{Q}$-almost everywhere, and satisfies that
$$
\int{\mathbb{E}_{\tilde{\Phi}}I_{\Prob}(x,\tilde{\Phi}})d\mathbb{Q}(x) = \mathbb{E}_{\tilde{\Phi}}\int{I_{\Prob}(x,\tilde{\Phi}})d\mathbb{Q}(x).
$$

We call $\mathbb{E}_{\tilde{\Phi}}I_{\Prob}$ the {\bf random $\Prob$-information function} (with respect to $\tilde{\Phi}$).

\begin{lemma}
\label{lem:info}
For any tempered F{\o}lner sequence $\{F_n\}_n$, any invariant random past $\tilde{\Phi}$, and any $\mathbb{Q} \in \mathcal{M}_G(X(\Gamma))$,
$$
\lim_n \left[ - \frac{1}{\lvert F_nU_0 \rvert} \log \Prob([x_{F_nU_0}]) -  \frac{1}{\lvert F_n \rvert}\sum_{g \in F_n} \mathbb{E}_{\tilde{\Phi}}I_{\Prob}(g \cdot x,\tilde{\Phi}) \right] = 0
$$
for $\mathbb{Q}$-almost every $x$.
\end{lemma}

\begin{proof}
Fix a (tempered) F{\o}lner sequence $\{F_n\}_n$. By the properties of $\tilde{\Phi}$, for $\tilde{\nu}$-almost every instance $\Phi$, every $F_n$ can be ordered as $g_1, \dots,g_{\lvert F_n \rvert}$ so that $\Phi(g_i) \cap F_n = \{g_1,\dots,g_{i-1}\}$. Then,
$$
\Prob([x_{F_nU_0}]) = \prod_{g \in F_n} \Prob([x_{gU_0}] \vert [x_{(\Phi(g) \cap F_n)U_0}]).
$$

Given $\ell > 0$, let $K_\ell = \{g \in G: \dist_\Gamma(U_0,gU_0) \leq \ell\}$. If $g \in \mathrm{Int}_{K_\ell}(F_n) = \{g \in G: K_\ell g \subseteq F_n\}$, then
$$
\lvert \Prob([x_{gU_0}] \vert [x_{(\Phi(g) \cap F_n)U_0}]) - \Prob([x_{gU_0}] \vert [x_{(\Phi(g) \cap K_\ell g)U_0}]) \rvert \leq \lvert U_0 \rvert\delta(\ell)
$$
and
$$
\lvert \Prob([x_{gU_0}] \vert [x_{(\Phi(g) \cap K_\ell g)U_0}]) - \Prob([x_{gU_0}] \vert [x_{\Phi(g)U_0}]) \rvert \leq \lvert U_0 \rvert\delta(\ell),
$$
so
$$
\lvert \Prob([x_{gU_0}] \vert [x_{(\Phi(g) \cap F_n)U_0}]) - \Prob([x_{gU_0}] \vert [x_{\Phi(g)U_0}]) \rvert \leq 2\lvert U_0 \rvert\delta(\ell).
$$

On the other hand, by Lemma \ref{lem:bound} and the discussion after it, for every $g \in G$ and $U \subseteq V$, 
$$
\Prob([x_{gU_0}] \vert [x_U]) \geq C^{\lvert \Gamma / G \rvert} > 0.
$$ 

Therefore, by the Mean Value Theorem,
$$
\lvert \log\Prob([x_{gU_0}] \vert [x_{(\Phi(g) \cap F_n)U_0}]) - \log\Prob([x_{gU_0}] \vert [x_{\Phi(g)U_0}]) \rvert \leq \frac{2\lvert U_0\rvert}{C^{\lvert \Gamma / G\rvert}}\delta(\ell).
$$

Notice that
\begin{align*}
\log\Prob([x_{F_nU_0}])	&	=		\sum_{g \in \mathrm{Int}_{K_\ell}(F_n)} \log\Prob([x_{gU_0}] \vert [x_{(\Phi(g) \cap F_n)U_0}]) \\
					&	\qquad	+ \sum_{g \in F_n \setminus \mathrm{Int}_{K_\ell}(F_n)} \log\Prob([x_{gU_0}] \vert [x_{(\Phi(g) \cap F_n)U_0}]),
\end{align*}
so
\begin{align*}
\lvert \log\Prob([x_{FU_0}]) - \sum_{g \in F_n} \log\Prob([x_{gU_0}] \vert [x_{\Phi(g)U_0}])\rvert	&	\leq	\lvert \mathrm{Int}_{K_\ell}(F_n)\rvert\frac{2\lvert U_0\rvert}{C^{\lvert \Gamma / G\rvert}}\delta(\ell) \\
																	&	\qquad	+ 2\lvert F_n \setminus \mathrm{Int}_{K_\ell}(F_n)\rvert\log(C^{-\lvert \Gamma / G\rvert}).
\end{align*}

Now, given $\epsilon > 0$, there exists $\ell$ and $n_0$ such that for every $n \geq n_0$,
$$
\lvert \frac{\log\Prob([x_{F_nU_0}])}{\lvert F_n\rvert} - \frac{1}{\lvert F_n\rvert}\sum_{g \in F_n} \log\Prob([x_{gU_0}] \vert [x_{\Phi(g)U_0}]) \rvert \leq \epsilon.
$$

By $G$-invariance of $\Prob$,
\begin{align*}
\Prob([x_{gU_0}] \vert [x_{\Phi(g)U_0}])	&	=	\Prob(g^{-1} \cdot [(g \cdot x)_{U_0}] \vert [x_{\Phi(g)U_0}])		\\
								&	=	\Prob([(g \cdot x)_{U_0}] \vert g \cdot [x_{\Phi(g)U_0}])		\\
								&	=	\Prob([(g \cdot x)_{U_0}] \vert [(g \cdot x)_{\Phi(g)g^{-1}U_0}]),
\end{align*}
and combining this fact with the previous estimate, we obtain that
$$
\lvert \frac{\log\Prob([x_{F_nU_0}])}{\lvert F_n\rvert} - \frac{1}{\lvert F_n\rvert}\sum_{g \in F_n} \log\Prob([(g \cdot x)_{U_0}] \vert [(g \cdot x)_{\Phi(g)g^{-1}U_0}]) \rvert \leq	\epsilon
$$

Integrating against $\tilde{\nu}$, we obtain that, for $\mathbb{Q}$-almost every $x$,
$$
\lvert \frac{\log\Prob([x_{F_nU_0}])}{\lvert F_n\rvert} - \frac{1}{\lvert F_n \rvert}\sum_{g \in F_n} \mathbb{E}_{\tilde{\Phi}}\log\Prob([(g \cdot x)_{U_0}] \vert [(g \cdot x)_{\tilde{\Phi}(g)g^{-1}U_0}]) \rvert \leq	\epsilon
$$
and since $\tilde{\Phi}(g)$ has the same distribution as $\tilde{\Phi}(1_G)g$, we get that
\begin{align*}
\mathbb{E}_{\tilde{\Phi}}\log\Prob([(g \cdot x)_{U_0}] \vert [(g \cdot x)_{\tilde{\Phi}(g)g^{-1}U_0}])	&	=	\mathbb{E}_{\tilde{\Phi}}\log\Prob([(g \cdot x)_{U_0}] \vert [(g \cdot x)_{\tilde{\Phi}(1_G)gg^{-1}U_0}])	\\
																			&	=	\mathbb{E}_{\tilde{\Phi}}\log\Prob([(g \cdot x)_{U_0}] \vert [(g \cdot x)_{\tilde{\Phi}(1_G)U_0}])	\\
																			&	=	-\lvert \Gamma/G \rvert\mathbb{E}_{\tilde{\Phi}} I_\Prob(g \cdot x, \tilde{\Phi}),
\end{align*}
so
$$
\lvert -\frac{\log\Prob([x_{F_nU_0}])}{\lvert F_nU_0 \rvert} - \frac{1}{\lvert F_n \rvert}\sum_{g \in F_n} \mathbb{E}_{\tilde{\Phi}} I_\Prob(g \cdot x, \tilde{\Phi}) \rvert \leq \epsilon,
$$
and since $\epsilon$ is arbitrary and the limit exists $\mathbb{Q}$-almost surely, we conclude.
\end{proof}

We have the following representation theorem for free energy, which can be regarded as a randomized generalization of the results in \cite{1-gamarnik,1-marcus,2-briceno} tailored for the specific case of the hardcore model. Notice that, in contrast with \cite{1-gamarnik,1-marcus,2-briceno}, the representation holds in every amenable group and not just virtually orderable groups.

\begin{theorem}
\label{thm:rep1}
Let $(\Gamma,\lambda) \in \mathcal{H}_G$ such that the Gibbs $(\Gamma,\lambda)$-specification satisfies SSM and $\Prob$ is the unique Gibbs measure. Then,
$$
f_G(\Gamma,\lambda) = \int{\left(\mathbb{E}_{\tilde{\Phi}}I_{\Prob} + \phi_\lambda\right)}d\mathbb{Q}
$$
for any $\mathbb{Q} \in \mathcal{M}_G(X(\Gamma))$ and for any invariant random past $\tilde{\Phi}$ of $G$. In particular,
$$
f_G(\Gamma,\lambda) = \mathbb{E}_{\tilde{\Phi}} I_{\Prob}(0^\Gamma,\tilde{\Phi}).
$$
\end{theorem}

\begin{proof}
First, notice that if the statement holds for every $\mathbb{Q} \in \mathcal{M}^{\mathrm{erg}}_G(X(\Gamma))$, then it holds for every $\mathbb{Q} \in \mathcal{M}_G(X(\Gamma))$ by the Ergodic Decomposition Theorem. Then, without loss of generality, we can assume that $\mathbb{Q}$ is $G$-ergodic. Considering this, by Theorem \ref{thm:smb}, for $\mathbb{Q}$-almost every $x$,
$$
\lim_n \left[ - \frac{1}{\lvert F_nU_0 \rvert} \log \Prob([x_{F_n}]) \right] = -\int{\phi_\lambda} d\mathbb{Q} + f_G(\Gamma,\lambda).
$$

By Lemma \ref{lem:info}, for $\mathbb{Q}$-almost every $x$,
$$
\lim_n \left[ - \frac{1}{\lvert F_nU_0 \rvert} \log \Prob([x_{F_n}]) \right] = \lim_n \frac{1}{\lvert F_n \rvert}\sum_{g \in F_n} \mathbb{E}_{\tilde{\Phi}} I_\Prob(g \cdot x, \tilde{\Phi}).
$$

Therefore, for $\mathbb{Q}$-almost every $x$,
$$
f_G(\Gamma,\lambda) - \int{\phi_\lambda} d\mathbb{Q} = \lim_n \frac{1}{\lvert F_n \rvert}\sum_{g \in F_n} \mathbb{E}_{\tilde{\Phi}} I_\Prob(g \cdot x, \tilde{\Phi}).
$$

Integrating against $\mathbb{Q}$, we obtain that
\begin{align*}
\int{ \lim_n \frac{1}{\lvert F_n \rvert}\sum_{g \in F_n} \mathbb{E}_{\tilde{\Phi}}I_{\Prob}(g \cdot x,\tilde{\Phi})}d\mathbb{Q}	&	=	 \lim_n\int{ \frac{1}{\lvert F_n \rvert}\sum_{g \in F_n} \mathbb{E}_{\tilde{\Phi}}I_{\Prob}(g \cdot x,\tilde{\Phi})}d\mathbb{Q}	\\
																					&	=	 \lim_n\frac{1}{\lvert F_n \rvert}\sum_{g \in F_n}\mathbb{E}_{\tilde{\Phi}} \int{I_{\Prob}(g \cdot x,\tilde{\Phi})}d\mathbb{Q}	\\
																					&	=	 \lim_n\frac{1}{\lvert F_n \rvert}\sum_{g \in F_n}\mathbb{E}_{\tilde{\Phi}} \int{I_{\Prob}(x,\tilde{\Phi})}d\mathbb{Q}	\\
																					&	=	 \mathbb{E}_{\tilde{\Phi}} \int{I_{\Prob}(x,\tilde{\Phi})}d\mathbb{Q}	\\
																					&	=	 \int{\mathbb{E}_{\tilde{\Phi}} I_{\Prob}(x,\tilde{\Phi})}d\mathbb{Q},
\end{align*}
where the first equality is due to the Dominated Convergence Theorem, the second and last equalities are due to Tonelli's Theorem, and the third equality is due to the $G$-invariance of $\mathbb{Q}$. We conclude that
$$
f_G(\Gamma,\lambda) = \int{\left(\mathbb{E}_{\tilde{\Phi}}I_{\Prob} + \phi_\lambda\right)}d\mathbb{Q}.
$$

In particular, if $\mathbb{Q} = \delta_{0^\Gamma} \in \mathcal{M}_G(X(\Gamma))$, the Dirac measure supported on $0^\Gamma$, then
$$
f_G(\Gamma,\lambda) = \mathbb{E}_{\tilde{\Phi}} I_{\Prob}(0^\Gamma,\tilde{\Phi}) + \phi_\lambda(0^\Gamma) = \mathbb{E}_{\tilde{\Phi}} I_{\Prob}(0^\Gamma,\tilde{\Phi}).
$$
\end{proof}

\begin{figure}[ht]
\centering
\includegraphics[scale = 1.5]{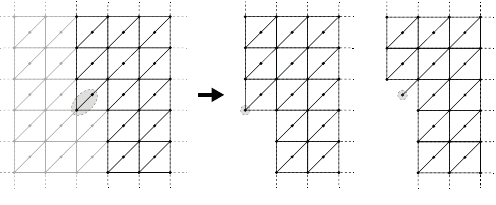}
\caption{The graph $\Gamma \setminus \Phi_\prec U_0$ and the corresponding graphs $\Gamma_i(\Phi_\prec)$ for $G = \mathbb{Z}^2$ and the lexicographic order $\prec$}
\label{fig:lex}
\end{figure}

\subsection{An arboreal representation of free energy}

The following theorem tell us that, under some special conditions, $f_G(\Gamma,\lambda)$ can be expressed using $\lvert \Gamma / G \rvert$ terms depending on the probability that the roots of some particular trees are unoccupied. Roughly, the trees involved are the trees of self-avoiding walks that are rooted at the vertices of a given fundamental domain and explore the graph $\Gamma$ without entering to the ``past graph'' induced by an invariant random past.

\begin{theorem}
\label{thm:rep2}
Let $(\Gamma,\lambda) \in \mathcal{H}_G$ such that the Gibbs $(T_{\mathrm{SAW}}(\Gamma,v),\overline{\lambda})$-specification satisfies SSM for every $v \in V(\Gamma)$ and let $v_1,\dots,v_{\lvert \Gamma / G \rvert}$ be an arbitrary ordering of a fundamental domain $U_0$. Given an invariant random past $\tilde{\Phi}$ of $G$, denote by $\Gamma_i(\tilde{\Phi})$ the random graph given by $\Gamma \setminus (\tilde{\Phi}(\iden)U_0 \cup \{v_1,\dots,v_{i-1}\})$. Then, 
$$
f_G(\Gamma,\lambda) = -\frac{1}{\lvert \Gamma/G\rvert}\sum_{i=1}^{\lvert \Gamma / G \rvert} \mathbb{E}_{\tilde{\Phi}} \log \Prob_{T_{\mathrm{SAW}}(\Gamma_i(\tilde{\Phi}),v_i),\overline{\lambda}}([0^{\rho_i}]),
$$
where $\rho_i$ denotes the root of $T_{\mathrm{SAW}(\Gamma_i(\Phi),v_i)}$. In particular, if $\prec$ is a deterministic invariant order of $G$,
$$
f_G(\Gamma,\lambda) = -\frac{1}{\lvert \Gamma/G \rvert}\sum_{i=1}^{\lvert \Gamma / G \rvert} \log \Prob_{T_{\mathrm{SAW}}(\Gamma_i(\Phi_\prec),v_i),\overline{\lambda}}([0^{\rho_i}]).
$$
\end{theorem}

\begin{proof}
By Theorem \ref{thm:rep1}, we know that
$$
f_G(\Gamma,\lambda) = \mathbb{E}_{\tilde{\Phi}}I_{\Prob_{\Gamma,\lambda}}(0^\Gamma,\tilde{\Phi}) = -\frac{1}{\lvert \Gamma/G \rvert}\mathbb{E}_{\tilde{\Phi}}\log\Prob_{\Gamma,\lambda}([0^{U_0}] \vert [0^{\tilde{\Phi}(1_G)U_0}]).
$$

By iterating conditional probabilities, linearity of expectation, and Proposition \ref{prop:GKssm} (see Figure \ref{fig:lex}), 
\begin{align*}
\mathbb{E}_{\tilde{\Phi}}\log\Prob_{\Gamma,\lambda}([0^{U_0}] \vert [0^{\tilde{\Phi}(1_G)U_0}])	&	=	\sum_{i=1}^{\lvert \Gamma / G \rvert}\mathbb{E}_{\tilde{\Phi}}\log\Prob_{\Gamma,\lambda}([0^{v_i}] \vert [0^{\tilde{\Phi}(1_G)U_0 \cup \{v_1,\dots,v_{i-1}\}}])	\\
																		&	=	\sum_{i=1}^{\lvert \Gamma / G \rvert}\mathbb{E}_{\tilde{\Phi}}\log\Prob_{\Gamma \setminus (\tilde{\Phi}(1_G)U_0 \cup \{v_1,\dots,v_{i-1}\}),\lambda}([0^{v_i}]) \\
																		&	=	\sum_{i=1}^{\lvert \Gamma / G \rvert}\mathbb{E}_{\tilde{\Phi}}\log\Prob_{\Gamma_i(\tilde{\Phi}),\lambda}([0^{v_i}]).
\end{align*}

Finally, by Theorem \ref{thm:tsaw}, we obtain
$$
\sum_{i=1}^{\lvert \Gamma / G \rvert}\mathbb{E}_{\tilde{\Phi}}\log\Prob_{\Gamma_i(\tilde{\Phi}),\lambda}([0^{v_i}]) = \sum_{i=1}^{\lvert \Gamma / G \rvert} \mathbb{E}_{\tilde{\Phi}}\log \Prob_{T_{\mathrm{SAW}}(\Gamma_i(\tilde{\Phi}),v_i),\overline{\lambda}}([0^{\rho_i}]).
$$
					
In particular, if $\prec$ is a deterministic invariant order on $G$ (see Figure \ref{fig:trees}), we have that
$$
f_G(\Gamma,\lambda) = -\frac{1}{\lvert \Gamma/G \rvert}\sum_{i=1}^{\lvert \Gamma / G \rvert} \log \Prob_{T_{\mathrm{SAW}}(\Gamma_i(\Phi_\prec),v_i),\overline{\lambda}}([0^{\rho_i}]).
$$
\end{proof}

\section{A computational phase transition in the thermodynamic limit}
\label{sec8}

Given an amenable countable group $G$, we are interested in having an algorithm to efficiently approximate $f_G(\Gamma,\lambda)$ in some uniform way over $\mathcal{H}_G$. 

Let $\mathbb{M} \subseteq \mathcal{H}_G$ be a family of hardcore models. We will say that $\mathbb{M}$ admits an {\bf additive fully polynomial-time randomized approximation scheme (additive FPRAS)} for $f_G(\Gamma,\lambda)$ if there is a probabilistic algorithm such that, given an input $(\Gamma,\lambda) \in \mathbb{M}$ and $\epsilon > 0$, outputs $\hat{f}$ with
$$
\Prob\left[\lvert f_G(\Gamma,\lambda) - \hat{f} \rvert \leq \epsilon\right] \geq \frac{3}{4},
$$
in polynomial time in $\lvert \left<\Gamma,\lambda\right> \rvert$ and $\epsilon^{-1}$, where $\lvert \left<\Gamma,\lambda\right> \rvert$ denotes the length of any reasonable representation $\left<\Gamma,\lambda\right>$ of $(\Gamma,\lambda)$. Similarly, we will say that $\mathbb{M}$ admits an {\bf additive fully polynomial-time approximation scheme (additive FPTAS)} for $f_G(\Gamma,\lambda)$ if there is a deterministic additive FPRAS with null failure probability.

 An additive FPRAS and an additive FPTAS will be what we will regard as an efficient and uniform approximation algorithm for $f_G(\Gamma,\lambda)$, random and deterministic, respectively.

\begin{figure}[ht]
\centering
\includegraphics[scale = 1.5]{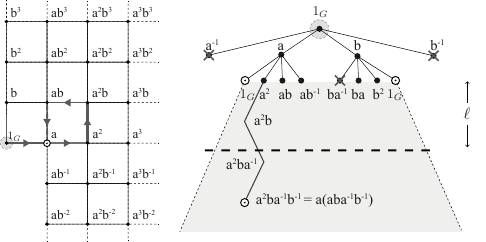}
\caption{A representation of $T_{\mathrm{SAW}(\Gamma_i(\Phi_\prec),v_i)}$ and a logarithmic depth truncation for $G = \mathbb{Z}^2$ and $\Gamma = \Cay(\mathbb{Z}^2,\{\pm(1,0),\pm(0,1)\})$}
\label{fig:trees}
\end{figure}

\begin{remark}
The constant $\frac{3}{4}$ in the definition of additive FPRAS is the standard choice for minimum success probability but it can be replaced by any constant bounded away from $\frac{1}{2}$ without any sensible change in the definition. In order to not have to deal with numerical details about the representation of $\lambda$, we will always implicitly assume that the values taken by $\lambda$ have a bounded number of digits uniformly on $\mathbb{M}$.
\end{remark}

\subsection{Weitz's algorithm and a computational phase transition}

Observe that if $G$ is the trivial group $\{1\}$, then $\mathcal{H}_{\{1\}}$ is exactly the family of finite hardcore models. In this case, we have that
$$
f_G(\Gamma,\lambda) = \frac{\log Z_\Gamma(\lambda)}{\lvert V(\Gamma)\rvert},
$$
and we can translate an approximation of $Z_\Gamma(\lambda)$ into an approximation of $f_G(\Gamma,\lambda)$ and vice versa.

In this finite context, it is common to consider a \emph{fully polynomial-time approximation scheme}. Given a family $\mathbb{M} \subseteq \mathcal{H}_{\{1\}}$ of finite hardcore models, we will say that $\mathbb{M}$ admits a {\bf fully polynomial-time approximation scheme (FPTAS)} for $Z_\Gamma(\lambda)$ if there is an algorithm such that, given an input $(\Gamma,\lambda) \in \mathbb{M}$ and $\epsilon > 0$, outputs $\hat{Z}$ with
$$
Z_\Gamma(\lambda)e^{-\epsilon} \leq \hat{Z} \leq Z_\Gamma(\lambda)e^{\epsilon},
$$
in polynomial time in $\lvert V(\Gamma) \rvert$ and $\epsilon^{-1}$. An FPTAS is regarded as an efficient and uniform approximation algorithm for $Z_\Gamma(\lambda)$.

If we take logarithms and divide by $\lvert V(\Gamma) \rvert$ in the previous equation, we obtain that
$$
\lvert f_G(\Gamma,\lambda) - \hat{f} \rvert  = \lvert \frac{\log Z_\Gamma}{\lvert V(\Gamma) \rvert} - \frac{\log \hat{Z}}{\lvert V(\Gamma) \rvert} \rvert \leq \frac{\epsilon}{\lvert V(\Gamma) \rvert },
$$
where $\hat{f} = \frac{\log \hat{Z}}{\lvert V(\Gamma) \rvert}$, so an FTPAS for $Z_\Gamma(\lambda)$ is equivalent to an additive FPTAS for $f_G(\Gamma,\lambda)$, since a polynomial in $\lvert V(\Gamma) \rvert$ and $\epsilon^{-1}$ is also a polynomial in $\lvert \left<\Gamma,\lambda\right> \rvert$ and $\lvert V(\Gamma) \rvert\epsilon^{-1}$ and vice versa. The same correspondence holds between the natural randomized counterparts (FPRAS and additive FPRAS, respectively).

We will fix a positive integer $\Delta$ and $\lambda_0 > 0$. Given these parameters, we aim to develop a fully polynomial-time additive approximation on $\mathbb{M} = \mathcal{H}_G^\Delta(\lambda_0)$.

The main theorem in \cite{1-weitz} was the development of an FPTAS for $Z_\Gamma(\lambda_0)$ on $\mathcal{H}_{\{1\}}^\Delta(\lambda_0)$ for $\lambda_0 < \lambda_c(\Delta)$. It is not difficult to see that the theorem extends to non-constant activity functions $\lambda$. Then, and also translated into the language of free energy, we have the following result.

\begin{theorem}[{\cite{1-weitz}}]
\label{thm:weitz}
For every $\Delta \in \mathbb{N}$ and $0 < \lambda_0 < \lambda_c(\Delta)$, there exists an FPTAS (resp. additive FPTAS) on $\mathcal{H}_{\{1\}}^\Delta(\lambda_0)$ for $Z_\Gamma(\lambda)$ (resp. for $f_G(\Gamma,\lambda)$).
\end{theorem}

This theorem was subsequently refined in \cite{1-sinclair} by considering connective constants instead of maximum degree $\Delta$. A very interesting fact is that when classifying graphs according to their maximum degree, then Theorem \ref{thm:weitz} is in some sense optimal due to the following theorem.

\begin{theorem}[{\cite{1-sly,2-sly}}]
\label{thm:sly}
For every $\Delta \geq 3$ and $\lambda_0 > \lambda_c(\Delta)$, there does not exist an FPRAS (resp. an additive FPRAS) on $\mathcal{H}_{\{1\}}^\Delta(\lambda_0)$ for $Z_\Gamma(\lambda_0)$ (resp. for $f_G(\Gamma,\lambda_0)$), unless $\mathrm{NP} = \mathrm{RP}$.
\end{theorem}

\begin{remark}
Notice that the lack of existence of an FPRAS (resp. additive FPRAS) directly implies the lack of existence of an FPTAS (resp. additive FPTAS).
\end{remark}

The combination of Theorem \ref{thm:weitz} and Theorem \ref{thm:sly} is what is regarded as a \emph{computational phase transition}. We aim to extend these theorems to the infinite setting. A theoretical advantage about considering $f_G(\Gamma,\lambda)$ instead of $Z_\Gamma(\lambda)$ is that the free energy still makes sense in infinite graphs and at the same time recovers the theory for $Z_\Gamma(\lambda)$ in the finite case.

\subsection{An extension of Weitz's algorithm to the infinite setting}

For algorithmic purposes, in this section we only consider finitely generated groups $G$ with a fixed symmetric set of generators $S$. In this case, if $(\Gamma,\lambda) \in \mathcal{H}_{G}^\Delta$, then it suffices to know $\Gamma[U_0]$ for some fundamental domain $U_0$ and $L = \{(v_1,s,v_2) \in U_0 \times S \times U_0: (v_1,sv_2) \in E(\Gamma)\}$ in order to fully reconstruct the graph $\Gamma$. In particular, the size of the necessary information to reconstruct $\Gamma$ is bounded by a polynomial in $\Delta$ and $\lvert \Gamma / G \rvert$. In addition, given a $G$-invariant activity function $\lambda: V(\Gamma) \to \mathbb{Q}_{>0}$ (i.e., that only takes positive rational values), we only need to know $\left.\lambda\right\vert_{U_0}$ to recover $\lambda$, i.e., just $\lvert \Gamma / G \rvert$ many rational numbers. Therefore, in this context, the length $\lvert \left<\Gamma,\lambda\right> \rvert$ of the representation $\left<\Gamma,\lambda\right>$ of a hardcore model $(\Gamma, \lambda)$ will be polynomial in $\Delta$ and $\lvert \Gamma / G \rvert$.

First, we are interested in being able to generate in an effective way balls of arbitrary radius in $\Cay(G,S)$. Given an input word $w \in S^*$, we will assume that we can decide whether $e(w) = 1_G$ or not in time $\exp(O(\lvert w \rvert))$, where $\lvert w \rvert$ denotes the length of $w$, $e: S^* \to G$ is the usual evaluation map, and the $O$-notation regards $\lvert S \rvert$ and $\Delta$ as constants. In other words, we will assume that \emph{the word problem of $G$ can be solved in exponential time}. By problems that can be solved in exponential time, we mean the set of decision problems that can be solved by a deterministic Turing machine in time $\exp(O(n))$. This complexity class is sometimes known as $\mathrm{E}$; it contains $\mathrm{P}$ and it is strictly contained in $\mathrm{EXP}$.

Now, if the word problem can be solved in exponential time, then $\Cay(G,S)$ is constructible in time $\exp(O(\ell))$ as well (see \cite[Theorem 5.10]{1-meier}); this is to say, given $\ell > 0$, we can generate $B_{\Cay(G,S)}(1_G,\ell)$ in time $\exp(O(\ell))$. Having that, it is possible to construct $\Gamma[B_{\Cay(G,S)}(1_G,\ell)U_0]$ in time $O(\mathrm{poly}(\lvert \Gamma / G \rvert)\exp(O(\ell)))$ by identifying each $g \in B_{\Cay(G,S)}(1_G,\ell)$ with a copy $\Gamma[gU_0]$ of $\Gamma[U_0]$ and by connecting it to other adjacent copies according to $L$.

In the ordered case, we will also consider the situation where \emph{the algebraic past of $(G,\prec)$ can be decided in exponential time}, i.e., given an input word $w \in S^*$, we will consider that we can decide whether $e(w) \in \Phi_\prec$ or not in time $\exp(O(\lvert w \rvert))$. Notice that this implies that the word problem can be solved in time $\exp(O(\lvert w \rvert))$, since $e(w) = 1_G$ if and only if $e(w) \notin \Phi_\prec$ and $e(w^{-1}) \notin \Phi_\prec$. In particular, and since $\lvert B_{\Cay(G,S)}(1_G,\ell) \rvert \leq \lvert S \rvert^\ell$, by identifying and removing all the copies $\Gamma[gU_0]$ with $g \in \Phi_\prec$ in $\Gamma[B_{\Cay(G,S)}(1_G,\ell)U_0]$, we can construct $\Gamma[(B_{\Cay(G,S)}(1_G,\ell) \setminus \Phi_\prec)U_0]$ in time $O(\mathrm{poly}(\lvert \Gamma / G \rvert)\exp(O(\ell)))$. Recall that we are not losing generality if we assume $G$ to be ordered instead of just virtually ordered.

Considering all this, we have the following key theorem.

\begin{theorem}
\label{thm:rand-approx}
Let $G$ be a finitely generated amenable group such that its word problem can be solved in exponential time. Then, for every $\Delta \in \mathbb{N}$ and $0 < \lambda_0 < \lambda_c(\Delta)$, there exists an additive FPRAS on $\mathcal{H}_G^\Delta(\lambda_0)$ for $f_G(\Gamma,\lambda)$. If, in addition, $G$ is orderable and has an algebraic past that can be decided in exponential time, then the algorithm can be chosen to be deterministic, i.e., an additive FPTAS.
\end{theorem}

\begin{proof}
Pick $(\Gamma,\lambda)$ as in the statement and enumerate as $U_0 = \left\{v_1,\dots,v_n\right\}$ the fundamental domain of $G \acts \Gamma$. Denote $n = \lvert \Gamma / G \rvert$. Then, by Theorem \ref{thm:rep1},
$$f_G(\Gamma,\lambda) = \sum_{i=1}^n -\mathbb{E}_{\tilde{\Phi}}\log\Prob_{\Gamma_i(\tilde{\Phi}),\lambda}([0^{v_i}]),$$
where $\Gamma_i(\tilde{\Phi}) = \Gamma \setminus (\tilde{\Phi}(\iden)U_0 \cup \{v_1,\dots,v_{i-1}\})$ and $\tilde{\Phi}$ is any invariant random past of $G$. Let $f_i = -\mathbb{E}_{\tilde{\Phi}}\log\Prob_{\Gamma_i(\tilde{\Phi}),\lambda}([0^{v_i}])$. Given $\epsilon > 0$, our goal is to generate numbers $\hat{f}_i$ such that
$$
\hat{f}_i - \frac{\epsilon}{n} \leq f_i \leq \hat{f}_i + \frac{\epsilon}{n}
$$
for every $i = 1,\dots,n$. If we manage to obtain these approximations, we have that
$$
\sum_{i=1}^n \hat{f}_i - \epsilon \leq \sum_{i=1}^n f_i = f_G(\Gamma,\lambda) \leq \sum_{i=1}^n \hat{f}_i + \epsilon,
$$
so $\hat{f} := \sum_{i=1}^n \hat{f}_i$ will be the required approximation. By SSM in $\Gamma$, we have that, for every $\ell > 0$,
$$
\lvert \log\Prob_{\Gamma_i(\tilde{\Phi}),\lambda}([0^{v_i}]) - \log\Prob_{\Gamma_i(\tilde{\Phi}) \cap B_\Gamma(v_i,\ell),\lambda}([0^{v_i}]) \rvert \leq C\exp(-\alpha\ell),
$$
and, again by SSM but now on the tree of self-avoiding walks,
$$
\lvert \log\Prob_{\Gamma_i(\tilde{\Phi}) \cap B_\Gamma(v_i,\ell),\lambda}([0^{v_i}]) - \log \Prob_{T_{\mathrm{SAW}}(\Gamma_i(\tilde{\Phi}) \cap B_\Gamma(v_i,\ell),v_i)\cap B(\rho_i,\ell),\overline{\lambda}}([0^{\rho_i}]) \rvert \leq C\exp(-\alpha\ell),
$$
where $C$ is a constant that depends on $\Delta$ and $\lambda_0$. Notice that $T_{\mathrm{SAW}}(\Gamma_i(\tilde{\Phi}) \cap B_\Gamma(v_i,\ell),v_i)\cap B(\rho_i,\ell) = T_{\mathrm{SAW}}(\Gamma_i(\tilde{\Phi}),v_i)\cap B(\rho_i,\ell)$, so
$$
\lvert f_i - \mathbb{E}_{\tilde{\Phi}}\log \Prob_{T_{\mathrm{SAW}}(\Gamma_i(\tilde{\Phi}),v_i)\cap B(\rho_i,\ell),\overline{\lambda}}([0^{\rho_i}]) \rvert \leq 2C\exp(-\alpha\ell).
$$

In order to conclude, it suffices to define $\hat{f}_i$ as $\mathbb{E}_{\tilde{\Phi}}\log \Prob_{T_{\mathrm{SAW}}( \Gamma_i(\tilde{\Phi}),v_i)\cap B(\rho_i,\ell),\overline{\lambda}}([0^{\rho_i}])$ with $\ell$ so that $2C\exp(-\alpha\ell) \leq \frac{\epsilon}{n}$ and show that each approximation $\hat{f}_i$ can be efficiently computed. Notice that we can pick $\ell$ to be $\ell^* = \left\lceil\frac{1}{\alpha}\log(2Cn\epsilon^{-1})\right\rceil$. 

Let's first assume that $\tilde{\Phi}$ is a (deterministic) algebraic past $\Phi_\prec$ of $G$ that can be decided in exponential time. The general probabilistic case will be a slight variation of this case. In order to compute $\hat{f}_i$, first generate the ball $\Gamma[B_{\Cay(G,S)}(1_G,\ell^*)U_0]$, which takes time $\mathrm{poly}(n)\exp(O(\ell^*)) = \mathrm{poly}((1+\epsilon^{-1})n)$. Notice that $B_{\Gamma}(v_i,\ell) \subseteq B_{\Cay(G,S)}(1_G,\ell^*)U_0$. Next, remove the vertices that are at distance greater than $\ell^*$ to $v_i$ and the ones which belong to $\Phi_\prec U_0 \cup \{v_1,\dots,v_{i-1}\}$. This procedure also takes time $\mathrm{poly}(n)\exp(O(\ell^*))$, since $\Phi_\prec$ can be decided in exponential time. Having this, construct the tree of self-avoiding walks $T_i := T_{\mathrm{SAW}}(\Gamma_i(\Phi_\prec),v_i) \cap B(\rho_i,\ell^*)$ (which is a subtree of the tree of self-avoiding walks of $\Gamma_i(\Phi_\prec)$). Using the recursive procedure, compute $\Prob_{T_i,\overline{\lambda}}([0^{\rho_i}])$ and then compute its logarithm. For every $i = 1,\dots,n$, 
$$
\lvert T_i \rvert \leq \Delta^{\ell^*} \leq \left(C(1+\lambda_0)(1+\epsilon^{-1})n\right)^{\frac{\log\Delta}{\alpha}},
$$
which is also a bound for the order of time required for computing $\hat{f}_i$, because $T_i$ is a tree (see Figure \ref{fig:trees}). Finally, since we require to do this procedure $n$ times for each $i = 1,\dots,n$, we have that the total order of the algorithm is still $\mathrm{poly}((1+\epsilon^{-1})n)$, i.e., a polynomial in $n$ and $\epsilon^{-1}$, where the constants involved depend only on $\Delta$, $\lambda_0$, and $\lvert S \rvert$. This gives the desired additive FPTAS in the ordered case.

Now, in the general not necessarily ordered case, we consider the following variation of the previous algorithm. Let $\tilde{\Phi}_{\rm unif}$ be the uniform random order in $G$. Then,
\begin{align*}
	&	-\mathbb{E}_{\tilde{\Phi}_{\rm unif}}\log\Prob_{\Gamma_i(\tilde{\Phi}_{\rm unif}) \cap B_\Gamma(v_i,\ell),\lambda}([0^{v_i}]) \\
=	& -\frac{1}{2^m} \sum_{\Phi \subseteq B_{\Cay(G,S)}(1_G,\ell) \setminus \{1_G\}} \log\Prob_{\Gamma_i(\Phi)  \cap B_\Gamma(v_i,\ell),\lambda}([0^{v_i}])
\end{align*}
with $m = \lvert B_{\Cay(G,S)}(v_i,\ell) \rvert-1$. Consider the random variable $X$ with probability distribution $\Prob(X = -\log\Prob_{\Gamma_i(\Phi)  \cap B_\Gamma(v_i,\ell),\lambda}([0^{v_i}])) = \frac{1}{2^m}$ for $\Phi \subseteq B_{\Cay(G,S)}(1_G,\ell) \setminus \{1_G\}$. Then, $\mathbb{E}[X] = -\mathbb{E}_{\tilde{\Phi}_{\rm unif}}\log\Prob_{\Gamma_i(\tilde{\Phi}_{\rm unif}) \cap B_\Gamma(v_i,\ell),\lambda}([0^{v_i}])$. Due to Lemma \ref{lem:probbound}, we have that
$$
0 < \log\left(1 + \frac{\lambda_-}{(1+\lambda_+)^\Delta}\right) \leq X \leq \log(1+\lambda_+) < \infty.
$$

In particular,
$$
\mathrm{Var}(X) = \mathbb{E}[X^2] - \mathbb{E}[X]^2 \leq (\log(1+\lambda_+))^2 - \left(\log\left(1 + \frac{\lambda_-}{(1+\lambda_+)^\Delta}\right)\right)^2 =: V(\lambda,\Delta) < \infty.
$$

Now, let $\left\{X_{1}, \ldots, X_{N}\right\}$ be a random sample of size $N$ of the variable $X$ and define the sample average $\bar{X}_N \equiv \frac{1}{N} \sum_{j=1}^N X_j$. Notice that $\mathbb{E}[\bar{X}_N] = \mathbb{E}[X]$ and $\mathrm{Var}(\bar{X}_N) = \frac{1}{N} \mathrm{Var}(X) \leq \frac{V(\lambda,\Delta)}{N}$. Therefore, by Chebyshev's inequality, for any $\delta > 0$,
$$
\Prob\left(\lvert \bar{X}_N-\mathbb{E}[X] \rvert \leq \left(\frac{V(\lambda,\Delta)}{N\delta}\right)^{1/2}\right) \geq 1- \delta.
$$

We are interested in having $\lvert \bar{X}_N-\mathbb{E}[X] \rvert \leq \frac{\epsilon}{n}$ with probability greater than $\frac{3}{4}$. In order to guarantee this, we need to take a number of samples $N$ so that
$$
\left(\frac{V(\lambda,\Delta)}{N\delta}\right)^{1/2} \leq \frac{\epsilon}{n},
$$
and $\delta$ such that $1-\delta \geq \frac{3}{4}$, i.e., it suffices to take $N \geq 4V(\lambda,\Delta)(n\epsilon^{-1})^2$. Notice that we need to take a number of samples polynomial in $n$ and $\epsilon^{-1}$, and that each sample can be computed in polynomial time, exactly as in the ordered case. This gives the desired additive FPRAS in the general case.

\end{proof}

\begin{remark}
Notice that Theorem \ref{thm:rand-approx} holds for groups of exponential growth, despite it involves a polynomial time algorithm. In addition, in virtue of Theorem \ref{thm:sinclair}, the families of graphs in Theorem \ref{thm:rand-approx} could be parameterized according to their connective constant instead of the maximum degree.
\end{remark}

Next, we reduce the problem of approximating the partition function of a finite hardcore model to the problem of approximating the free energy of a hardcore model in $\mathcal{H}_G$.

\begin{proposition}
\label{prop:phase}
Let $G$ be an amenable group. Then, for every $\Delta \geq 3$ and $\lambda_0 > \lambda_c(\Delta)$, there does not exist an additive FPRAS on $\mathcal{H}_G^\Delta(\lambda_0)$ for $f_G(\Gamma,\lambda_0)$, unless $\mathrm{NP} = \mathrm{RP}$.
\end{proposition}

\begin{proof}
Suppose that we have an additive FPRAS on $\mathcal{H}_G^\Delta(\lambda_0)$ for $f_G(\Gamma,\lambda)$ for some amenable group $G$. Then, we claim that we would have an FPRAS on $\mathcal{H}_{\{1\}}^\Delta(\lambda_0)$ for $Z_\Gamma(\lambda_0)$, contradicting Theorem \ref{thm:sly}. Indeed, given the input $(\Gamma,\lambda) \in \mathcal{H}_{\{1\}}^\Delta(\lambda_0)$, it suffices to consider the graph $\Gamma^G$ made out of copies $\{\Gamma_g\}_{g \in G}$ of $\Gamma$ indexed by $g \in G$, where $v_g$ denotes the copy in $\Gamma_g$ of the vertex $v$ in $\Gamma$. Then, there is a natural action $G \acts \Gamma^G$ consisting on just translating copies of vertices, i.e., $gv_h = v_{hg}$, and a fundamental domain of the action is $U_0 = V(\Gamma_{1_G})$. Therefore, since $\lvert \Gamma^G / G \rvert = \lvert V(\Gamma) \rvert$, if we could $\epsilon$-approximate in an additive way $f_G(\Gamma,\lambda_0)$ in polynomial time in $\lvert \Gamma^G / G \rvert$ and $\epsilon^{-1}$, then we would be able to $\epsilon$-approximate in a multiplicative way $Z_\Gamma(\lambda_0)$ in polynomial time in $\lvert V(\Gamma) \rvert$ and $\epsilon^{-1}$, because
\begin{align*}
f_G(\Gamma,\lambda_0)	&	=	\inf_{F \in \mathcal{F}(G)} \frac{\log Z_{\Gamma^G}(FU_0,\lambda_0)}{\lvert FU_0 \rvert}		\\
					&	=	\inf_{F \in \mathcal{F}(G)} \frac{\log Z_{\Gamma^G}(U_0,\lambda_0)^{\lvert F \rvert}}{\lvert F \rvert\lvert U_0 \rvert} 	\\
					&	=	\frac{\log Z_{\Gamma^G}(U_0,\lambda_0)}{\lvert U_0 \rvert}							\\
					&	=	\frac{\log Z_{\Gamma}(\lambda_0)}{\lvert V(\Gamma) \rvert},
\end{align*}
but this contradicts Theorem \ref{thm:sly}.
\end{proof}

Considering Theorem \ref{thm:rand-approx} and Proposition \ref{prop:phase}, we have the following corollary.

\begin{corollary}
\label{cor:phase}
Let $G$ be a finitely generated amenable group such that its word problem can be solved in exponential time. Then, for every $\Delta \geq 3$ and $\lambda_0 > 0$, if $\lambda_0 < \lambda_c(\Delta)$, there exists an additive FPRAS on $\mathcal{H}_G^\Delta(\lambda_0)$ for $f_G(\Gamma,\lambda)$. If, in addition, $G$ has a finite index orderable subgroup such that its algebraic past can be decided in exponential time, then the algorithm can be chosen to be deterministic, i.e., there exists an additive FPTAS on $\mathcal{H}_G^\Delta(\lambda_0)$ for $f_G(\Gamma,\lambda)$. On the other hand, if $\lambda_0 > \lambda_c(\Delta)$, there is no additive FPRAS on $\mathcal{H}_G^\Delta(\lambda_0)$ for $f_G(\Gamma,\lambda)$, unless $\mathrm{NP} = \mathrm{RP}$.
\end{corollary}

\begin{remark}
Notice that Corollary \ref{cor:phase} still holds for $\Delta = 1$ and $\Delta = 2$. The first case is trivial and in the second case, there is no phase transition and the conditions for the existence of an additive FPRAS (resp. additive FPTAS) hold for every $\lambda_0$.
\end{remark}

To ask that the word problem can be solved in exponential time seems to be a natural requirement for having an efficient algorithm for approximating $f_G(\Gamma,\lambda)$ and, fortunately, there are several classes of finitely generated groups which satisfy this condition.

\begin{example}
Lipton and Zalcstein \cite{1-lipton} proved that every linear group over a field of characteristic zero has a word problem that can be solved in logarithmic space. This result was extended by Simon \cite{1-simonword} to linear groups over a field of prime characteristic. In particular, the word problem of all finitely generated amenable linear groups---which by the Tits alternative \cite{1-tits1972} must be virtually solvable---can be solved in logarithmic space, and therefore polynomial time.

Due to a result from Mal'tsev \cite{1-maltsev1951}, all solvable subgroups of the integer general linear group $\mathrm{GL}_d(\mathbb{Z})$ are polycyclic (i.e., solvable groups in which every subgroup is finitely generated) and virtually polycyclic groups coincide with the class of polycyclic-by-finite groups, which are always finitely presented, residually finite, and have many other desirable algorithmic properties (see \cite{1-baumslag1991}). On the other hand, Auslander \cite{1-auslander1967} and Swan \cite{1-swan1967} proved that any polycyclic group is a subgroup of the integer general linear group. This shows that the class of polycyclic groups is a general and natural setup for the application of our results, since they are amenable, finitely generated, and their word problem can be solved in polynomial time. Examples of polycyclic groups include all finitely generated abelian groups and all finitely generated nilpotent groups.
\end{example}

To understand how to guarantee the existence of an algebraic past that can be decided in exponential time, we start by observing two basic facts: (1) the group of integers $\mathbb{Z}$ is orderable with the natural order and its algebraic past can be decided in linear time and (2) the following lemma.

\begin{lemma}
\label{lem:extorder}
Let $(H_1,\prec_1)$ and $(H_2,\prec_2)$ be two ordered groups and let $G$ be a finitely generated group which is an extension of $H_2$ by $H_1$, i.e., there is a short exact sequence
$$
1 \longrightarrow H_1 \stackrel{\iota}{\longrightarrow} G \stackrel{\pi}{\longrightarrow} H_2 \longrightarrow 1.
$$

Then, $G$ can be ordered by considering the algebraic past $\Phi := \iota(\Phi_1) \cup \pi^{-1}(\Phi_2)$, where $\Phi_i$ denotes the algebraic past of $H_i$ for $i=1,2$. In particular, if $\iota(\Phi_1)$ and $\pi^{-1}(\Phi_2)$ can be decided in exponential time, then $\Phi$ can be decided in exponential time as well.
\end{lemma}

\begin{proof}
Consider the set $\Phi := \iota(\Phi_1) \cup \pi^{-1}(\Phi_2)$. It suffices to check that $\Phi$ is a semigroup (i.e., $\Phi^2 \subseteq \Phi$) and $G = \Phi \sqcup \{1_G\} \sqcup \Phi^{-1}$. Indeed, since $\iota(H_1) = \pi^{-1}(1_{H_2})$, we have that
\begin{align*}
G	&	=	\pi^{-1}(H_2)	\\
	&	=	\pi^{-1}(\Phi_2 \sqcup \{1_{H_2}\} \sqcup \Phi_2^{-1})		\\
	&	=	\pi^{-1}(\Phi_2) \sqcup \pi^{-1}(1_{H_2}) \sqcup \pi^{-1}(\Phi_2^{-1})		\\
	&	=	\pi^{-1}(\Phi_2) \sqcup \iota(H_1) \sqcup \pi^{-1}(\Phi_2)^{-1}		\\
	&	=	\pi^{-1}(\Phi_2) \sqcup \iota(\Phi_1 \sqcup \{1_{H_1}\} \sqcup \Phi_1^{-1}) \sqcup \pi^{-1}(\Phi_2)^{-1}		\\
	&	=	[\iota(\Phi_1) \sqcup \pi^{-1}(\Phi_2)] \sqcup \iota(1_{H_1}) \sqcup [\iota(\Phi_1)^{-1} \sqcup \pi^{-1}(\Phi_2)^{-1}	] \\
	&	=	[\iota(\Phi_1) \sqcup \pi^{-1}(\Phi_2)] \sqcup \{1_G\} \sqcup [\iota(\Phi_1)\sqcup \pi^{-1}(\Phi_2)]^{-1} \\
	&	=	\Phi \sqcup \{1_G\} \sqcup \Phi^{-1}.
\end{align*}

On the other hand, since $\iota(\Phi_1) \subseteq \iota(H_1) = \pi^{-1}(1_{H_2})$ and $\Phi_i^2 \subseteq \Phi_i$ for $i=1,2$,
\begin{align*}
\Phi^2	&	=			[\iota(\Phi_1) \sqcup \pi^{-1}(\Phi_2)]^2	\\
		&	=			\iota(\Phi_1^2) \cup \iota(\Phi_1)\pi^{-1}(\Phi_2) \cup \pi^{-1}(\Phi_2)\iota(\Phi_1) \cup \pi^{-1}(\Phi_2^2)	\\
		&	\subseteq 		\iota(\Phi_1) \cup \pi^{-1}(1_{H_2})\pi^{-1}(\Phi_2) \cup \pi^{-1}(\Phi_2)\pi^{-1}(1_{H_2}) \cup \pi^{-1}(\Phi_2)	\\
		&	=			\iota(\Phi_1) \cup \pi^{-1}(\Phi_2) \cup \pi^{-1}(\Phi_2) \cup \pi^{-1}(\Phi_2)	\\
		&	=			\Phi.
\end{align*}

Therefore, $\Phi$ is an algebraic past for $G$ and it induces the invariant order $\prec$, where $g \prec h$ for $h,g \in G$ if and only if $\pi(gh^{-1}) \prec_2 1_{H_2}$ or [$\pi(gh^{-1}) = 1_{H_2}$ and $gh^{-1} \in \iota(\Phi_1)$]. In particular, if $\iota(\Phi_1)$ and $\pi^{-1}(\Phi_2)$ can be decided in exponential time, then it is direct that $\Phi$ can also be decided in exponential time.
\end{proof}

The previous lemma can be used as a tool for constructing algebraic pasts that can be decided in exponential time in new groups out of simpler ones. We have the following example.

\begin{example}
By the fundamental theorem of finitely generated abelian groups, every finitely generated abelian group $G$ is isomorphic to a group of the form $\mathbb{Z}^d \oplus \mathbb{Z}_{q_1} \oplus \cdots \oplus \mathbb{Z}_{q_t}$, where $d \geq 0$ is the rank and $q_1, \ldots, q_t$ are powers of prime numbers. In particular, $[G:\mathbb{Z}^d]$ is finite, so $\mathbb{Z}^d$ is a finite index subgroup of $G$. On the other hand, $\mathbb{Z}^d$ is an orderable group. A canonical presentation of $\mathbb{Z}^d$ is given by
$$\left< a_1, \dots, a_d \mid \{[a_i, a_j]: 1 \leq i,j \leq d\}\right>,$$
where $[g, h] = ghg^{-1}h^{-1}$ is the commutator of $g$ and $h$. A normal form for $\mathbb{Z}^d$ is given by $\{a_1^{i_1}\cdots a_d^{i_d}: i_1,\dots,i_d \in \mathbb{Z}\}$ and, given any word $w \in \{a^{\pm 1}_1,\dots,a^{\pm 1}_d\}^*$, it takes linear time to obtain its normal form. A canonical order of $\mathbb{Z}^d$ is the lexicographic order $\prec$, where we declare $a_1^{i_1}\cdots a_d^{i_d} \prec a_1^{j_1}\cdots a_d^{j_d}$ if $i_k < j_k$ for some $1 \leq k \leq d$ and $i_m = j_m$ for $m < k$, where $<$ is the usual order in $\mathbb{Z}$. It is easy to see that it can be decided in polynomial time whether $a_1^{i_1}\cdots a_d^{i_d} \prec a_1^{0}\cdots a_d^{0}$ or not. An alternative way to see this is through Lemma \ref{lem:extorder}, by observing that $\mathbb{Z}^d$ is an extension of $\mathbb{Z}^{d-1}$ by $\mathbb{Z}$ and proceed inductively until reaching the base case $\mathbb{Z}^1 = \mathbb{Z}$.
 
Another illustrative example is the discrete Heisenberg group
$$
H_3(\mathbb{Z}) := \left\{\left(\begin{array}{ccc}
1	&	i	&	k	\\
0	&	1	&	j	\\
0	&	0	&	1
\end{array}\right): i, j, k \in \mathbb{Z}\right\} \subseteq \mathrm{SL}_3(\mathbb{Z}).
$$
The group $H_3(\mathbb{Z})$ is a non-abelian nilpotent (and therefore amenable with polynomial growth) finitely generated group. A presentation of $H_3(\mathbb{Z})$ is given by
$$\left< a, b, c \mid[a, c],[b, c],[a, b]c^{-1}\right>,$$
where we identify $a$, $b$, and $c$ with
$$
\left(\begin{array}{lll}
1 & 1 & 0 \\
0 & 1 & 0 \\
0 & 0 & 1
\end{array}\right),
\quad
\left(\begin{array}{lll}
1 & 0 & 0 \\
0 & 1 & 1 \\
0 & 0 & 1
\end{array}\right),
\quad	\text{and}	\quad
\left(\begin{array}{lll}
1 & 0 & 1 \\
0 & 1 & 0 \\
0 & 0 & 1
\end{array}\right),
$$
respectively. A normal form for $H_3(\mathbb{Z})$ is given by $\{b^jc^ka^i: i,j,k \in \mathbb{Z}\}$, where 
$$
b^jc^ka^i = \left(\begin{array}{lll}
1 & i & k \\
0 & 1 & j \\
0 & 0 & 1
\end{array}\right).
$$

It is not difficult to check that given a word $w \in \{a^{\pm 1},b^{\pm 1},c^{\pm 1}\}^*$ in its normal form and a generator $s \in \{a^{\pm 1},b^{\pm 1},c^{\pm 1}\}$, it takes linear time to write $sw$ in its normal form. Observe that it is enough to measure how much time it takes this particular operation and then proceed inductively. Now, it is known that $H_3(\mathbb{Z})$ is an extension of $\mathbb{Z}^2$ by $\mathbb{Z}$, i.e.,
$$
1 \longrightarrow \mathbb{Z} \stackrel{\iota}{\longrightarrow} H_3(\mathbb{Z}) \stackrel{\pi}{\longrightarrow} \mathbb{Z}^2 \longrightarrow 1,
$$
with $\iota: \mathbb{Z} \longrightarrow H_3(\mathbb{Z})$ and $\pi: H_3(\mathbb{Z}) \longrightarrow \mathbb{Z}^2$ given by $\iota(k) = c^k$ and $\pi(b^{j} c^{z} a^{i}) = (i,j)$, respectively. Considering Lemma \ref{lem:extorder} and that $\mathbb{Z}$ and $\mathbb{Z}^2$ have algebraic pasts $\Phi_\mathbb{Z}$ and $\Phi_{\mathbb{Z}^2}$, respectively, such that $\iota(\Phi_\mathbb{Z})$ and $\pi^{-1}(\Phi_{\mathbb{Z}^2})$ can be decided in exponential time, we conclude that $H_3(\mathbb{Z})$ also has an algebraic past that can be decided in exponential time. More concretely, this algebraic past $\Phi_{H_3(\mathbb{Z})}$ is defined by declaring that $b^jc^ka^i \in \Phi_{H_3(\mathbb{Z})}$ if and only if (1) $i < 0$ or (2) $i = 0$ and $j < 0$ or (3) $i=j=0$ and $k < 0$, which takes linear time to decide.

Finally, one other example is the case of the Baumslag-Solitar group $BS(1,2)$. A presentation of $BS(1,2)$ is given by $\left< a, b \mid bab^{-1}a^{-2}\right>$, where we identify $a$ and $b$ with the linear functions $x \mapsto 2x$ and $x \mapsto x+1$ in $\mathrm{Homeo}(\mathbb{R})$, respectively. The group $BS(1,2)$ is a non-nilpotent solvable (and therefore amenable with exponential growth) finitely generated group.  A normal form for $BS(1,2)$ is given by $\{a^{-j}b^{2k+1}a^{j+i}: i,j,k \in \mathbb{Z}\} \cup \{a^i: i \in \mathbb{Z}\}$ and it can also be checked that given a word $w \in \{a^{\pm 1},b^{\pm 1}\}^*$ in its normal form and a generator $s \in \{a^{\pm 1},b^{\pm 1}\}$, it takes polynomial time to write $sw$ in its normal form. It is known that $BS(1,2)$ is an extension of $\mathbb{Z}$ by $\mathbb{Z}\left[\frac{1}{2}\right]$, the group of dyadic rationals, i.e.,
$$
1 \longrightarrow \mathbb{Z}\left[\frac{1}{2}\right] \stackrel{\iota}{\longrightarrow} BS(1,2) \stackrel{\pi}{\longrightarrow} \mathbb{Z} \longrightarrow 1,
$$
with $\iota: \mathbb{Z}\left[\frac{1}{2}\right] \longrightarrow BS(1,2)$ and $\pi: BS(1,2) \longrightarrow \mathbb{Z}$ given by $\iota\left(\frac{2k+1}{2^j}\right) = a^{-j}b^{2k+1}a^{j}$, and $\pi(a^{-j}b^{2k+1}a^{j+i}) = i$ and $\pi(a^{i}) = i$, respectively. Then, due to Lemma \ref{lem:extorder} and the fact that $\mathbb{Z}\left[\frac{1}{2}\right]$ and $\mathbb{Z}$ have algebraic pasts $\Phi_{\mathbb{Z}\left[\frac{1}{2}\right]}$ and $\Phi_{\mathbb{Z}}$, respectively, such that $\iota(\Phi_{\mathbb{Z}\left[\frac{1}{2}\right]})$ and $\pi^{-1}(\Phi_{\mathbb{Z}})$ can be decided in exponential time, we conclude that $BS(1,2)$ also has an algebraic past $\Phi_{BS(1,2)}$ that can be decided in exponential time. More concretely, this algebraic past is defined by declaring that $a^{-j}b^{2k+1}a^{j+i} \in \Phi_{BS(1,2)}$ if and only if (1) $i < 0$ or (2) $i = 0$ and $k < 0$, which takes linear time to decide. This construction can be easily generalized to the group $BS(1,n)$ given by the presentation $\left< a, b \mid bab^{-1}a^{-n}\right>$.

\end{example}

The previous facts about word problems and algebraic pasts give us general conditions for efficiently generating $\Gamma[B_{\Cay(G,S)}(1_G,\ell)U_0]$ and $\Gamma[(B_{\Cay(G,S)}(1_G,\ell) \setminus \Phi_\prec)U_0]$, respectively. 

\section{Reductions}
\label{sec9}

In this section we provide a set of reductions which exploit the combinatorial properties of independent sets and relate the results already obtained for hardcore models with other systems. 

\subsection{$G$-subshifts and conjugacies}

Given a countable group $G$ and a finite set $\Sigma$ endowed with the discrete topology, the {\bf full shift} is the set $\Sigma^G$ of maps $\omega: G \to \Sigma$ endowed with the product topology. We define the {\bf $G$-shift} as the group action $G \times \Sigma^G \to \Sigma^G$ given by $(g,\omega) \mapsto g \cdot \omega$, where $(g \cdot \omega)(h) = \omega(hg)$ for all $h \in G$. A {\bf $G$-subshift} $\Omega$ is a $G$-invariant closed subset of $\Sigma^G$.

Given two $G$-subshifts $\Omega_1$ and $\Omega_2$, we say that a map $\varphi: \Omega_1 \to \Omega_2$ is a {\bf conjugacy} if it is bijective, continuous, and $G$-equivariant, i.e., $g \cdot \varphi(x) = \varphi(g \cdot x)$ for every $\omega \in \Omega_1$ and $g \in G$. In this context, these maps are characterized as \emph{sliding block codes} (e.g., see \cite{1-lind,1-ceccherini}) and provide a notion of isomorphism between $G$-subshifts.

Any $G$-subshift $\Omega$ is characterized by the existence of a family of forbidden patterns $\mathfrak{F} \subseteq \bigcup_{F \in \F(G)} \Sigma^F$ such that $\Omega = X_{\mathfrak{F}}$, where
$$
X_\mathfrak{F} = \{\omega \in \Sigma^G: (g \cdot x)_{F} \notin \mathfrak{F} \text{ for all } g \in G\}.
$$

If the family $\mathfrak{F}$ can be chosen to be finite, we say that $\Omega$ is a {\bf $G$-subshift of finite type ($G$-SFT)}. Given a finite set $S \subseteq G$, we can consider a family of $\lvert \Sigma \rvert \times \lvert \Sigma \rvert$ binary matrices ${\bf M} = \{M_s\}_{s \in S}$ with rows and columns indexed by the elements of $\Sigma$, and define the set
$$
\Omega_{{\bf M}} = \{\omega \in \Sigma^G: M_s(\omega(g),\omega(sg)) = 1 \text{ for all } g \in G, s \in S\}.
$$

The set $\Omega_{{\bf M}}$ is a special kind of $G$-SFT known as {\bf nearest neighbor (n.n.) $G$-SFT}. It is known that for every $G$-SFT there exists a conjugacy to a n.n. $G$-SFT, so we are not losing much generality by considering n.n. $G$-SFTs instead of general $G$-SFTs.

We say that a n.n. $G$-SFT $\Omega_{{\bf M}}$ has a {\bf safe symbol} if there exists $a \in \Sigma$ such that $a$ can be adjacent to any other symbol $b \in \Sigma$. Formally, this means that, for all $s \in S$ and $b \in \Sigma$, $M_s(a,b) = M_s(b,a) = 1$.

\subsection{Entropy and potentials}

Given a $G$-subshift $\Omega$, we define its {\bf topological entropy} as
$$
h_G(\Omega) := \lim_n \frac{\log \lvert \Omega_{F_n} \rvert}{\lvert F_n \rvert},
$$
where $\{F_n\}_n$ is a F{\o}lner sequence and $\Omega_F = \{\omega_F: \omega \in \Omega\}$ is the set of restrictions of points in $\Omega$ to the set $F \subseteq G$. It is known that the definition of $h_G(\Omega)$ is independent of the choice of F{\o}lner sequence and is also a \emph{conjugacy invariant}, i.e., if $\varphi: \Omega_1 \to \Omega_2$ is a conjugacy, then $h_G(\Omega_1) = h_G(\Omega_2)$.

A {\bf potential} is any continuous function $\phi: \Omega \to \mathbb{R}$. Given a potential, we define the {\bf pressure} as
$$
p_G(\phi) := \lim_n \frac{\log \lvert Z_{F_n}(\phi) \rvert}{\lvert F_n \rvert},
$$
where $Z_{F_n}(\phi) = \sum_{w \in \Omega_F} \sup_{\omega \in [w]}\exp(\sum_{g \in F} \phi(g \cdot \omega))$. Notice that $p_G(0) = h_G(\Omega)$.

A {\bf single-site potential} is any potential that only depends on the value of $\omega$ at $1_G$, i.e., $\omega_{1_G}$. In other words, and without risk of ambiguity, we can think that a single-site potential is just a function $\phi: \Sigma \to \mathbb{R}$. In this case, $Z_{F_n}(\phi)$ has the following simpler expression:
$$
Z_{F_n}(\phi) = \sum_{w \in \Omega_F} \prod_{g \in F}\exp(\phi(w(g))).
$$

In this context, we will say that a symbol $a \in \Sigma$ is a {\bf vacuum state} if $a$ is a safe symbol and $\phi(a) = 0$.

\subsection{From a hardcore model to a n.n. $G$-SFT with a vacuum state}

Let $(\Gamma,\lambda)$ be a hardcore model in $\mathcal{H}_G$. If $G \acts \Gamma$ is transitive, then $\Gamma = \Cay(G,S)$ for some finite symmetric set $S \subseteq G$. Then, it is easy to see that if $\Sigma = \{0,1\}$ and, for all $s \in S$,
$$
M_s = \left(\begin{array}{cc} 1 & 1\\1 & 0\end{array}\right),
$$
then $\Omega_{{\bf M}}$ coincides with the set $X(\Gamma)$ and $0$ is a safe symbol. In addition, there is a natural relationship between the activity function $\lambda$ and the single-site potential given by $\phi(0) = 0$ and $\phi(1) = \log\lambda(v)$, where $v$ is some (or any) vertex $v$. In other words, if $G \acts \Gamma$ is transitive, then $(\Gamma,\lambda)$ corresponds to a n.n. $G$-SFT with a vacuum state.

More generally, if $G \acts \Gamma$ is almost transitive, then $(\Gamma,\lambda)$ can also be interpreted as a n.n. $G$-SFT with a vacuum state. Indeed, consider the set $\Sigma_\Gamma = X(\Gamma[U_0])$, i.e., the set of independent sets of the subgraph $\Gamma[U_0]$ induced by some fundamental domain $U_0$. Since $\Gamma$ is locally finite and $G \acts \Gamma$ is free, there must exist a finite set $S \subseteq G \setminus \{1_G\}$ such that $SU_0$ contains all the vertices adjacent to $U_0$. Considering this, we define a collection of matrices ${\bf M}_\Gamma = \{M_s\}_{s \in S}$, where
$$
M_s(x,x') = 
\begin{cases}
1	&	\text{ if } xx' \in X(\Gamma[U_0 \cup sU_0]),	\\
0	&	\text{ otherwise,}
\end{cases}
$$
and $xx'$ denotes the concatenation of the independent set $x$ of $\Gamma[U_0]$ and the independent set $x'$ of $\Gamma[sU_0]$. In other words, $M_s(x,x') = 1$ if and only if the union of the independent set $x$ and the independent set $x'$ is also an independent set of $\Gamma[U_0 \cup sU_0]$.

Then, there is a natural identification between $\Omega_{{\bf M}_\Gamma} \subseteq \Sigma_\Gamma^G$ and $X(\Gamma)$. In particular, the symbol $0^{U_0} \in X(\Gamma[U_0])$ plays the role of a safe symbol in $\Omega_{{\bf M}_\Gamma}$. Moreover, we can define the single-site potential $\phi_\lambda: \Omega_{{\bf M}_\Gamma} \to \mathbb{R}$ given by $\phi_\lambda(\omega) = \sum_{v \in U_0} \omega_{1_G}(v) \log \lambda(v)$. Then, for every $F \in \mathcal{F}(G)$,
\begin{align*}
Z_{F}(\phi_\lambda)	&	=	\sum_{w \in \Omega_F} \prod_{g \in F}\exp(\phi_\lambda(w(g)))								\\
				&	=	\sum_{w \in X(\Gamma[FU_0])} \prod_{g \in F}\exp\left(\sum_{v \in U_0} w(g)(v) \log \lambda(v)\right)	\\
				&	=	\sum_{w \in X(\Gamma[FU_0])} \prod_{g \in F} \prod_{v \in U_0} \exp(w(g)(v) \log \lambda(v))			\\
				&	=	\sum_{w \in X(\Gamma[FU_0])} \prod_{v \in FU_0} \lambda(v)^{w(v)}								\\
				&	=	Z_{\Gamma}(FU_0,\lambda).
\end{align*}

Therefore, $p_G(\Omega_\Gamma, \phi_\lambda) = f_G(\Gamma,\lambda)$. In the language of dynamics, for every almost transitive and locally finite graph $\Gamma$, there exists a n.n. $G$-SFT with a safe symbol $\Omega_\Gamma$ such that $G \acts X(\Gamma)$ and $G \acts \Omega_\Gamma$ are conjugated. Moreover, this gives us a way to identify any hardcore model $(\Gamma,\lambda) \in \mathcal{H}_G$ with the corresponding $G$-SFT $\Omega_\Gamma$ and the single-site potential $\phi_\lambda$.

\begin{figure}[ht]
\centering
\includegraphics[scale = 0.8]{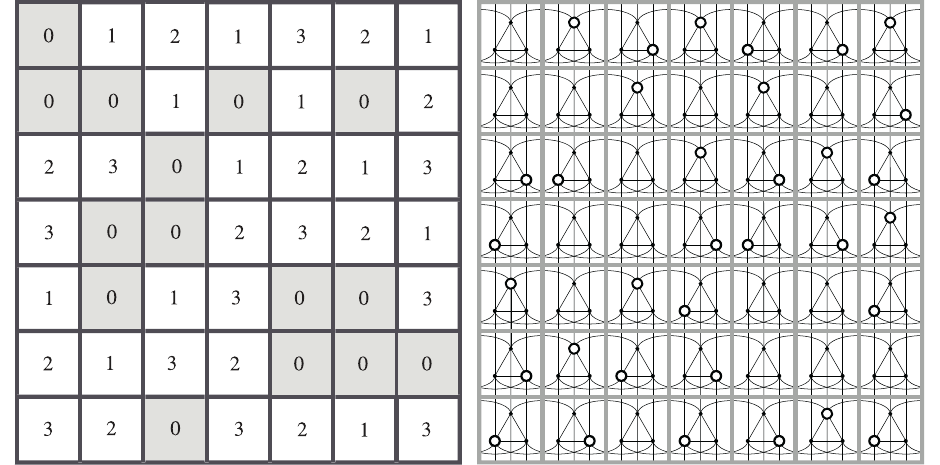}
\caption{On the left, a sample of a configuration in the n.n. SFT $\Omega$ corresponding to proper $3$-colorings of $\Cay(\mathbb{Z}^2, \{\pm(1,0), \pm(0,1)\})$ plus a safe symbol $0$, where each square corresponds to an element of $\mathbb{Z}^2$. On the right, the independent set in the graph $\Gamma_\Omega$ representing the configuration in $\Omega$}
\label{fig:3cols}
\end{figure}

\subsection{From a n.n. $G$-SFT with a vacuum state to a hardcore model}

Conversely, given a n.n. $G$-SFT $\Omega$ and a potential with a vacuum state, we can translate this scenario into a hardcore model. Indeed, consider the graph $\Gamma_\Omega$ defined as follows:
\begin{itemize}
\item for every $g \in G$, consider a finite graph $\Gamma_g$ isomorphic to $K_{\lvert \Sigma \rvert}$, the complete graph with $\lvert \Sigma \rvert$ vertices. In other words, for each $g \in G$ and for each $a \in \Sigma$ there will be a vertex $v_{g,a} \in V(\Gamma_g)$ and for every $a \neq b$, the edge $(v_{g,a},v_{g,b})$ will belong to $E(\Gamma_g)$;
\item the graph $\Gamma_\Omega$ will be the union of all the finite graphs $\Gamma_g$ plus some extra edges;
\item for every $s \in S$ and $a,b \in \Sigma$, we add the edge $(v_{1_G,a},v_{s,b})$ if and only if $M_s(a,b) = 0$;
\item we define $\lambda_\phi: \Gamma_\Omega \to  \mathbb{R}_{>0}$ as $\lambda_\phi(v_{g,a}) = \exp(\phi(a))$ for every $g \in G$ and $a \in \Sigma$.
\end{itemize}

Then, $G$ acts on $\Gamma_\Omega$ in the natural way and $V(\Gamma_{1_G})$ corresponds to a fundamental domain of the action $G \acts \Gamma_\Omega$. In the language of dynamics, for every n.n. $G$-SFT with a safe symbol $\Omega$, there exists an almost transitive and locally finite graph $\Gamma_\Omega$ such that $G \acts \Omega$ and $G \acts X(\Gamma_\Omega)$ are conjugated. Moreover, it is clear that
$$
f_G(\Gamma_\Omega,\lambda_\phi) = p_G(\Omega,\phi),
$$
so in particular, all the representation and approximation theorems for free energy of hardcore models can be used to represent and approximate the pressure of n.n. $G$-SFTs $\Omega$ and potentials $\phi$ with a vacuum state, provided $(\Gamma_\Omega,\lambda_\phi)$ satisfies the corresponding hypotheses. Relevant cases like the \emph{Widom-Rowlinson model} \cite{2-georgii} and graph homomorphisms from $\Gamma$ to any finite graph with some vertex (which plays the role of a safe symbol) connected to every other vertex fall in this category.

\subsection{Topological entropy and constraintedness of n.n. $G$-SFTs with safe symbols}
\label{sec9-4}

Let $\Omega \subseteq \Sigma^G$ be a n.n. $G$-SFT with $\lvert \Sigma \rvert = n_s + n_u$, where $n_s$ denotes the number of safe symbols in $\Sigma$ and $n_u$ denotes the number of symbols that are not safe symbols (\emph{unsafe}). Consider the n.n. $G$-SFT $\Omega_{n_u} \subseteq \Sigma_{n_u}^G$ obtained after collapsing all the safe symbols in $\Sigma$ into a single one, so that the $\lvert \Sigma_{n_u} \rvert = 1 + n_u$, and construct the graph $\Gamma_{\Omega_{n_u}}$. Then, given $F \in \F(G)$, we have that
\begin{align*}
\lvert \Omega_F \rvert	&	=	\sum_{x \in X(\Gamma_{\Omega_{n_u}},FU_0)}	\prod_{v \in FU_0} 1^{x(v)}n_s^{1-x(v)}			\\
			&	=	\sum_{x \in X(\Gamma_{\Omega_{n_u}},FU_0)}	\prod_{v \in FU_0} \left(\frac{1}{n_s}\right)^{x(v)}n_s			\\
			&	=	n_s^{\lvert FU_0 \rvert} \sum_{x \in X(\Gamma_{\Omega_{n_u}},FU_0)} \prod_{v \in FU_0} \left(\frac{1}{n_s}\right)^{x(v)}		\\
			&	=	n_s^{\lvert FU_0 \rvert} Z_{\Gamma_{\Omega_{n_u}}}(FU_0, \frac{1}{n_s}),
\end{align*}
so, considering that $n_u = \lvert U_0 \rvert = \lvert \Gamma_{\Omega_{n_u}} / G \rvert$,
\begin{align*}
h_G(\Omega)	&	=	\lim_n \frac{\log \lvert \Omega_{F_n} \rvert}{\lvert F_n \rvert}								\\
			&	=	\lvert U_0 \rvert\log n_s + \lim_n \frac{Z_{\Gamma_{\Omega_{n_u}}}(FU_0, 1/n_s)}{\lvert F_n \rvert}	\\
			&	=	\lvert U_0 \rvert\log n_s + \lvert U_0 \rvert f_G(\Gamma_{\Omega_{n_u}},1/n_s)	\\
			&	=	n_u\left(\log n_s + f_G(\Gamma_{\Omega_{n_u}},1/n_s)\right).
\end{align*}

Therefore, to understand and approximate $h_G(\Omega)$ reduces to study the hardcore model on $\Gamma_{\Omega,n_s}$ with constant activity $\frac{1}{n_s}$. In particular, if
$$
\frac{1}{n_s} < \lambda_c(\mu(\Gamma_{\Omega_{n_u}})),
$$
the hardcore model $(\Gamma_{\Omega,n_s},1/n_s)$ satisfies exponential SSM and the theory developed in the previous sections applies. This motivates the definition of the {\bf constraintedness} of a n.n. $G$-SFT $\Omega$ as the connective constant of $\Gamma_{\Omega_{n_u}}$, i.e.,
$$
\mu(\Omega) := \mu(\Gamma_{\Omega_{n_u}}),
$$
which can be regarded as a measure of how much constrained is $\Omega$ (the higher $\mu(\Omega)$, the more constrained it is). Notice that if
$$
\frac{1}{n_s} < \lambda_c(\mu(\Omega)+1),
$$
then $(\Gamma_{\Omega_{n_u}},1/n_s)$ satisfies exponential SSM. In particular, $\Omega_{n_u}$ has a unique \emph{measure of maximal entropy} and therefore, also $\Omega$ has unique measure of maximal entropy, namely, the pushforward measure (see \cite{2-burton,1-haggstrom}). Moreover, the topological topological entropy of $\Omega_{n_u}$ has an arboreal representation and can be approximated efficiently. Since $\mu(\Omega) \leq \Delta(\Gamma_{\Omega_{n_u}}) - 1$, we have that it suffices that
$$
\frac{1}{n_s} < \lambda_c(\Delta(\Gamma_{\Omega_{n_u}})).
$$

For example, the n.n. $G$-SFT $\Omega$ represented in Figure \ref{fig:3cols} satisfies that $\Delta(\Gamma_{\Omega_{n_u}}) = 6$ and $\lambda_c(6) = \frac{5^5}{4^6} = \frac{3125}{4096}$; then, if $n_s > \frac{4096}{3125} = 1.31072$, we see that it suffices to have 2 copies of the safe symbol $0$ in order to have exponential SSM.

In general, since each vertex of the fundamental domain is connected to $n_u -1$ vertices in the clique and to at most $n_u$ vertices for each element $s$ in the generating set $S$, we see that each vertex in $\Gamma_{\Omega_{n_u}}$ is connected to at most $(n_u - 1) + \lvert S \rvert n_u$ other vertices. Then, we can estimate that
$$
\Delta(\Gamma_{\Omega_{n_u}}) \leq (\lvert S \rvert + 1)n_u - 1,
$$
so, in particular, if
$$
\frac{1}{n_s} < \lambda_c((\lvert S \rvert + 1)n_u - 1),
$$
exponential SSM holds (and therefore, again, uniqueness of measure of maximal entropy). This last equation and its relationship with the constraintedness of $\Omega$ has a similar flavor to the relationship between the percolation threshold $p_c(\mathbb{Z}^d)$ of the $\mathbb{Z}^d$ lattice and the concept of \emph{generosity} for $\mathbb{Z}^d$-SFTs introduced in \cite{1-haggstrom} by H\"aggstr\"om.

\begin{figure}[ht]
\centering
\includegraphics[scale = 2.4]{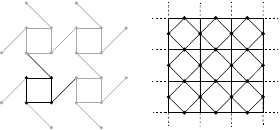}
\caption{On the left, an almost transitive graph $\Gamma$ with $\Gamma[U_0] \cong C_4$, the $4$-cycle. On the right, a portion of the line graph of $\Cay(\mathbb{Z}^2,\{\pm(1,0),\pm(0,1)\})$}
\label{fig:2graphs}
\end{figure}

\begin{remark}
It may be the case that a n.n. $G$-SFT $\Omega \subseteq \Sigma^G$ with a safe symbol could be represented by a graph $\Gamma$ which is better in terms of connectedness or maximum degree compared with the canonical representation $\Gamma_\Omega$, since we could encode $\Sigma$ using other fundamental domains, with a lower connectivity than the complete graph. For example, the n.n. $\mathbb{Z}^2$-SFT $\Omega_\Gamma$ corresponding to the graph $\Gamma$ on the left in Figure \ref{fig:2graphs} has $7$ symbols (the $7$ independent sets of the $4$-cycle), including a safe one. However, the canonical graph representation of $\Omega_\Gamma$, i.e., the graph $\Gamma_{\Omega_\Gamma}$, has a fundamental domain consisting of a clique with $6$ vertices, without considering extra connections. In particular, we see that both, $\Gamma$ and  $\Gamma_{\Omega_\Gamma}$ represent $\Omega$, but $\Delta(\Gamma) = 3 < 6 \leq \Delta(\Gamma_{\Omega_\Gamma})$. This motivates a finer notion of constraintedness, namely,
$$
\tilde{\mu}(\Omega) = \inf\{\mu(\Gamma): \Gamma \text{ represents } \Omega_{n_u}\},
$$
and the aforementioned results would still hold if we replace $\mu(\Omega)$ by $\tilde{\mu}(\Omega)$. Notice that a fundamental domain $U_0$ has at least $\lvert U_0\rvert+1$ independent sets (the empty one and all the singletons). In particular, this implies that $\tilde{\mu}(\Omega)$ is a minimum, since we only need to optimize over graphs $\Gamma$ with a fundamental domain $U_0$ such that $X(\Gamma[U_0]) = \lvert \Sigma_{n_u} \rvert$.
\end{remark}

\subsection{The monomer-dimer model and line graphs}

Given a graph $\Gamma = (V,E)$, we say that two different edges $e_1,e_2 \in E$ are \emph{incident} if they have one vertex in common. A {\bf matching} in $\Gamma$ is a subset $M$ of $E$ without incident edges. In a total parallel with the hardcore model case, we can represent a matching with an indicator function $m: E \to \{0,1\}$, denote the set of matchings of $\Gamma$ by $X^e(\Gamma)$, and define the associated partition function for some activity function $\lambda: E \to  \mathbb{R}_{> 0}$ as
$$
Z_\Gamma^{e}(\lambda) = \sum_{m \in X^e(\Gamma)} \prod_{e \in E}\lambda(e)^{m(e)}.
$$

The pair $(\Gamma,\lambda)$ is called the {\bf monomer-dimer model} and, as for the case of the hardcore model, we can define its associated free energy and Gibbs measures for a Gibbs specification adapted to this case.

An important feature of the monomer-dimer model is that, despite all its similarities with the hardcore model, it exhibits the SSM property for all values of $\lambda$ \cite{1-bayati} and, in particular, there is no phase transition \cite{1-heilmann}.

Considering this, most of the results presented in this paper, in particular the ones related to representation and approximation, can be adapted to counting matchings (see \cite{1-gamarnik} for a particular case), and there will not be a phase transition. One way to see this is through the \emph{line graph} $L(\Gamma)$ of the given graph $\Gamma$. Indeed, if we define $L(\Gamma)$ as the graph with set of vertices $E$ and set of edges containing all the adjacent edges in $E$, it is direct to see that there is a correspondence between matchings in $\Gamma$ and independent sets in $L(\Gamma)$, i.e.,
$$
Z_\Gamma^{e}(\lambda) = Z_{L(\Gamma)}(\lambda).
$$

In particular, this tell us that all the results in our paper that involve some restriction on $\lambda$, apply to every graph that can be obtained as a line graph of another one without restriction on $\lambda$. For example, the graph $\Gamma$ represented on the right in Figure \ref{fig:2graphs} corresponds to the line graph of the Cayley graph of $\mathbb{Z}^2$ with canonical generators, i.e.,
$$\Gamma = L(\Cay(\mathbb{Z}^2,\{\pm(1,0),\pm(0,1)\})).$$

Then, this observation implies that we can represent and approximate $f_{\mathbb{Z}^2}(\Gamma,\lambda)$ for every $\mathbb{Z}^2$-invariant activity function $\lambda$ on $\Gamma$.

\subsection{From almost free to free}
\label{subsec:9-7}

Suppose that $G \acts \Gamma_*$ is almost free, i.e., $\lvert \Stab_G(v) \rvert < \infty$ for all $v$. We proceed to show how to reduce the computation of $f_G(\Gamma,\lambda)$ to the computation of $f_G(\Gamma_*,\lambda_*)$ for some free action $G \acts \Gamma_*$, where $\Gamma_*$ and $\lambda_*$ are some suitable auxiliary graph and activity function, respectively. We consider this as another example of the versatility of independent sets for representing many phenomena.

Given a graph $\Gamma$ and an almost free action $G \acts \Gamma$, let $\Gamma_*$ be the new graph obtained by setting
$$
V(\Gamma_*) = \{(v,s): v \in V(\Gamma), s \in \Stab_G(v)\},
$$
and
$$
E(\Gamma_*) = \bigcup_{\substack{v \in V(\Gamma)\\s,s' \in \Stab_G(v)\\s \neq s'}} \{((v,s),(v,s'))\} \cup \bigcup_{\substack{(v,v') \in E(\Gamma)\\s \in \Stab_G(v)\\s' \in \Stab_G(v')}} \{((v,s),(v',s'))\}.
$$

In simple words, for each $v$ in $V(\Gamma)$ there are $\lvert \Stab_G(v) \rvert$ copies of $v$ in $V(\Gamma_*)$ such that
\begin{enumerate}
\item the $\lvert \Stab_G(v) \rvert$ copies of $v$ form a clique in $\Gamma_*$;
\item for $(v,v') \in E(\Gamma)$, each copy of $v$ is connected to all copies of $v'$ in $\Gamma_*$.
\end{enumerate}

Next, consider the activity function $\lambda_*: V(\Gamma_*) \to  \mathbb{R}_{>0}$ given by
$$\lambda_*((v,s)) = \lambda(v)/\lvert \Stab_G(v) \rvert.$$

Notice that for every $U \Subset V(\Gamma)$ we have that
$$
Z_\Gamma(U,\lambda) = Z_{\Gamma_*}(U_*,\lambda_*),
$$
where $U_* \Subset V(\Gamma_*)$ denotes the set of all copies of vertices in $U$. Indeed, notice that each independent set in $\Gamma_*[U_*]$ can be naturally identified with a unique independent set in $\Gamma[U]$: for $x' \in X(\Gamma_*)$, we can define $x \in X(\Gamma)$ as $x(v) = 1$ if and only if there exists $s \in \Stab_G(v)$ so that $x'((v,s)) = 1$. Conversely, if $\Gamma$ is finite, each independent set $x \in X(\Gamma)$ can be identified with $\prod_{v \in \Gamma} \lvert \Stab_G(v) \rvert^{x(v)}$ copies in $X(\Gamma_*)$. Therefore,
\begin{align*}
Z_{\Gamma_*}(U_*,\lambda_*)	&	=	\sum_{x' \in X(\Gamma_*[U_*])} \prod_{(v,s) \in U_*} \lambda_*((v,s))^{x'((v,s))}				\\
						&	=	\sum_{x \in X(\Gamma[U])} \prod_{v \in U} \lvert \Stab_G(v) \rvert^{x(v)} \prod_{v \in U} \left(\frac{\lambda(v)}{\lvert \Stab_G(v) \rvert}\right)^{x(v)}	\\
						&	=	\sum_{x \in X(\Gamma[U])} \prod_{v \in U} \lambda(v)^{x(v)}						\\
						&	=	Z_{\Gamma}(U,\lambda).
\end{align*}

Now, pick a fundamental domain $U_0 \subseteq V$ of $G \acts \Gamma$ and, for each $v \in U_0$, consider the set of left cosets $\{g\Stab_G(v): g \in G\}$ with a fixed set of representatives $R(v)$, i.e., $G = \bigsqcup_{g \in R(v)} g\Stab_G(v)$ for each $v \in U_0$. Without loss of generality, $1_G \in R(v)$ for every $v$. Next, given $v \in U_0$ and $h \in G$, define $t_{v,h}$ as the unique element in $\Stab_G(v)$ such that $ht_{v,h}^{-1} \in R(v)$. In addition, given another element $r \in G$, define $\psi^v_{h \to r} := ht_{v,h}^{-1}t_{v,r}h^{-1}$. Notice that if $t \in \Stab_G(v)$, then $t_{v,ht} = t_{v,h} t$ and $\psi^v_{ht \to rt} = \psi^v_{h \to r}$. Next, consider the action of $G$ on $\Gamma_*$ that, given a vertex $(u,s) \in V(\Gamma_*)$ and $h \in G$, it is defined as
$$
h \cdot (u,s) = (hu, hs\psi^v_{g \to hg}h^{-1}),
$$
where $u \in Gv$ for $v \in U_0$ and $g$ is some (or any) element in $G$ such that $u = gv$. First, let's check that this is indeed an action. If $v \in U_0$, then
\begin{align*}
r \cdot (h \cdot (v,1_G))	&	=	r \cdot (hv, h\psi^v_{1_G \to h}h^{-1})	\\
				&	=	r \cdot (hv, ht_{v,h}h^{-1})	\\
				&	=	(rhv, rht_{v,h}h^{-1}\psi^v_{h \to rh}r^{-1})	\\
				&	=	(rhv, rht_{v,h}h^{-1}(ht_{v,h}^{-1}t_{v,rh}h^{-1})r^{-1})	\\
				&	=	(rhv, rht_{v,rh}(rh)^{-1})	\\
				&	=	(rh) \cdot (v,1_G).
\end{align*}

Therefore, if $(u,s)$ is arbitrary, with $u = gv$ and $s = gtg^{-1}$ for $g \in R(v)$ and $t \in \Stab_G(v)$, we have that
\begin{align*}
r \cdot (h \cdot (u,s))	&	=	r \cdot (h \cdot ((gt) \cdot (v,1_G))	\\
				&	=	r \cdot ((hgt) \cdot (v,1_G))	\\
				&	=	(rhgt) \cdot (v,1_G)	\\
				&	=	(rh) \cdot ((gt) \cdot (v,1_G))	\\
				&	=	(rh) \cdot (u,s).
\end{align*}

Now, let's check that the action is free. If $h \cdot (v,1_G) = (v,1_G)$ for $v \in U_0$, then $(hv, ht_{v,h}h^{-1}) = (v,1_G)$, so $h \in \Stab_G(v)$ and $t_{v,h} = 1_G$. Therefore, $h = 1_G$. In the general case, if $h \cdot (gv,s) = (gv,s)$ with $u = gv$ and $s = gtg^{-1}$ for $g \in R(v)$ and $t \in \Stab_G(v)$, we have that $(hgt) \cdot (v, 1_G) = h \cdot ((gt) \cdot (v,1_G)) = (gt) \cdot (v,1_G)$, so $(t^{-1}g^{-1}hgt) \cdot (v,1_G) = (v,1_G)$ and, by the previous step, $t^{-1}g^{-1}hgt = 1_G$ or, equivalently, $h = 1_G$.

Now, if $G \acts \Gamma$ is almost transitive, then $G \acts \Gamma_*$ is almost transitive, too, with $\lvert \Gamma_*/G \rvert = \lvert \Gamma/G \rvert$. Indeed, if $U_0$ is a fundamental domain for $G \acts \Gamma$, we have that $U_0 \times \{1_G\}$ is a fundamental domain for $G \acts \Gamma_*$, since for every $(u,s) \in V(\Gamma_*)$ there exists a unique $v \in U_0$, $g \in G$, and $t \in \Stab_G(v)$ such that $u = gv$ and $s = gtg^{-1}$, so $(gt) \cdot (v,1_G) = (u,s)$.

Finally, if $K = \bigcup_{v \in U_0} \Stab_G(v)$, observe that $U_0 \times \{1_G\} \subseteq (U_0)_* \subseteq K (U_0 \times \{1_G\})$, so
$$
Z_{\Gamma_*}(F_n (U_0 \times \{1_G\}),\lambda_*) \leq Z_{\Gamma_*}(F_n (U_0)_*,\lambda_*) \leq Z_{\Gamma_*}(F_n K (U_0 \times \{1_G\}),\lambda_*)
$$
and, since $Z_\Gamma(F_nU_0,\lambda) = Z_{\Gamma_*}((F_nU_0)_*,\lambda_*) = Z_{\Gamma_*}(F_n(U_0)_*,\lambda_*)$, applying logarithms, dividing by $\lvert F_nU_0 \rvert$, and taking limit in $n$, we obtain that
$$
f_{G}(\Gamma,U_0,\lambda) = f_{G}(\Gamma_*,U_0 \times \{1_G\},\lambda_*),
$$
where we have used that $\{F_nK\}_n$ is a F{\o}lner sequence, $\lim_n \frac{\lvert F_nK \rvert}{\lvert F_n \rvert} = 1$, and $\lvert F_n(U_0 \times \{1_G\}) \rvert = \lvert F_nU_0 \rvert = \lvert F_n \rvert\lvert U_0 \rvert$.

\subsection{Spectral radius of matrices and occupation probabilities on trees}

A curious consequence of the hardcore model representation of a n.n. $G$-SFT with a safe symbol is that when $G = \mathbb{Z}$ and $S = \{1\}$, then $h_\mathbb{Z}(\Omega)$ has a well known characterization in terms of the transition matrix $M = M_1$ \cite{1-lind}. 

If $M$ is irreducible and aperiodic, there is always a unique stationary Markov chain $\Prob_M$ associated to $M$ such that $\log \lambda_M = h_\mathbb{Z}(\Omega_M)$, where $\lambda_M$ denotes the Perron eigenvalue of $M$ and we consider the natural invariant order in $\mathbb{Z}$. 

Now, if $M$ is a matrix such that the $i$th row and the $i$th column have no zeros, then $M$ is irreducible and aperiodic. In fact, the $i$th symbol, let's call it $a$, is a safe symbol. In this case, we have that
$$
\log \lambda_M = h_\mathbb{Z}(\Omega_M) = -\log \Prob_M([a^{0}] \vert [a^{-\mathbb{N}}]) = -\log \Prob_M([a^{0}] \vert [a^{-1}]),
$$

Therefore, $\lambda_M = \frac{1}{\Prob_M(a^{0} \vert a^{-1})}$ and to compute the spectral radius of $M$ reduces to compute $\Prob_M(a^{0} \vert a^{-1})$. For example, consider the following matrix 
$$
M = 
\left(\begin{array}{cccc}
1	&	1	&	1	&	1	\\
1	&	1	&	0	&	1	\\
1	&	0	&	1	&	0	\\
1	&	1	&	1	&	0	\\
\end{array}\right),
$$
where $a$ is the symbol associated to the first row. Given this matrix $M$, we can always construct a graph representation $\Gamma_{\Omega_M}$ of $\Omega_M$ as in Figure \ref{fig:graph1d}.

\begin{figure}[ht]
\centering
\includegraphics[scale = 2.4]{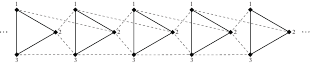}
\caption{A graph representation of the $0$-$1$ matrix $M$}
\label{fig:graph1d}
\end{figure}

Now, it is known that $\lvert \Prob_M(0^{0} \vert 0^{-1}) - \Prob_M(0^{0} \vert 0^{\{n,-1\}}) \rvert$ goes to zero exponentially fast as $n$ goes to infinity and, since every Markov chain is a Markov random field, we have that
\begin{align*}
\Prob_M(a^{0} \vert a^{\{n,-1\}})	&	=	\Prob_{\Gamma_{\Omega_M}[\{0,\dots,n-1\}U_0],1}(0^{U_0})	\\
							&	=	\sum_{i=1}^{U_0}\Prob_{\Gamma_{\Omega_M}[\{0,\dots,n-1\}U_0],1}(0^{v_i} \vert 0^{\{v_1,\dots,v_{i-1}\}})	\\
							&	=	\sum_{i=1}^{U_0}\Prob_{\Gamma_{\Omega_M}[\{0,\dots,n-1\}U_0 \setminus \{v_1,\dots,v_{i-1}\}],1}(0^{v_i})	\\
							&	=	\sum_{i=1}^{U_0}\Prob_{T_\mathrm{SAW}\Gamma_{\Omega_M}[\{0,\dots,n-1\}U_0 \setminus \{v_1,\dots,v_{i-1}\}],1}(0^{v_i}).
\end{align*}

This gives us an arboreal representation and a method to compute the spectral radius of any matrix $M$ satisfying the conditions described above, that we believe could be of independent interest.

\section*{Acknowledgements}

I would like to thank Brian Marcus for his valuable suggestions after reading a first draft of this work, Tom Meyerovitch for a helpful discussion about invariant random orders, and the anonymous referee for a careful reading of the manuscript and many useful comments and suggestions. I acknowledge the support of ANID/FONDECYT de Iniciación en Investigación 11200892.

\bibliographystyle{abbrv}
\bibliography{biblio}

\end{document}